 \documentclass [11pt,reqno]{amsart}
\usepackage {amsmath,amssymb,verbatim,geometry,color,esint}
\numberwithin{equation}{section}
 \usepackage{xcolor}
\usepackage[all]{xy}
\usepackage[pdftex]{graphicx}
\usepackage{tikz}
\usetikzlibrary{math}

\newtheorem{theorem}{Theorem}[section]
\newtheorem{thm}[theorem]{Theorem}
\newtheorem{lemma}[theorem]{Lemma}
\newtheorem{lem}[theorem]{Lemma}
\newtheorem{prop}[theorem]{Proposition}
\newtheorem{coro}[theorem]{Corollary}

\theoremstyle{definition}

\newtheorem{defi}[theorem]{Definition}
\newtheorem{example}[theorem]{Example}
\newtheorem{exa}[theorem]{Example}

 \newtheorem*{ackn}{Acknowledgements}
 
 \newtheorem*{thmA}{Theorem A} 
 \newtheorem*{thmB}{Theorem B} 
\newtheorem*{thmC}{Theorem C}

 \theoremstyle{plain}
\newtheorem*{namedthm}{\namedthmname}
\newcounter{namedthm}

\makeatletter

 \newcommand{\R}{\mathbb R}
 \newcommand{\Q}{\mathbb Q}
 \newcommand{\C}{\mathbb C}
  \newcommand{\PP}{\mathbb P}
 \newcommand{\N}{\mathbb N}
 \newcommand{\Z}{\mathbb Z}

 \newcommand{\reg}{\mathrm{\rm reg}}
 \newcommand{\sing}{\operatorname{\rm sing}}

 \newcommand{\e}{\varepsilon}

 \newcommand{\f}{\varphi}
 
 \newcommand{\p}{\psi}

 \newcommand \PSH {{\rm PSH}}

 \newcommand \vol{{\rm Vol}}
 
\newcommand{\Ric}{{\rm Ric}}

   \newcommand \ord {{\rm ord}}

\newcommand{\fa}{\mathfrak{a}}

\newcommand{\cO}{\mathcal{O}}

\newcommand{\ie}{{\rm i.e.\ }}

\DeclareMathOperator{\Num}{N^1}
\newcommand{\hvol}{\widehat{\mathrm{vol}}}
\newcommand{\ld}{\mathrm{A}}

\DeclareMathOperator{\DivVal}{DivVal}
\DeclareMathOperator{\Divb}{Div_{\mathrm{b}}}

\DeclareMathOperator{\Div}{Div}

\DeclareMathOperator{\lct}{lct}

\newcommand{\inter}{\cdot\ldots\cdot}

 \usepackage{hyperref}
\hypersetup{
    unicode=false,        
    pdftoolbar=true,      
    pdfmenubar=true,       
    pdffitwindow=false,     
    pdfstartview={FitH},    
    pdftitle={Uniform estimates for CMAE},    
    pdfauthor={Guedj, Lu},     
    colorlinks=true,       
   linkcolor=blue,          
    citecolor=blue,        
    filecolor=black,      
    urlcolor=blue}



\begin{document}


\title[K\"ahler-Einstein metrics near a log-terminal singularity]{K\"ahler-Einstein metrics with positive curvature near an isolated log terminal singularity}

\author{Vincent Guedj \& Antonio Trusiani, \\
{\tiny Appendix by S\'ebastien Boucksom}}


\address{Institut de Math\'ematiques de Toulouse  \& Institut Universitaire de France  \\ Universit\'e de Toulouse \\
118 route de Narbonne \\
31400 Toulouse, France\\}

\email{\href{mailto:vincent.guedj@math.univ-toulouse.fr}{vincent.guedj@math.univ-toulouse.fr}}
\urladdr{\href{https://www.math.univ-toulouse.fr/~guedj}{https://www.math.univ-toulouse.fr/~guedj/}}

\address{Institut de Math\'ematiques de Toulouse   \\ Universit\'e de Toulouse \\
118 route de Narbonne \\
31400 Toulouse, France\\}

\email{\href{mailto:antonio.trusiani@math.univ-toulouse.fr}{antonio.trusiani@math.univ-toulouse.fr}}
\urladdr{\href{https://sites.google.com/view/antonio-trusiani/home}{https://sites.google.com/view/antonio-trusiani/home}}
\date{\today}

 \address{Institut de Math\'ematiques  de Jussieu - Paris Rive Gauche  \\
 Sorbonne Université - Campus Pierre et Marie Curie \\
4 place Jussieu  \\
75252 Paris Cedex 05, France \\}

\email{\href{mailto:sebastien.boucksom@imj-prg.fr }{sebastien.boucksom@imj-prg.fr }}
\urladdr{\href{http://sebastien.boucksom.perso.math.cnrs.fr}{http://sebastien.boucksom.perso.math.cnrs.fr}}

 \begin{abstract}
 We  analyze the existence of  K\"ahler-Einstein   metrics of positive curvature in the neighborhood of
  a germ of a  log terminal singularity $(X,p)$.
 This boils down to solve a Dirichlet problem for certain complex Monge-Amp\`ere equations.
 We  show that the solvability of the latter is independent of the shape of the domain
 and of the boundary data.
   We establish a Moser-Trudinger $(MT)_{\gamma}$ inequality in subcritical regimes $\gamma<\gamma_p$
   and establish the existence of smooth solutions in that cases.
   We show that the expected critical exponent $\hat{\gamma}_p=\frac{n+1}{n} \widehat{\mathrm{vol}}(X,p)^{1/n}$ can be expressed
   in terms of the normalized volume, an important algebraic invariant of the singularity.
 \end{abstract}

 \maketitle

\tableofcontents

\section*{Introduction}

 Let $(X,p)$ be a germ of an isolated singularity. 
 We  analyze the existence of local  K\"ahler-Einstein 
 metrics of positive curvature in a neighborhood of $p$. 
 It follows from \cite[Proposition 3.8]{BBEGZ} that the singularity has
 to be   {\it log terminal},  a relatively mild type of singularity
 that plays a central role in birational geometry.
 We refer the reader to Definition \ref{def:lt} for a precise formulation and simply indicate here
 that  a prototypical example is the vertex of the affine cone over a Fano manifold. Consider indeed
 $$
 X=\{ z \in \C^{n+1}, \, P(z)=0 \},
 $$
 $P$ a homogeneous polynomial of degree $d \in \N^*$ so that 
 $H=\{ [z] \in \C\PP^n, \, P(z)=0 \}$ is a smooth hypersurface of the complex projective space.
 Then $(X,0)$ is log terminal if and only if $H$ is Fano (which is equivalent 
 here to $d <n+1$).
 Thus log terminal singularities can be seen as a local analogue of Fano varieties.
 
 \smallskip
 
Given a local embedding $(X,p)\hookrightarrow (\C^{N},0)$, 
constructing such a local  K\"ahler-Einstein  metric
boils down to solve a complex Monge-Amp\`ere equation
\begin{equation*}
     (MA)_{\gamma,\phi,\Omega}
    \begin{cases}
    (dd^{c}\varphi)^{n}=\frac{e^{-\gamma\varphi}d\mu_{p}}{\int_{\Omega}e^{-\gamma \varphi} d\mu_{p}}\\
    \varphi_{\vert\partial\Omega}=\phi,
    \end{cases}
\end{equation*} 
where $\Omega$ is a 
smooth neighborhood of $p$, 
$\phi$ is a smooth boundary data,  $\mu_{p}$ is an adapted volume form (see Definition \ref{def:adapted}), and 
$\gamma>0$ is a parameter.
We seek for a solution $\varphi \in   {\mathcal C}^{\infty}\big(\Omega\setminus\{p\}\big) \cap {\mathcal C}^0(\overline{\Omega})$
which is strictly plurisubharmonic in $\Omega \setminus \{ p \}$,
so that $\omega_{KE}:=dd^c \f$ is a K\"ahler form in $\Omega \setminus \{p\}$  
satisfying the Einstein equation
$$
\Ric(\omega_{KE})=\gamma \omega_{KE}.
$$


\medskip

An important motivation comes from the global study of positively curved 
K\"ahler-Einstein  metrics $\omega_{KE}$ on $\Q$-Fano varieties. 
Such canonical singular metrics have been constructed in \cite{BBEGZ} and further studied in \cite{BBJ21,LTW21,Li22}, extending
the resolution of the Yau-Tian-Donaldson conjecture \cite{CDS15} to this   singular context.
Despite recent important progress \cite{HS17,Druel18,HP19,BGL20}, the geometry of these singular metrics remains mysterious
and one needs to better understand the asymptotic behaviour of $\omega_{KE}$ near the singularities.

We restrict the metric $\omega_{KE}$ to a neighborhood of $p$ and wish to analyze
the behavior of its local potentials $\omega_{KE}=dd^c \f_{KE}$ near $p$.
The latter solve a Monge-Amp\`ere equation $(MA)_{\gamma,\phi,\Omega}$,
as can be seen by locally trivializing a representative of  the first Chern class
(after an appropriate rescaling). 
The boundary data is thus given by the solution $\f_{KE}=\phi$ itself.

\smallskip

Studying the corresponding family of equations $(MA)_{\gamma,\phi,\Omega}$
we will show that:
\begin{itemize}
    \item the possibility of solving $(MA)_{\gamma,\phi,\Omega}$ is independent of $\Omega$ and $\phi$;
    \item the largest exponent ${\gamma}_{crit}(X,p)$ for which we can solve $(MA)_{\gamma,\phi,\Omega}$ 
    only depends on the algebraic nature of the log-terminal singularity;
\end{itemize}

Following earlier works dealing with the case of compact K\"ahler varieties 
or the local smooth setting \cite{BBGZ13,GKY13,BB22,BBEGZ},
 we develop a variational approach to solve these equations.
 A crucial role is   played by
 $$
 E_{\phi}(\f)= \frac{1}{n+1}\sum_{j=0}^{n}\int_{\Omega}(\varphi-\phi_{0})(dd^{c}\varphi)^{j}\wedge(dd^{c}\phi_{0})^{n-j},
 $$
the Monge-Amp\`ere energy of $\f$ relative to a plurisubharmonic extension $\phi_0$ of $\phi$.
 This energy is a primitive of  the Monge-Amp\`ere operator and a building block of the functional
$F_{\gamma}$ whose Euler-Lagrange equation is $(MA)_{\gamma,\phi,\Omega}$,
$$
\f \in {\mathcal T}_{\phi}(\Omega) \mapsto 
F_{\gamma}(\f)= E_{\phi}(\f)+\frac{1}{\gamma} \log \int_{\Omega} e^{-\gamma \f} d\mu_p \in \R,
$$
Here ${\mathcal T}_{\phi}(\Omega)$ denotes the set of all
 plurisubharmonic functions $\f$ in $\Omega$ 
which are continuous on $\overline{\Omega}$ and such that $\f_{|\partial \Omega}=\phi$.

\smallskip

In order to solve $(MA)_{\gamma,\phi,\Omega}$ one can try and 
extremize $F_{\gamma}$ by
showing that it is a proper functional.
Our first main result in this direction (Theorem \ref{thm:MT})
is the following Moser-Trudinger type inequality.

\begin{thmA}
For any $0<\gamma<\frac{n+1}{n} \alpha (X,\mu_p)$, there exists $C_{\gamma}>0$ such that   
\begin{equation} \label{eq:MTintro} 
\tag{{$MT_{\gamma}$}}
\left( \int_{\Omega} e^{-\gamma \f} d\mu_p \right)^{\frac{1}{\gamma}} \leq C_{\gamma} \exp \left( -E_{\phi}(\f) \right).
\end{equation}
for all $\f \in {\mathcal T}_{\phi}(\Omega)$.
\end{thmA}

The {\it alpha invariant} of the singularity $(X,p)$ is defined by 
$$
\alpha (X,\mu_p):=\sup \left\{ \alpha>0, \,  
\sup_{\f \in {\mathcal F}_1(\Omega)} \int_{\Omega} e^{-\alpha \f} d\mu_p <+\infty 
\right\}.
$$
where  ${\mathcal F}_1(\Omega)$ denotes the set of plurisubharmonic functions $\f$ with $\phi$-boundary values,
whose Monge-Amp\`ere mass is bounded by  $\int_{\Omega} (dd^c \f)^n \leq 1$.

When $(X,p)$ is smooth, Theorem A has been obtained independently in \cite{BB22,GKY13}
with $\alpha(X,\mu_p)=n$ (the normalizations and methods are quite different in these two
works, but they eventually produce the same critical exponent).

\smallskip

We introduce
$$
\gamma_{crit}(X,p):=\sup \{ \gamma>0
\text{ such that } \eqref{eq:MTintro} \text{ holds} \}.
$$
We show in Section \ref{sec:invariance} that 
$\gamma_{crit}(X,p)$  depends neither on the size of 
$\Omega$, nor on  the boundary values $\phi$.
While Theorem A provides a lower bound for
$\gamma_{crit}(X,p)$, we provide an upper bound in
Theorem \ref{thm:UB}, which yields
$$
\frac{n+1}{n} \alpha (X,\mu_p) \leq \gamma_{crit}(X,p) \leq \frac{n+1}{n}\widehat{\rm vol}(X,p)^{1/n},
$$
where $\widehat{\rm vol}(X,p)$ denotes the {\it normalized volume} of the singularity $(X,p)$.
This is an algebraic invariant of the singularity at $p$ introduced by Chi Li in \cite{Li18}, 
which has recently played a  key role in the algebraic understanding of the moduli space of K-stable Fano varieties
(see \cite{Blum18,Liu18,LWX21,LWZ22} and the references therein);
we refer  to Definition \ref{def:normvol} for a precise definition.
  
  \smallskip
  
  When $p$ is smooth then $\alpha(X,\mu_p)= \widehat{\rm vol}(X,p)^{1/n}=n$ by
\cite{ACKPZ09,Dem09}.
It is tempting to conjecture that the equality $\alpha(X,\mu_p)= \widehat{\rm vol}(X,p)^{1/n}$
always holds. We establish  in Section \ref{sec:minoralpha}
 the following partial bounds on $\alpha(X,\mu_p)$.

\begin{thmB}
 The following inequalities hold 
 $$
  \frac{n}{{\rm mult}(X,p)^{1-1/n}} \frac{{\rm lct}(X,p)}{1+{\rm lct}(X,p)}  
 \leq \alpha(X,\mu_p) \leq \widehat{\rm vol}(X,p)^{1/n}.
 $$
 
  Moreover $\alpha(X,\mu_p)=\widehat{\rm vol}(X,p)^{1/n} $ if $(X,p)$ is an admissible singularity.
\end{thmB}

Here ${\rm mult}(X,p)$ denotes the algebraic multiplicity of $(X,p)$,
while ${\rm lct}(X,p)$ is its log canonical threshold (see Definition \ref{def:lct}).
Bounded  $\alpha(X,\mu_p)$ from below is quite involved; we show that
$\alpha(X,\mu_p)=\widehat{\rm vol}(X,p)^{1/n}$ when $n=2$,
but our lower-bound is not sharp when $n \geq 3$ unless $(X,p)$ is an {\it admissible singularity},
a notion introduced in \cite{LTW21}.
The vertex of the affine cone over a smooth Fano manifold is an example of admissible singularity
(see Section \ref{sec:minoralpha}).

Using analytic Green functions and Demailly's comparison theorem, we provide in Propositions \ref{pro:alphavsvol1}, \ref{pro:Skoda} evidence
for the equality $\alpha(X,\mu_p) = \widehat{\rm vol}(X,p)^{1/n}$.
The Appendix, written by S.Boucksom, uses an algebraic approach based on \cite{BdFF}, to
establish a stronger result than Proposition \ref{pro:Skoda}.

\smallskip

We note in Lemma \ref{lem:coercive} that if \eqref{eq:MTintro} holds, then
$F_{\gamma}$ is {\it coercive} (a strong quantitative version of properness).
When $\gamma <{\gamma}_{crit}(X,p)$,
we then further show the existence of  smooth solutions to $(MA)_{\gamma,\phi,\Omega}$.

\begin{thmC}
If $\gamma <\gamma_{crit}(X,p)$ then
there exists a plurisubharmonic function
$\f \in {\mathcal C}^{\infty}({\Omega} \setminus \{p\})$
which is continuous in $\overline{\Omega}$ with 
$\f_{|\partial \Omega} =\phi$, and such that
$$
(dd^{c}\varphi)^{n}=\frac{e^{-\gamma\varphi}d\mu_{p}}{\int_{\Omega}e^{-\gamma \varphi} d\mu_{p}}
\; \; \text{ in } \; \; \Omega.
$$
\end{thmC}

We expect the solution to be unique, at least when $\Omega$
is a generic Stein neighborhood of $p$. We refer the reader
to \cite{GKY13,BB22} for partial results in this direction when 
$p$ is a smooth point.


   \begin{ackn} 
The authors are   supported by the research project Hermetic (ANR-11-LABX-0040),
the ANR projects Paraplui. The second author is supported by a grant from the
Knut and Alice Wallenberg foundation. We thank S\'ebastien Boucksom for writing the Appendix, 
which enhances Proposition \ref{pro:Skoda}.
\end{ackn}

    \section{Preliminaries}
    
        \subsection{Analysis on singular spaces}

     Let $X$ be a reduced complex analytic space  of pure dimension $n \ge 1$. 
 We let $X_{\reg}$ denote the complex manifold of regular points of $X$ and
 $
 X_{\sing} := X \setminus X_{\reg}
 $
 be the set of singular points; this is an analytic subset of $X$  of complex codimension $\ge 1$. 
  We always assume in this article that:
 \begin{itemize}
 \item $X_{\sing}=\{p\}$ consists of a single isolated point;
 \item $X_{\reg}$ is locally irreducible at $p$;
 \item $U$ is a fixed neighborhood of $p$ and  $j: U \hookrightarrow \mathbb C^N$ is a local embedding onto an analytic subset 
 of  $\mathbb C^N$ for some $N \ge 1$.
 \end{itemize}
 As we are interested in the asymptotic behavior of K\"ahler-Einstein potentials near the singular point $p$, we shall identify
 $X$ with $j(U)$ in the sequel.

  \subsubsection{Plurisubharmonic functions}

 Using the local embedding $j$, it is possible to define the spaces of smooth forms  on $X$
 as restriction of smooth forms of $\C^N$.
  The notion of currents on $X$ is   defined by duality; the operators $\partial $ and $\bar{\partial }$, $d$, $d^c$ and $dd^c$ are 
 also well defined by duality (see \cite{Dem85} for more details). 

Here $d=\partial+\overline{\partial}$ and $d^c=\frac{1}{4i\pi}(\partial-\overline{\partial})$ are real operators 
and $dd^c =\frac{i}{2\pi}\partial\overline{\partial}$. With this normalization the function
$z \in \C^n \mapsto \rho_{FS}(z)=\log[1+|z|^2] \in \R$ 
is smooth and plurisubharmonic in $\C^n$, with
$$
\int_{\C^n} (dd^c \rho_{FS})^n=1.
$$

 \begin{defi} 
  We say that a function  $u : X \longrightarrow \mathbb R \cup \{-\infty\}$ 
    is  plurisubharmonic (psh for short) on $X$ if it is  the restriction of a plurisubharmonic function 
 of $\mathbb C^N$. 
 
  We let $PSH(X)$ denote the set of all  plurisubharmonic functions on $X$ that are not identically $-\infty$.
  \end{defi}
 
Recall  that  $u$ is  called weakly plurisubharmonic on $X$ if it is locally bounded from above on $X$ and  its restriction to   $X_{\reg}$ is plurisubharmonic.
 One can extend it to $X$ by $u^* (p) := \limsup_{X_{\reg} \ni y \to p } \,  u (y)$.
 Since $X$ is locally irreducible, it follows from the work of Forn\ae ss-Narasimhan \cite{FN80}  that
 $u$ is weakly plurisubharmonic if and only if $u^*$ is plurisubharmonic
 (see \cite[Corollary 1.11]{Dem85}).

\smallskip

If $u \in PSH(X)$, then $u$ is upper semi-continuous on $X$ and locally integrable 
with respect to the volume form 
$$
dV_X:=\omega_{eucl}^n \wedge [X].
$$
Here $[X]$ denotes the current of integration along $X$ and $\omega_{eucl}:=\sum_{j=1}^N i dz_j \wedge d\overline{z_j}$
is the euclidean K\"ahler form.
In particular    $dd^c u$ is a well defined current of bidegree $(1,1)$ which is positive.

 \subsubsection{Pseudoconvex domains and boundary data}

 Following \cite{FN80}  we say that $X$ is Stein if it admits a ${\mathcal C}^2$-smooth strongly plurisubharmonic exhaustion.

\begin{defi} 
A domain $\Omega \Subset X$ is   strongly pseudoconvex if it admits a negative ${\mathcal C}^2$-smooth strongly plurisubharmonic exhaustion,
 i.e. a function $\rho$ strongly plurisubharmonic in a neighborhood $\Omega'$ of $\overline{\Omega}$ such that 
 $\Omega := \{ x \in \Omega' \, ; \, \rho (x) < 0\}$,
 $d\rho \neq 0$ on $\partial \Omega$,
  and for any $c < 0$, 
 $$
 \Omega_c := \{x \in \Omega'; \, \rho (x) < c\} \Subset \Omega
 $$
 is relatively compact.
  \end{defi}
 
  We are interested in solving a Dirichlet problem for some complex Monge-Amp\`ere equations
  in a bounded strongly pseudoconvex domain $\Omega=\{\rho<0\}$, 
  with given boundary data $\phi \in {\mathcal C}^{\infty}(\partial \Omega)$.

  \begin{defi}
  Given $\phi \in {\mathcal C}^{\infty}(\partial \Omega)$,
  we fix $\phi_0$ a plurisubharmonic function in $\Omega$
  which is ${\mathcal C}^{\infty}$-smooth near $\overline{\Omega}$
  and such that $\phi_0{|\partial \Omega}=\phi$.
  \end{defi}
  
  Such an extension can be obtained as follows:
  we pick $\tilde{\phi}$ an arbitrary ${\mathcal C}^2$-smooth extension  
  to $\overline{\Omega}$, and then consider 
  $\phi_0:=\tilde{\phi}+A \rho$,
  for $A$ so large that   $\phi_0$ is ${\mathcal C}^2$-smooth 
  and plurisubharmonic in $\overline{\Omega}$.
    All quantities   introduced in the remainder of the paper  
 are essentially independent of the  particular choice of the extension.

 \subsubsection{Monge-Amp\`ere operators} \label{sec:MA}

 The complex Monge-Amp\`ere operator $(dd^c \cdot)^n$ acts on
 a smooth psh functions  $\f$. When $X=\C^n$, it boils down to
 $$
 (dd^c \f)^n=c_n \det \left( \frac{\partial^2 \f}{\partial z_j \overline{\partial} z_k} \right) \omega_{eucl}^n,
 $$
 where $c_n>0$ is a normalizing constant.

 \subsubsection*{Bounded functions}
  Following \cite{BT82} this operator can be extended
 to the class $PSH(X) \cap L^{\infty}_{loc}$ by using  approximation
 by smooth psh functions: given $\f \in  PSH(X) \cap L^{\infty}_{loc}$,
 there exists a unique positive Radon measure $\mu_{\f}$ on $X$ such that for {\it any}  sequence $(\f_j)$
 of smooth psh functions decreasing  to $\f$, one has
 $$
 \mu_{\f}=\lim (dd^c \f_j)^n,
 $$
 where the limit holds in the weak sense.
 One then sets $(dd^c \f)^n:=\mu_{\f}$.

 \begin{defi}
 We set 
 $$
 \mathcal{T}^{\infty}_{\phi}(\Omega):=\Big\{\varphi\in SPSH(\Omega)
 \cap {\mathcal C}^{\infty}(\overline{\Omega})\, : 
 \,  \varphi_{|\partial \Omega}=\phi \Big\},
 $$
 where $SPSH(\Omega)$ is the set of strictly plurisubharmonic functions,
 and
 $$
 \mathcal{T}_{\phi}(\Omega):=\Big\{\varphi\in PSH(\Omega)\cap {\mathcal C}^0(\overline{\Omega})\, : 
 \,  \varphi_{|\partial \Omega}=\phi, \, \int_{\Omega}(dd^{c}\varphi)^{n}<+\infty\Big\},
 $$
 \end{defi}
 
 This   latter class has been introduced by Cegrell in \cite{Ceg98}; it can be used as a psh version of 
 test functions (in the sense of distributions), as well as a building block for finite energy classes of
 mildly unbounded functions.
 
 \begin{lem} \label{lem:smoothing}
 Any $\f \in \mathcal{T}_{\phi}(\Omega)$ is a quasi-decreasing limit of functions
 in $ \mathcal{T}^{\infty}_{\phi}(\Omega)$.
 \end{lem}
 
 \begin{proof}
Fix a local embedding $X \hookrightarrow \C^N$. A function 
$\f \in \mathcal{T}_{\phi}(\Omega)$ is the restriction of an ambient continuous psh function $\p$.
We use standard convolution in $\C^N$ to find a  sequence of smooth 
strictly psh functions 
$\p_j$ decreasing  to $\p$. Consider $\f_j={\p_j}_{|X}-\e_j$, where $0 <\e_j$ goes to zero so that
$\f_j<\phi_0$ near $\partial \Omega$ (the  functions ${\p_j}_{|X}$ uniformly converge to $\f$ by continuity).
Set  $\tilde{\f_j}:=\tilde{\max}(\f_j,A_j \rho+\phi_0)$, where 
$\tilde{\max}$ is a regularized maximum, then $\f_j \in  \mathcal{T}^{\infty}_{\phi}(\Omega)$
converges to $\f$ as $A_j \rightarrow +\infty$.
 \end{proof}

  \subsubsection*{Midly unbounded functions}
  
  The   complex Monge-Amp\`ere operator can be defined for mildly unbounded psh functions.
  We refer the reader to \cite{Ceg04,Blo06} for the case of smooth domains in $\C^n$;
  their analysis easily extends to our context.
  
\begin{defi}
We let $ \mathcal{F}(\Omega)$ denote the set of all functions $\varphi\in PSH(\Omega)$ which are decreasing limit 
of a sequence of functions $\varphi_{j}\in \mathcal{T}_{\phi}(\Omega)$ such that
$$
\sup_{j}\int_{\Omega}(dd^{c}\varphi_{j})^{n}<+\infty.
$$
\end{defi}

The operator $(dd^c \cdot )^n$ is well defined on  $ \mathcal{F}(\Omega)$, continuous along monotonic sequences,
and yields Radon measures $(dd^c \f)^n$ which have finite mass in $\Omega$.
We endow $ \mathcal{F}(\Omega)$ with the $L^1$-topology. 
Let us stress that the operator $\f \mapsto (dd^c \f)^n$ is {\it not} continuous
for the $L^1$-topology, but the class  $ \mathcal{F}(\Omega)$ enjoys the following  useful compactness property.

\begin{prop} \label{prop:Compact}
The set $ \mathcal{F}_1(\Omega)=\big\{\varphi\in\mathcal{F}(\Omega);  \, \int_{\Omega}(dd^{c}\varphi)^{n}\leq 1 \big\}$ is compact.
\end{prop}

This is shown in \cite[Observation A.3]{Zer09} for smooth domains,
and the same proof applies in our midly singular context.  
Let us stress that the Monge-Amp\`ere operator cannot be defined for all 
psh functions: there is e.g. no reasonable way to make sense of $(dd^c \log |z_1|)^n$.
A consequence of Proposition \ref{prop:Compact} is that one can not approximate
such a function by a decreasing sequence of psh functions with prescribed boundary values and
uniformly bounded Monge-Amp\`ere masses.

   \subsection{Adapted volume form}

     \subsubsection{Log terminal singularities}

    Let $Y$ be   a connected normal  complex variety 
    such that $K_Y$ is $\Q$-Cartier
    near $p \in Y$.    One can consider the $dd^c$-cohomology class of $-K_Y$, denoted by $c_1(Y)$. 
   
Given a log-resolution $\pi: \tilde{Y} \to Y$ of $(Y,p)$, 
 there exists a unique $\Q$-divisor $\sum_i a_i E_i$ whose push-forward to $Y$ is $0$ and with
$$
K_{\tilde{Y}}=\pi^*(K_Y)+\sum_i a_i E_i. 
$$

    \begin{defi} \label{def:lt}
     The coefficient $a_i\in\Q$ is  the \emph{discrepancy} of $Y$ along $E_j$.
    One says that $p$ is a
{\it  log terminal} singularity if $a_j>-1$ for all $j$. 
     \end{defi}

It is  classical   that this  condition is independent of the choice of resolution. 
   In the remainder of this article we assume that
   \begin{itemize}
   \item the singularity $(X,0)$ is log terminal.
   \item $Y=\Omega$ is a strongly  pseudoconvex neighborhood of $0=p \in X$;
   \item the canonical  bundle $K_{\Omega}$ is $\Q$-Cartier and  $rK_{\Omega}=0$ for some $r\in\N$.
   \end{itemize}

 \begin{defi}\cite[Definition 6.5]{EGZ09}
 \label{def:adapted}
Fix $\sigma$  a nowhere vanishing holomorphic section of $rK_{\Omega}$, 
and $h$ a smooth hermitian metric of $ K_{\Omega}$,
then
$$
\mu_{p}=\lambda \frac{\big(c_{n}\sigma\wedge \bar{\sigma}\big)^{1/r}}{|\sigma|^{2/r}_{h^r}}
$$
is an \emph{adapted measure}, where $\lambda >0$ is a positive normalizing constant.
    \end{defi}
    
    Observe  that $\mu_p$ is independent of the choice of $\sigma$, and  
    $$
    dd^c \log \mu_{p}=-\Theta_h(K_{\Omega})
    $$
    is the curvature of $h$, as follows from the Poincar\'e-Lelong formula.
    
The measure $\mu_p$ has finite mass 
by \cite[Lemma 6.4]{EGZ09}:
    let $\pi: \tilde{\Omega}\to \Omega$ be a resolution of $(\Omega,0)$, then
$$
\pi^{*}\mu_{p}=\lambda \prod_{j=1}^{M}\lvert s_{E_{j}} \rvert^{2a_{j}} dV_{\tilde{\Omega}},
$$
where $dV_{\tilde{\Omega}}$ is a smooth volume form on $\tilde{\Omega}$, $E_{1},\dots,E_{M}$ are exceptional divisors, the $s_{E_j}$'s are holomorphic sections such that 
$E_j=(s_{E_j}=0)$, and  
$$
rK_{\tilde{\Omega}}=\pi^{*}(rK_{\Omega})+r\sum_{j=1}^{M}a_{j}E_{j}=r\sum_{j=1}^{M}a_{j}E_{j}.
$$
Thus $\tilde{f}=\prod_{j=1}^{M}\lvert s_{E_{j}} \rvert^{2a_{j}}$ belongs to $L^s(dV_{\tilde{\Omega}})$
for some $s>1$,  as $p$ is log terminal.

  \begin{defi}
 We choose $\lambda=\lambda_{\Omega}$ so that 
 $\mu_p$ is a probability measure in $\Omega$.
 \end{defi}
   
The results to follow are independent of this (convenient) normalization.

    \subsubsection{Ricci curvature}
    
  Let  $\omega$ be a positive closed current of bidegree $(1,1)$ in $\Omega$ with bounded local potentials.
  Its top power $\omega^n$ is well-defined as explained in Section \ref{sec:MA}.
  If $\omega^n$ is absolutely continuous with respect to $dV_X$, then we set
  $$
  {\rm Ric}(\omega):=-dd^c \log \omega^n.
  $$

  \begin{defi}
  We say that $\omega$ is a K\"ahler-Einstein metric if it satisfies
  $$
  {\rm Ric}(\omega)=\gamma \omega
  $$
  for some $\gamma \in \R$.
  \end{defi}
  
  In this article we are mainly interested in the case when $\gamma >0$.  
   We choose the hermitian metric $h \equiv 1$ for $K_{\Omega}$, so that
   $\Theta_h=0$.  Since 
   $$
   {\rm Ric}(\omega)={\rm Ric}(\mu_p)-dd^c \log (\omega^n/\mu_p),
   $$
   the above K\"ahler-Einstein equation is equivalent, writing $\omega=dd^c \f$, to
   $$
   (dd^c \f)^n=e^{-\gamma \f} e^{w} \mu_p,
   $$
  where $w$ is a pluriharmonic function in $\Omega$.
Changing $\f$ in $\f-w/\gamma$ and then $\f$ in $t \f$
(observe that ${\rm Ric}(t \omega)={\rm Ric}(\omega)$ for any $t>0$),
we can normalize $\omega$ by
  $\int_{\Omega} \omega^n=1$
and reduce to
$$
(dd^c \f)^n= \frac{e^{-\gamma \f} \mu_p}{\int_{\Omega} e^{-\gamma \f} \mu_p}.
$$
Seeking for a K\"ahler-Einstein metric thus leads one to solve  $(MA)_{\gamma,\phi,\Omega}$. 
 
 Conversely solving    $(MA)_{\gamma,\phi,\Omega}$
 will produce a K\"ahler-Einstein metric $\omega=dd^c \f$, if we can establish
 enough regularity of the solution $\f$.

\subsubsection{Log canonical threshold}
  
We consider the density $f=\mu_p/dV_X$. It is related to the density
$\tilde{f}$ in a resolution by  
$$
 \pi^* \mu_p=f \circ \pi \cdot \pi^* dV_X=\tilde{f} dV_{\tilde{\Omega}}.
$$
  An analytic expression for $f$ is obtained as follows.
  Recall that $dV_X=\omega_{eucl}^n \wedge [X]$, 
where $\omega_{eucl}$ denotes the euclidean K\"ahler form on $\C^{N}$.  Set
$
dz_{I}=dz_{i_{1}}\wedge \cdots \wedge dz_{i_{n}},
$
where $1 \leq i_1< \cdots <i_n \leq N$.
There exists germs of holomorphic functions $f_{I}\in \mathcal{O}_{\Omega,0}$ such that
$
\big(dz_{I}\big)^{r}=f_{I}\sigma
$
since $\sigma$ is a local generator of $rK_{X}$. In particular the volume form $dV_{X}:=\omega_{eucl}^{n}\wedge [\Omega]$ is comparable to$\Big(\sum_{I}\lvert f_{I}\rvert^{\frac{2}{r}}\Big)\mu_{p},$ i.e.  
$$
\mu_{p}=f dV_{X},
\; \text{ with } \; 
f:=\Big(\sum_{I}\lvert f_{I}\rvert^{\frac{2}{r}}\Big)^{-1}.
$$

The germs of holomorphic functions $f_{I}$ generate an ideal $\mathcal{I}_{p}^{r}$, where $\mathcal{I}^{r}$ is an ideal sheaf associated to the singularities of $(X,p)$. In particular  
$$
\pi^{-1}\mathcal{I}^{r}\cdot \mathcal{O}_{\tilde{\Omega}}=\mathcal{O}_{\tilde{\Omega}}\big(-r\sum_{j=1}^{M}b_{j}E_{j}\big)
$$
for coefficients $b_{j}\in\N$ such that 
$
f\circ \pi\sim \prod_{j=1}^{M}\lvert s_{E_{j}} \rvert^{-2b_{j}}.
$

\begin{defi} \label{def:lct}
The \emph{log canonical threshold} of 
$(X,p)$  is given by
$$
\mathrm{lct}(X,p):=\inf_{j\in{1,\dots,M}} \frac{a_{j}+1}{b_{j}}.
$$
\end{defi}

We let the reader check that the definition is independent of the choice of resolution,
and that $\mathrm{lct}(X,\mathcal{I})\in (0,n].$
One can  equivalently use the following point of view:
  if $\mathcal{I}$ is a general ideal sheaf, 
\begin{equation}
    \label{eqn:lct}
    \mathrm{lct}(X,\mathcal{I}):=\inf_{E/X}\frac{A_{X}(E)}{\mathrm{ord}_{E}(\mathcal{I})}
\end{equation}
where 
$
A_{X}(E):=1+\mathrm{ord}_{E}(K_{Y/X})
$
 is the log-discrepancy of $E$, and 
the infimum is over all prime divisors $E$ on resolutions $Y$ of $X$.
When $\mathcal{I}$ is supported at $p$ we can   restrict in (\ref{eqn:lct}) to consider prime divisors centered at $p$.

\begin{exa}
The ordinary double point (ODP) 
$
X=\left\{ z \in \C^{n+1}, \; \sum_{j=0}^n z_j^2=0\right\}
$
is the simplest isolated log terminal singularity which is not a quotient singularity when $n \geq 3$
(when $n=2$, log-terminal singularities are precisely the singularities of the form
$X=\C^2/G$, $G \subset GL(2,\C)$ a finite subgroup).

In this case   $\mathcal{I}^2=(z_1^2,\dots,z_n^2)$. Indeed the $n$-forms
$$
\sigma_j:=\frac{\big(dz_0\wedge \cdots\wedge \widehat{dz_j}\wedge \cdots \wedge dz_n\big)^2}{z_j^2}=-\frac{\big(dz_0\wedge \cdots \wedge \widehat{dz_j}\wedge \cdots \wedge dz_n\big)^2}{\sum_{k\neq j} z_k^2},
$$
defined  on $U_j:=\{z_j\neq 0\}$, glue together to give a local generator $\sigma$ of $2 K_X$ 
(note that $\sum_{j=0}^n z_j dz_j=0$). 
In particular $|f_I|^{2/r}=|z_j|^2$ where $j=[0,n] \setminus I$, $r=2$  and
$$
\mu_{p} \sim \frac{1}{\sum_{j=0}^n \lvert z_j \rvert^2}dV_X 
$$

If $\pi: \rm Bl_0 \C^{n+1}\to \C^{n+1}$ denotes the blow-up at $0$, $E$ the exceptional divisor,
 and  $F$ the restriction of $E$ to $Y$, the strict transform of $X$, 
 we obtain 
$$
\pi^{-1}\mathcal{I}^2\cdot \mathcal{O}_Y=\mathcal{O}_Y(-2F)
\; \; \text{ and } \; \; 
\pi^* \mu_p= \lvert s_F \rvert^{2(n-2)} dV_Y
$$
for a smooth volume form $dV_Y$.
Thus ${\rm lct}(X,p)={\rm lct}(X,\mathcal{I})=n-1$.
\end{exa}

We will need the following result which connects ${\rm lct}(X,p)$
and the integrability properties of the density $f=\mu_p/dV_X$.

\begin{lem} \label{lem:integrabilityf}
The density $f=\mu_p/dV_X$ belongs to $L^r(dV_X)$ for  
$r<1+{\rm lct}(X,p)$.
\end{lem}

\begin{proof}
Let $\pi:\tilde{\Omega} \rightarrow \Omega$ be a resolution of the singularity. Recall that
$$
f \circ \pi \sim \prod_{j=1}^{M}\lvert s_{E_{j}} \rvert^{-2b_{j}}
\; \; \text{ and } \; \;
\tilde{f}=\prod_{j=1}^{M}\lvert s_{E_{j}} \rvert^{2a_{j}},
\; \; \text{ hence } \; \;
\pi^* dV_X \sim  \prod_{j=1}^{M}\lvert s_{E_{j}} \rvert^{2(a_{j}+b_j)} dV_{\tilde{\Omega}}.
$$
It follows that 
$
\int_{{\Omega}} f^r dV_X \sim \int_{\tilde{\Omega}} 
\prod_{j=1}^{M}\lvert s_{E_{j}} \rvert^{2(a_{j}+b_j)-2rb_j} dV_{\tilde{\Omega}}<+\infty
$
if and only if $r<\frac{1+a_j+b_j}{b_j}$ for all $j$, which  yields the statement since  
 ${\rm lct}(X,p)=\inf_j \frac{1+a_j}{b_j}$.
\end{proof}

\subsection{Normalized volume}
\label{ssec:Alg}

The \emph{(Hilbert-Samuel) multiplicity} of an ideal $\mathcal{I}$ supported at $p$ is defined as
$$
\mathrm{e}(X,\mathcal{I}):=\lim_{m\to +\infty}\frac{l(\mathcal{O}_{X,p}/\mathcal{I}^{m})}{m^{n}/n!}
$$
where $l$ denotes the length of an Artinian module.

 Given a divisor $E$ over $X$ centered at $p$, the volume of $E$ over $p\in X$ is  
$$
\mathrm{vol}_{X,p}(E):=\lim_{m\to +\infty}\frac{l(\mathcal{O}_{X,p}/\mathfrak{a}_{m}(E))}{m^{n}/n!}
$$
where $\mathfrak{a}_{m}(E):=\big\{f\in \mathcal{O}_{X,p}\, : \, \mathrm{ord}_{E}(f)\geq m\big\}$ (see \cite{ELS03}).

\begin{defi} \label{def:normvol}
\cite{Li18}
The \emph{normalized volume of $p\in X$} is  
$$
\widehat{\mathrm{vol}}(X,p):=\inf_{E/X}\widehat{\mathrm{vol}}_{X,p}(E)
$$
where the infimum runs over all prime divisors $E$ over $X$ centered at $p$, and
$$
\widehat{\mathrm{vol}}_{X,p}(E):=A_{X}(E)^{n}\cdot\mathrm{vol}_{X,p}(E).
$$
is the \emph{normalized volume of $E$ over $(x\in X)$}.
\end{defi}

We shall need the following important result.

\begin{thm}\cite[Theorem 27]{Liu18}
\label{thm:NormVol}
Let $(X,p)$ be a log terminal singularity of complex dimension $\dim_{\C} X=n$. Then  
$$
\widehat{\mathrm{vol}}(X,p)=\inf_{\mathcal{I}\, \mathrm{supported}\, \mathrm{at}\, p} \mathrm{lct}(X,\mathcal{I})^{n}\cdot \mathrm{e}(X,\mathcal{I}).
$$
\end{thm}

Observe that the quantity $\mathrm{lct}(X,\mathcal{I})^{n}\cdot \mathrm{e}(X,\mathcal{I})$ is invariant under rescaling $\mathcal{I}\to \mathcal{I}^{r}$, $r\in\mathbb{N}$.
One can actually only consider coherent ideal sheaves supported at $p$. 
Indeed   any ideal $\mathcal{I}$ supported at $p$ is associated to a closed subscheme $Z$ such that $\mathrm{Supp}\,Z=\{p\}$ 
\cite[Corollary II.5.10]{Hart77},  while any ideal associated to a closed subscheme is coherent 
\cite[Proposition II.5.9]{Hart77}.

   \begin{exa}
Consider  again 
$
X=\left\{ z \in \C^{n+1}, \; \sum_{j=0}^n z_j^2=0\right\}.
$
Recall  that $\mathcal{I}^2=(z_1^2,\dots,z_n^2)$ is the ideal sheaf associated to the adapted measure,
and that the ideal $\mathcal{I}^2$ corresponds to $2F$ where $F$ is the exceptional divisor
in the blow up at $p$. In particular $A_X(F)=n-1$.

We  observe here that $e(X,\mathcal{I}^2)=2^{n+1}$ 
and  ${\rm \widehat{vol}}_{X,p}(F)=2(n-1)^n$ since
$$
l\big(\mathcal{O}_{X,p}/\mathcal{I}^{2m}\big)=l\big(\mathcal{O}_{X,p}/\mathfrak{a}_{2m}(F)\big)=2^{n+1}\frac{m^n}{n!}+O(m^{n-1}).
$$
 In \cite[Example 5.3]{Li18} it is further shown that $F$ is a minimizer 
for the normalized volume of $p\in X$, i.e. that ${\rm \widehat{vol}}(X,p)=2(n-1)^n$.
\end{exa}

  \section{A variational approach}
  
  A variational approach for solving degenerate complex Monge-Amp\`ere equations
  has been developed in \cite{BBGZ13} in the context of compact K\"ahler manifolds.
  It notably applies to the construction of singular K\"ahler-Einstein metrics
  of non-positive curvature.
  This has been partially adapted to smooth pseudoconvex domains of $\C^n$ in \cite{ACC12}.
  
  The case of positive curvature is notoriously more difficult, as 
  illustrated by the resolution of the Yau-Tian-Donaldson conjecture by Chen-Donaldson-Sun \cite{CDS15}.
  It has been treated extensively in \cite{BBEGZ}, and eventually lead  to an alternative 
  solution of the Yau-Tian-Donaldson conjecture for Fano varieties \cite{BBJ21,LTW21,Li22}.
  Adapting \cite{BBEGZ} to our local singular context, we develop in this section
  a variational approach for solving the equation
  \begin{equation}
    \label{eqn:MA}
     (MA)_{\gamma,\phi,\Omega}
    \begin{cases}
    (dd^{c}\varphi)^{n}=\frac{e^{-\gamma\varphi}d\mu_{p}}{\int_{\Omega}e^{-\gamma \varphi} d\mu_{p}}\\
    \varphi_{\vert\partial\Omega}=\phi.
    \end{cases}
\end{equation}

\subsection{Monge-Amp\`ere energy}

\subsubsection{Smooth tests}
Fix $\Omega=\{\rho<0\}$ and $\phi$ as described previously, and  
$$
\mathcal{T}_{\phi}^{\infty}(\Omega)=\big\{\varphi\in SPSH(\Omega)\cap \mathcal{C}^{\infty}(\overline{\Omega})\, : 
\, \varphi_{\vert \partial\Omega}\equiv \phi\big\}.
$$

Recall that 
$\phi_{0}\in \mathcal{C}^{\infty}(\bar{\Omega})\cap PSH(\Omega)$ denotes a smooth psh extension
of $\phi$ to $\bar{\Omega}$.
We set   $\omega:=dd^{c}\phi_{0}$. This is a semi-positive form,
which can be assumed to be K\"ahler. However if $\phi \equiv 0$,
we can equally well take $\phi_0 \equiv 0$ and get $\omega \equiv 0$.

\begin{defi}
We call $E_{\phi}(\varphi):=\frac{1}{n+1}\sum_{j=0}^{n}\int_{\Omega}(\varphi-\phi_{0})(dd^{c}\varphi)^{j}\wedge(dd^{c}\phi_{0})^{n-j}$ the $\phi$-relative Monge-Amp\`ere energy of 
$\f \in \mathcal{T}_{\phi}^{\infty}(\Omega)$.
\end{defi}

While the formula depends on the choice of   $\phi_0$, it follows from Lemma 
\ref{lem:smoothIPP} that the difference of two such relative 
energies is constant, 
$$
E_{\phi_1}(\varphi)-E_{\phi_0}(\varphi)=E_{{\phi_1}}(\phi_0).
$$ 
For $\phi_{0}=0$, the formula  reduces to $E(\varphi):=E_{0}(\varphi)=\frac{1}{n+1}\int_{\Omega}\varphi(dd^{c}\varphi)^{n}$.

This definition is motivated by the fact the $E_{\phi}$ is a primitive of the Monge-Amp\`ere operator
for smooth psh functions with $\phi$-boundary values.

\begin{lem} \label{lem:smoothIPP}
Fix $\f \in \mathcal{T}_{\phi}^{\infty}(\Omega), v \in {\mathcal D}({\Omega})$.
Then $\f+tv \in \mathcal{T}_{\phi}^{\infty}(\Omega)$ for $t$ small, and
$$
{\frac{d}{dt} }_{|t=0}E_{\phi}(\varphi+tv)=\int_{\Omega} v (dd^c \f)^n.
$$
In particular $\f \mapsto E_{\phi}(\varphi)$ is increasing.
\end{lem}

Here ${\mathcal D}({\Omega})$ denotes the space of smooth functions with compact support in $\Omega$.

\begin{proof}
Fix $\f \in \mathcal{T}_{\phi}^{\infty}(\Omega)$  and $v \in {\mathcal D}({\Omega})$.
Since $v$ is smooth with compact support, the function $\pm v+C\rho$ is psh for $C>0$ large enough, while
$\f-\e \rho$ is psh for $\e>0$ small enough. It follows that $\f+tv$
is psh for $t$ small enough.  

Set $\omega=dd^c \phi_0$. The function $\p_t=\f-\phi_0+tv$ has zero boundary values, and
$$
E_{\phi}(\varphi+tv)=\frac{1}{n+1}\sum_{j=0}^{n}\int_{\Omega} \p_t (\omega+dd^{c}\p_t)^{j}\wedge \omega^{n-j}.
$$
It follows from Stokes theorem, as all functions involved in the integration by parts are identically zero on $\partial \Omega$,
that
\begin{eqnarray*}
\lefteqn{(n+1) {\frac{d}{dt} }E_{\phi}(\varphi+tv)} \\
&=& \sum_{j=0}^{n}\int_{\Omega} \dot{\p_t} (\omega+dd^{c}\p_t)^{j}\wedge \omega^{n-j}
+\sum_{j=1}^{n}\int_{\Omega} j {\p_t} dd^c \dot{\p_t} \wedge (\omega+dd^{c}\p_t)^{j-1}\wedge \omega^{n-j} \\
&=&\sum_{j=0}^{n}\int_{\Omega} \dot{\p_t} (\omega+dd^{c}\p_t)^{j}\wedge \omega^{n-j}
+\sum_{j=1}^{n}\int_{\Omega} j \dot{\p_t} dd^c {\p_t} \wedge (\omega+dd^{c}\p_t)^{j-1}\wedge \omega^{n-j} \\
&=& \sum_{j=0}^{n}\int_{\Omega} (j+1) \dot{\p_t} (\omega+dd^{c}\p_t)^{j}\wedge \omega^{n-j}
-\sum_{j=1}^{n}\int_{\Omega} j \dot{\p_t}  (\omega+dd^{c}\p_t)^{j-1}\wedge \omega^{n-j+1} \\
&=& (n+1) \int_{\Omega} \dot{\p_t} (\omega+dd^{c}\p_t)^n
\end{eqnarray*}
writing $dd^c \p_t=(\omega+dd^c \p_t)-\omega$ in the third line, and then distributing and relabelling   
so as to obtain a telescopic series.
The formula follows for $t=0$.

In short the derivative of $E_{\phi}$ is the complex Monge-Amp\`ere operator $(dd^c \f)^n$ which is a positive
measure. It follows that $\f \mapsto E_{\phi}(\f)$ is increasing.
\end{proof}

\subsubsection{Continuous setting}

The previous result extends to the case of continuous 
plurisubharmonic functions that are not necessarily strictly psh. Recall that
$$
\mathcal{T}_{\phi}(\Omega):=
\left\{\varphi\in PSH(\Omega)\cap C^{0}(\bar{\Omega}), \, \varphi_{|\partial \Omega}=\phi 
\text{ and } \int_{\Omega} (dd^c \f)^n <+\infty \right\}.
$$

We would like to  extend Lemma \ref{lem:smoothIPP} to this less regular setting.
As   $\f+tv$ is not necessarily plurisubharmonic, we need to project it 
 onto the cone of all plurisubharmonic functions.
 The following result will thus be useful.
 
\begin{lem}
\label{lem:1}
Fix $\varphi\in \mathcal{T}_{\phi}(\Omega)$ and $f \in {\mathcal D}({\Omega})$.
 Then $P(\varphi+f)\in\mathcal{T}_{\phi}(\Omega)$ where
$$
P(\varphi+f):=\sup\big\{\p \in PSH(\Omega), \, \p \leq \varphi+f\big\}.
$$
Moreover $\big(dd^{c}P(\varphi+f)\big)^{n}$
is supported on the contact set $\{P(\varphi+f)=\varphi+f\}$.
\end{lem}

\begin{proof}
Since $\f+f$ is bounded and continuous, it is classical to check that the envelope $P(\varphi+f)$ is  a well-defined psh function.
As $f$ has compact support, one moreover checks that 
$P(\varphi+f)$ is continuous on $\partial\Omega$ with 
$P(\varphi+f)_{|\partial \Omega}=\varphi_{|\partial\Omega}=\phi$. 
 
Solving Dirichlet problems in small "balls" not containing the singular point,
it follows from a balayage argument that the Monge-Amp\`ere measure of the envelope
$\big(dd^{c}P(\varphi+f)\big)^{n}$ is supported on the contact set $\{P(\varphi+f)=\varphi+f\}$.
\end{proof}

We extend $E_{\phi}(\cdot)$ to $\mathcal{T}_{\phi}(\Omega)$ by monotonicity, setting
$$
E_{\phi}(\varphi):=\inf\left\{E_{\phi}(\p), \, \p \in \mathcal{T}_{\phi}^{\infty}(\Omega)\text{ and } \f \leq \p \right\}.
$$

It has been observed by Berman and Boucksom (in the setting of compact K\"ahler manifolds \cite{BB10})
that $E_{\phi} \circ P$ is still differentiable, with $(E_{\phi} \circ P)'=E_{\phi}'\circ P$.
This result extends to our local singular setting.

\begin{prop}
\label{prop:DerivEnergy}
Fix $\varphi\in \mathcal{T}_{\phi}(\Omega)$ and $f\in\mathcal{D}(\Omega)$. 
Then $t\to E_{\phi}\big(P(\varphi+tf)\big)$ is differentiable and
$$
{\frac{d}{dt}}_{|t=0}E_{\phi}\big(P(\varphi+tf)\big) =\int_{\Omega}f(dd^{c}\varphi)^{n}.
$$
\end{prop}

\begin{proof}
The proof is very similar to the one in the compact case, we provide it as a courtesy to the reader.
Set $\varphi_{t}:=P(\varphi+tf)$. 
By Lemma \ref{lem:Ener} below we have
\begin{equation} \label{eq:proj}
\int_{\Omega}(\varphi_{t}-\varphi)(dd^{c}\varphi_{t})^{n}\leq E_{\phi}(\varphi_{t})-E_{\phi}(\varphi) 
 \leq \int_{\Omega}(\varphi_{t}-\varphi)(dd^{c}\varphi)^{n}.
\end{equation}

Since $\f_t-\f \leq tf$, the second inequality yields
$$
\limsup_{t\to 0^{+}}\frac{E_{\phi}(\varphi_{t})-E_{\phi}(\varphi)}{t}\leq \int_{X}f(dd^{c}\varphi)^{n},
$$
and
$$
\liminf_{t\to 0^{-}}\frac{E_{\phi}(\varphi_{t})-E_{\phi}(\varphi)}{t}\geq\int_{X}f (dd^{c}\varphi)^{n}.
$$

It follows from Lemma \ref{lem:1} that
 $(dd^{c}\varphi_{t})^{n}$ is supported on $\{\varphi_{t}=\varphi+tf\}$, hence
the first inequality in \eqref{eq:proj} yields
$$
\int_{\Omega}\frac{\varphi_{t}-\varphi}{t}(dd^{c}\varphi_{t})^{n}=\int_{\Omega}f(dd^{c}\varphi_{t})^{n}.
$$
Now $(dd^{c}\varphi_{t})^{n}\to (dd^{c}\varphi)^{n}$ weakly since $\varphi_{t}\to \varphi$ uniformly, therefore
$$
\liminf_{t\to 0^{+}}\frac{E_{\phi}(\varphi_{t})-E_{\phi}(\varphi)}{t}\geq \liminf_{t\to 0^{+}}\int_{\Omega}f(dd^{c}\varphi_{t})^{n}=\int_{\Omega}f(dd^{c}\varphi)^{n},
$$
and
$$
\limsup_{t\to0^{-}}\frac{E_{\phi}(\varphi_{t})-E_{\phi}(\varphi)}{t}\leq\limsup_{t\to 0^{-}}\int_{\Omega}f(dd^{c}\varphi_{t})^{n}=\int_{\Omega}f(dd^{c}\varphi)^{n}. 
$$
\end{proof}

\begin{lem}  \label{lem:Ener}
For any $\varphi_{1},\varphi_{2}\in \mathcal{T}_{\phi}(\Omega)$,
\begin{equation}
    \int_{\Omega}(\varphi_{1}-\varphi_{2})(dd^{c}\varphi_{1})^{n}\leq E_{\phi}(\varphi_{1})-E_{\phi}(\varphi_{2})\leq \int_{\Omega}(\varphi_{1}-\varphi_{2})(dd^{c}\varphi_{2})^{n},
\end{equation}
while if $\varphi_{1}\leq \varphi_{2}$ then
\begin{equation}
    E_{\phi}(\varphi_{1})-E_{\phi}(\varphi_{2})\leq \frac{1}{n+1}\int_{X}(\varphi_{1}-\varphi_{2})(dd^{c}\varphi_{1})^{n}.
\end{equation}

The energy is continuous along decreasing sequence in $\mathcal{T}_{\phi}(\Omega)$.
\end{lem}

\begin{proof}
It follows from  Stokes theorem that
$$
E_{\phi}(\varphi_{1})-E_{\phi}(\varphi_{2})=\frac{1}{n+1}\sum_{j=0}^{n}\int_{\Omega}(\varphi_{1}-\varphi_{2})(dd^{c}\varphi_{1})^{j}\wedge (dd^{c}\varphi_{2})^{n-j}
$$
and  
$$
\int_{\Omega}(\varphi_{1}-\varphi_{2})(dd^{c}\varphi_{1})^{j+1}\wedge (dd^{c}\varphi_{2})^{n-j-1}\leq \int_{\Omega}(\varphi_{1}-\varphi_{2})(dd^{c}\varphi_{1})^{j}\wedge (dd^{c}\varphi_{2})^{n-j},
$$
for any $j=0,\dots,n-1$. The desired inequalities  follow. 

Let $\varphi_{j}\in \mathcal{T}_{\phi}(\Omega)$ be a decreasing sequence converging to $\varphi \in \mathcal{T}_{\phi}(\Omega)$.
We  obtain
$$
0\leq E_{\phi}(\varphi_{j})-E_{\phi}(\varphi)\leq \int_{\Omega}(\varphi_{j}-\varphi)(dd^{c}\varphi)^{n}\to 0
$$
as $j\to +\infty$ by Monotone Convergence Theorem.
\end{proof}

\subsubsection{Finite energy class}

Let $PSH_{\phi}(\Omega)$ denote the set of  decreasing limits of
functions in $\mathcal{T}_{\phi}(\Omega)$.
We extend $E_{\phi}$ to $PSH_{\phi}(\Omega)$ by monotonicity, setting
$$
E_{\phi}(\varphi):=\inf\left\{E_{\phi}(\p), \, \p \in \mathcal{T}_{\phi}(\Omega)\text{ and } \f \leq \p \right\}.
$$

\begin{defi}
We set $\mathcal{E}^1(\Omega):=\left\{ \f \in PSH_{\phi}(\Omega); \; E_{\phi}(\f)>-\infty \right\}$.
\end{defi}

This "finite energy class" has been introduced and intensively studied by 
Cegrell for smooth domains of $\C^n$. His analysis extends to our mildly singular context.
We summarize here the key facts that we shall need.

\begin{theorem}[Cegrell] \label{thm:Cegrell}
The complex Monge-Amp\`ere operator $(dd^c \cdot)^n$ and the energy $E_{\phi}$ are well-defined on the class $\mathcal{E}^1(\Omega)$. Moreover
\begin{itemize}
\item functions in $\mathcal{E}^1(\Omega)$ have zero Lelong numbers;
\item the sets $\mathcal{G}_b(\Omega)=\{ \f \in \mathcal{E}^1(\Omega), \; -b \leq E_{\phi}(\f) \}$ are compact for all $b \in \R$;
\item Lemma \ref{lem:Ener} holds if $\f_1,\f_2 \in \mathcal{E}^1(\Omega)$;
\item if $\mu$ is a  non pluripolar probability measure such that 
$\mathcal{E}^1(\Omega) \subset L^1(\mu)$, then there exists a unique function 
$v \in \mathcal{E}^1(\Omega) \cap \mathcal{F}_1(\Omega)$ such that $\mu=(dd^c v)^n$.
\end{itemize}
\end{theorem}

We refer the reader to \cite[Theorems 3.8, 7.2 and 8.2]{Ceg98} for the proof of these results when
$\Omega$ is smooth.

\subsection{Ding functional}

\subsubsection{Euler-Lagrange equation}

The Ding functional is  
$$
F_{\gamma}(\varphi):=E_{\phi}(\varphi)+\frac{1}{\gamma}\log \int_{\Omega}e^{-\gamma \varphi}d\mu_{p}
$$

\begin{prop}
\label{prop:MaxSol}
If $\varphi$ maximizes $F_{\gamma}$ over $\mathcal{T}_{\phi}(\Omega)$
then $\varphi$ solves the complex Monge-Amp\`ere equation (\ref{eqn:MA}). 
\end{prop}

\begin{proof}
Assume that $\varphi$ maximizes $F_{\gamma}$ over $\mathcal{T}_{\phi}(\Omega)$,
fix $f \in \mathcal{D}(\Omega)$, and set $\varphi_{t}:=P(\varphi+tf)$.
 Then
$$
E_{\phi}(\varphi_{t})+\frac{1}{\gamma}\log \int_{\Omega}e^{-\gamma(\varphi+tf)}d\mu_{p}\leq F_{\gamma}(\varphi_{t})\leq F_{\gamma}(\varphi),
$$
i.e. the function $t\to E_{\phi}(\varphi_{t})+\frac{1}{\gamma}\log\int_{\Omega}e^{-\gamma(\varphi+tf)}d\mu_{p}$ 
reaches its maximum at $t=0$. 
Combining Proposition \ref{prop:DerivEnergy}   and Lemma \ref{lem:DerivAction} below, we obtain
$$
0=\frac{d}{dt}\Big(E_{\phi}(\varphi_{t})+\frac{1}{\gamma}\log\int_{\Omega}e^{-\gamma(\varphi+tf)}d\mu_{p}\Big)
=\int_{\Omega}f\left((dd^{c}\varphi)^{n}-\frac{e^{-\gamma\varphi}d\mu_{p}}{\int_{\Omega}e^{-\gamma\varphi}d\mu_{p}}\right),
$$
i.e. $\varphi$ solves $(\ref{eqn:MA})$. 
\end{proof}

\begin{lem}
\label{lem:DerivAction}
Fix $\varphi \in \mathcal{T}_{\phi}(\Omega)$, 
$f \in {\mathcal D}(\Omega)$, and set $\p_t:=\varphi+tf$. Then
$$
\frac{d}{dt}\Big(\log \int_{\Omega}e^{-\gamma\p_t}d\mu_{p}\Big)_{|t=0}=
-\gamma\frac{\int_{\Omega}fe^{-\gamma\varphi}d\mu_{p}}{\int_{\Omega}e^{-\gamma\varphi}d\mu_{p}}.
$$
\end{lem}

\begin{proof}
By chain rule, it is enough to observe that
$$
\frac{\int_{\Omega}e^{-\gamma \p_t}d\mu_{p}-\int_{\Omega}e^{-\gamma \varphi}d\mu_{p}}{t}=-\int_{\Omega}e^{-\gamma\varphi}\left(\frac{1-e^{-t\gamma f}}{t} \right)d\mu_{p}
$$
and   to apply Lebesgue Dominated  Convergence Theorem to conclude.
\end{proof}

\subsubsection{Coercivity}

In order to solve (\ref{eqn:MA}), one is  lead to try and maximize $F_{\gamma}$. 
We will show in Section \ref{sec:RicciIteration} that
when $F_{\gamma}$ is {\it coercive},
 the complex Monge-Amp\`ere equation (\ref{eqn:MA}) admits a solution 
 $\varphi\in \mathcal{T}_{\phi}(\Omega)$
 which is  smooth away from $p$.

\begin{defi}
The functional $F_{\gamma}$ is coercive if there exists $A,B>0$ such that
$$
F_{\gamma}(\varphi)\leq A E_{\phi}(\varphi)+B
$$
for all $\varphi\in \mathcal{T}_{\phi}(\Omega)$. 
\end{defi}

We observe in Lemma \ref{lem:coercive} that 
$E_{\phi}(\varphi) \leq  C(\phi_0)$
is  bounded from above, uniformly in $\varphi \in  \mathcal{T}_{\phi}(\Omega)$.
In particular if $F_{\gamma}$ is coercive with \emph{slope} $A>0$ 
then it is coercive for any $A'\in (0,A]$. 
We can thus assume,  without loss of generality, that $A\in(0,1)$.
The coercivity property is then equivalent to
$$
\frac{1}{\gamma}\log \int_{\Omega}e^{-\gamma\varphi}d\mu_{p}\leq (1-A)(-E_{\phi}(\varphi))+B,
$$
or, equivalently, to the following Moser-Trudinger inequality
$$
   \left( \int_{X}e^{-\gamma \varphi}d\mu_{p} \right)^{\frac{1}{\gamma}} \leq C  e^{(1-A)(-E_{\phi}(\varphi))}.
$$

We summarize these observations in the following.

\begin{prop} \label{pro:CoerciveMT}
Fix $\gamma>0$. The  following properties are equivalent:
\begin{itemize}
    \item[i)] $F_{\gamma}$ is coercive;
    \item[ii)] there exists $C_{\gamma}>0$ and $a\in (0,1)$ such that for all 
    $\varphi \in \mathcal{T}_{\phi}(\Omega)$,
    $$
    \Big(\int_{\Omega}e^{-\gamma\varphi}d\mu_{p}\Big)^{1/\gamma}\leq C_{\gamma} e^{-a E_{\phi}(\varphi)}.
    $$
\end{itemize}
\end{prop}

It follows from H\"older inequality (and the normalization $\mu_p(\Omega)=1$)
that if 
\begin{equation} \label{eq:MT}
\Big(\int_{\Omega}e^{-\gamma\varphi}d\mu_{p}\Big)^{1/\gamma}\leq Ce^{-E_{\phi}(\varphi)}
\end{equation}
holds for some $\gamma>0$, then it also holds for any $\gamma'<\gamma$.
It is thus natural to introduce the following critical exponent. 
 
\begin{defi}
We set
$$
\gamma_{crit}(X,p):=\sup\Big\{\gamma>0, \, (\ref{eq:MT}) \, \mbox{holds for all }\, 
\f \in {\mathcal F}_1(\Omega) \Big\}.
$$
\end{defi}

Recall that ${\mathcal F}_1(\Omega)$ denotes the set of psh functions $\f \in {\mathcal F}(\Omega)$
with $\phi$-boundary values,  whose Monge-Amp\`ere measure is well-defined with  
$\int_{\Omega} (dd^c \f)^n \leq 1$.

\begin{lem} \label{lem:coercive}
The functional $E_{\phi}$ is bounded from above on $\mathcal{T}_{\phi}(\Omega)$. Moreover
\begin{itemize}
\item if  $F_{\gamma}$ is coercive then $\gamma \leq \gamma_{crit}(X,p)$;
\item conversely if $\gamma < \gamma_{crit}(X,p)$ then $F_{\gamma}$ is coercive.
\end{itemize}
\end{lem}

\begin{proof}
Consider $\tilde{\phi_0}=P(\phi):=\sup \{ \p, \p  \in  \mathcal{T}_{\phi}(\Omega)\}$.
 This is the largest psh function
in $\Omega$ such that $\tilde{\phi_0} =\phi$ on $\partial \Omega$. The reader can check that it is
continuous on ${\overline \Omega}$ and satisfies
$(dd^c \tilde{\phi_0})^n=0$ in $\Omega$. If $\f \in \mathcal{T}_{\phi}(\Omega)$ then 
$\f \leq \tilde{\phi_0}$ hence $E_{\tilde{\phi_0}}(\f) \leq 0$. Thus
$$
E_{{\phi_0}}(\f) = E_{\tilde{\phi_0}}(\f)+E_{\phi_0}(\tilde{\phi_0}) 
\leq E_{\phi_0}(\tilde{\phi_0}),
$$
hence $E_{{\phi_0}}(\f)$ is uniformly bounded from above independently of the choice of $\phi_0$.

Similarly the coercivity of $F_{\gamma}$ or the Moser-Trudinger inequality \eqref{eq:MT} do not
depend on the choice of $\phi_0$. In the remainder of this proof we thus assume that 
$\phi_0=P(\phi)$. Since $E_{\phi}(\f) \leq 0$ in this case, it follows from
Proposition \ref{pro:CoerciveMT} that if $F_{\gamma}$ is coercive then
\eqref{eq:MT} holds, hence $\gamma \leq \gamma_{crit}(X,p)$. 

Conversely assume $\gamma < \gamma_{crit}(X,p)$. Fix $\gamma<\gamma'<\gamma_{crit}(X,p)$
and $\lambda=\gamma/\gamma'<1$.
We can assume without loss of generality that $\lambda$ is close to $1$.
We assume first that $\phi \equiv 0$. For $\f \in \mathcal{T}_{0}(\Omega)$
we observe that $\lambda \f \in \mathcal{T}_{0}(\Omega)$, with
$E_0(\lambda \f)=\lambda^{n+1} E_0(\f)$.
The Moser-Trudinger \eqref{eq:MT} applied to $(\gamma',\lambda \f)$ thus yields
$$
\Big(\int_{\Omega}e^{-\gamma\varphi}d\mu_{p}\Big)^{1/\gamma}
\leq C_{\gamma'} e^{-\lambda^n  E_0(\varphi)},
$$
so that $F_{\gamma}$ is coercive.    

We now treat the general case, replacing the condition $\phi \equiv 0$ by $(dd^c \phi_0)^n \equiv 0$.
For $\f \in \mathcal{T}_{\phi}(\Omega)$ we observe that
$\f_{\lambda}=\lambda \f+(1-\lambda) \phi_0 \in \mathcal{T}_{\phi}(\Omega)$, with
$\f_{\lambda}-\phi_0=\lambda(\f-\phi_0) \leq 0$ and 
\begin{eqnarray*}
(n+1) E_{\phi_0}(\f_{\lambda})
&=& \lambda \sum_{j=0}^{n-1} \int_{\Omega} (\f-\phi_0) (dd^c \f_{\lambda})^j \wedge (dd^c \phi_0)^{n-j} \\
&=& \lambda \sum_{k=1}^n \sum_{j=k}^{n} 
\left( \begin{array}{c} j \\ k \end{array} \right) 
\lambda^k (1-\lambda)^{j-k}
\int_{\Omega} (\f-\phi_0) (dd^c \f)^k \wedge (dd^c \phi_0)^{n-k}.
\end{eqnarray*}

Now $\sum_{j=k}^{n} 
\left( \begin{array}{c} j \\ k \end{array} \right) 
\lambda^k (1-\lambda)^{j-k}  \leq a <1$
for all $1 \leq k \leq n$,
since $\lambda<1$ can be chosen arbitrarily close to $1$.
Thus $E_{\phi_0}(\f_{\lambda}) \geq a \lambda E_{\phi_0}(\f)$ and the result follows
as previously by applying  the Moser-Trudinger inequality \eqref{eq:MT} to $\f_{\lambda}$.
 \end{proof}

\subsection{Invariance} \label{sec:invariance}

In this section we observe that the critical exponent
$\gamma_{crit}$ is essentially independent of the domain $\Omega$
and boundary values.

\subsubsection{Enlarging the domain}

We first reduce to the case of zero boundary values.

\begin{prop}
Let $\Omega_2$ be a smooth strongly pseudoconvex domain containing $\overline{\Omega}$.
If the Moser-Trudinger inequality holds for $(\gamma,\Omega_2,0)$ then it holds for 
$(\gamma, \Omega,\phi)$.
\end{prop}

\begin{proof}
Consider indeed $\f \in {\mathcal T}_{\phi}(\Omega)$ and set
$$
\f_2:=\sup \left\{ u \in {\mathcal T}_0(\Omega_2), \; 
\text{ such that } u \leq \f \text{ in } \Omega \right\}.
$$

The family ${\mathcal F}$ of such functions is non-empty, as it contains $A \rho_2$ for some large $A>1$,
where $\rho_2$ is a psh defining function for $\Omega_2$.
Moreover ${\mathcal F}$ is uniformly bounded from above by $0$, so 
the upper-envelope $\f_2$ is well-defined and psh, as ${\mathcal F}$ is compact.
Finally  $\f_2 \geq A \rho_2$, hence $\f_2$ has zero boundary values, 
and $\f_2$ is lower semi-continuous, as an envelope of continuous functions,
thus $\f_2 \in {\mathcal T}_0(\Omega_2)$.

Since $\f_2 \leq \f$ in $\Omega$, we observe that 
$$
\int_{\Omega} e^{-\gamma \f} d\mu_p \leq \int_{\Omega_2} e^{-\gamma \f_2} d\mu_p.
$$
Our claim will follow if we can show that on the other hand
$E_0(\f_2) \geq E_{\phi}(\f)$.

If $\f$ is smooth one can  show, 
by adapting standard techniques, that
\begin{itemize}
\item $\f_2$ is ${\mathcal C}^{1,\overline{1}}$-smooth in $\Omega_2 \setminus \{p\}$;
\item $(dd^c \f_2)^n=0$ in $\Omega_2 \setminus \Omega$
and  $(dd^c \f_2)^n={{\bf 1}}_{\{ \f_2 =\f\}} (dd^c \f)^n$ in $\overline{\Omega}$.
\end{itemize}
Assuming $\phi \geq 0$ and $\phi_0=\sup \{ \p, \; \p \in {\mathcal T}_{\phi}(\Omega) \}$,
we infer
\begin{eqnarray*}
E_0(\f_2) &=& \frac{1}{n+1}\int_{\Omega} \f_2 (dd^c \f_2)^n
= \frac{1}{n+1}\int_{\Omega} {{\bf 1}}_{\{ \f_2 =\f\}} \f (dd^c \f)^n \\
& \geq & \frac{1}{n+1}\int_{\Omega}   \f (dd^c \f)^n
\geq \frac{1}{n+1}\int_{\Omega}   (\f-\phi) (dd^c \f)^n 
\geq  E_{\phi}(\f).
\end{eqnarray*}

To get rid of the assumption $\phi \geq 0$, we observe that
the Moser-Trudinger inequality holds for given boundary data 
$\phi$ if and only if does so for $\phi+c$, for any $c \in \R$
(by changing $\f$ in $\f+c$).

Using Lemma \ref{lem:smoothing}, one can uniformly approximate $\f$ by a  
sequence of smooth $\f_{j} \in {\mathcal T}_{\phi}(\Omega)$.
The corresponding sequence $\f_{2,j}$ uniformly converges to $\f_2$,
and we obtain the desired inequality by passing to the limit
in $E_0(\f_{2,j}) \geq E_{\phi}(\f_{j})$.
\end{proof}

\subsubsection{Rescaling}

We now assume that $\phi=0$ and reformulate the coercivity property 
 after an appropriate rescaling.
Observe that for any $\lambda>0$, the map
 $$
 \varphi\in\mathcal{T}_{0}(\Omega) \mapsto \lambda \f \in \mathcal{T}_{0}(\Omega)
$$
is a homeomorphism.
This allows to reformulate the Moser-Trudinger inequality:

\begin{prop}
The following statements are equivalent:
\begin{itemize}
    \item[a)] $F_{\gamma}$ is coercive;
    \item[b)] $\exists C>0, B\in(0,1)$ such that  for all $\varphi\in\mathcal{T}_{0}(\Omega)$,
    $
    \int_{\Omega}e^{-\varphi}d\mu_{p}\leq Ce^{-\frac{B}{\gamma^{n}}E_{0}(\varphi)}.
    $
   \end{itemize}
\end{prop}

In particular we can define the following critical exponent.

\begin{defi}
We set
$$
\beta_{crit}:=\inf \left\{\beta>0; \, \sup_{\varphi\in\mathcal{T}_{0}(\Omega)}
\Big(\int_{\Omega}e^{-\varphi}d\mu_{p}/e^{-\beta E_{0}(\varphi)}\Big)<+\infty \right\}.
$$
\end{defi}

Note that $\gamma_{crit}^{n}=1/\beta_{crit}$, hence
it follows from the previous analysis that $F_{\gamma}$ is coercive 
if and only if $\gamma< \beta_{crit}^{-1/n}$.
When $p\in X$ is smooth, it has been shown in \cite[Theorem 9]{GKY13} and
 independently \cite[Theorem 1.5]{BB22}   that
$$
\beta_{crit}(\Omega)=\frac{1}{(n+1)^{n}},
$$
or equivalently that $\gamma_{crit}(\Omega)=n+1$.
In particular it does not depend on $\Omega$.

\smallskip

We extend this independence to the case of conical singularities.

\begin{prop}
Assume that $(X,p)$ is a conical singularity and fix $\lambda \in \C^*$.
The Moser-Trudinger inequality holds for $(\gamma,\Omega,0)$
iff it does so for $(\gamma,\lambda \Omega, 0)$.
\end{prop}

\begin{proof}
Let $D_{\lambda}$ denote the dilatation $z \mapsto \lambda z$
and set $\Omega_{\lambda}=D_{\lambda}( \Omega)$.
The fact that $(X,p)$ is conical ensures that 
$D_{\lambda}^* \mu_p =|\lambda |^{2a} \mu_p$ for some $a>0$.

For $\f \in {\mathcal T}_0(\Omega_{\lambda})$ we set
$\f_{\lambda}=\f \circ D_{\lambda} \in {\mathcal T}_0(\Omega)$ and
observe that
$$
|\lambda |^{2a} \left( \int_{\Omega} e^{-\gamma \f_{\lambda}} d\mu_p \right)^{1/\gamma}
=\int_{\Omega_{\lambda}} e^{-\gamma \f} d\mu_p,
$$
while $E_{\Omega,0}(\f_{\lambda})=E_{\Omega_{\lambda},0}(\f)$.
The conclusion follows.
\end{proof}

\section{Upper bound for the coercivity}

The purpose of this section is to establish the following upper bound
$$
\gamma_{crit}(X,p) \leq \frac{n+1}{n}\widehat{\mathrm{vol}}(X,p)^{1/n}.
$$

\subsection{Functions with algebraic singularities}

Let $\mathcal{I}$ be a coherent ideal sheaf, and 
$f_{1},\dots,f_{N}\in \mathcal{O}_{X,p}$ be local generators of $\mathcal{I}_{p}$. The plurisubharmonic function
$$
\varphi_{\mathcal{I}}:=\log\Big(\sum_{i=1}^{N}\lvert f_{i}\rvert^{2}\Big)
$$
 is well defined  near $p$, with  \emph{algebraic singularities}  encoded in $\mathcal{I}$. 
 

\begin{prop}
\label{prop:lct}
Let $\mathcal{I}$ be a coherent ideal sheaf supported at $p$. Then
$$
e(X,\mathcal{I})=\int_{\{p\}}\big(dd^c\varphi_\mathcal{I}\big)^n \; \; 
$$
and
$$
\mathrm{lct}(X,\mathcal{I})=\sup\left\{\alpha >0: \, 
\int_{\Omega}e^{-\alpha \varphi_{\mathcal{I}}}d\mu_{p}<+\infty\right\}
$$
where $\Omega$ is any (small) neighborhood of $p\in X$.
\end{prop}

These algebraic quantities are thus independent of the choice of generators.

\begin{proof}
The equality
  $e(X,\mathcal{I})=\int_{\{p\}}\big(dd^c\varphi_{\mathcal{I}}\big)^n$ is classical when $\Omega$ is smooth (see e.g. \cite[Lemma 2.1]{Dem12}), and the proof can be adapted to 
 the singular  context (see \cite[Chapter 4]{Dem85}).

\smallskip

Let $\pi:\tilde{\Omega}\to \Omega$ be a local log resolution of the ideal $(X,\mathcal{I})$, i.e. a composition of blow-ups such that
$
\pi^{*}\mu_{p}=\prod_{j=1}^{N}\lvert s_{E_{j}} \rvert^{2a_{j}}dV_{\tilde{\Omega}}
$
and
$$
\pi^{-1}\mathcal{I}\cdot \mathcal{O}_{\tilde{\Omega}}=\mathcal{O}_{\tilde{\Omega}}\Big(-\sum_{j=1}^{M}b_{j}E_{j}\Big)
$$
where $b_{j}\in\N$, $a_{j}\in\Q_{>-1}$,
and $E_{1},\dots,E_{M}$ have simple normal crossings. 
Observe that
$$
\int_{\Omega}e^{-\alpha \varphi_{\mathcal{I}}}d\mu_{p}\sim\int_{\tilde{\Omega}}\frac{\prod_{j=1}^{M}\lvert s_{E_{j}}\rvert^{2a_{j}}}{\prod_{j=1}^{M}\lvert s_{E_{j}}\rvert^{2b_{j}\alpha}}dV_{\tilde{\Omega}}=\int_{\tilde{\Omega}}\prod_{j=1}^{M}\lvert s_{E_{j}} \rvert^{2(a_{j}-\alpha b_{j})}dV_{\tilde{\Omega}},
$$
 is finite if and only if $a_{j}-\alpha b_{j}>-1$ for any $j=1,\dots,M$, i.e. if and only if
$$
\alpha< \inf_{j=1,\dots,M}\frac{a_{j}+1}{b_{j}}=\mathrm{lct}(X,\mathcal{I}),
$$
as recalled in Definition \ref{def:lct}.
\end{proof}

\subsection{Approximate Green functions}
   
 The functions $\lambda \varphi_{\mathcal{I}}$ play the role of  Green functions
 adapted to the singularity $(X,p)$.
 We show here how to approximate them from above by smooth
 functions with prescribed boundary values.

\begin{lem}
\label{lem:Cons}
Let $\mathcal{I}$ be a coherent ideal sheaf supported at $p$,
and let $f_{1},\dots,f_{m}$ denote local generators of $\mathcal{I}$.
Fix an open set $\Omega'\Subset \Omega$.
There exists a family $\big\{\varphi_{\mathcal{I},\lambda,\epsilon}\big\}_{\lambda>0, \epsilon> 0}\in PSH(\Omega)\cap \mathcal{C}^{\infty}(\bar{\Omega})$ such that
\begin{itemize}
    \item[i)] $\varphi_{\mathcal{I},\lambda,\epsilon|\partial\Omega}=\phi$ for any $\lambda>0, \epsilon\in [0,1]$;
    \item[ii)] $\varphi_{\mathcal{I},\lambda,\epsilon}=\lambda\log\big(\sum_{j=1}^{m}\lvert f_{j} \rvert^{2}+\epsilon^{2}\big)+\phi_{0}$ in $\Omega'$, 
    \item[iii)] $\varphi_{\mathcal{I},\lambda,\epsilon}\searrow \varphi_{\mathcal{I},\lambda,0}=:\varphi_{\mathcal{I},\lambda}$ as $\epsilon\searrow 0$ for any $\lambda>0$ fixed.
\end{itemize}
\end{lem}


\begin{proof}
Without loss of generality we can assume that 
that $\sum_{j=1}^{m}\lvert f_{j}\rvert^{2}\leq 1/e-1$ in $\Omega$. 
Let $\rho$ be a smooth psh exhaustion  for $\Omega$ and fix $0< r\ll 1$, $0<\delta\ll 1$ small enough. There exists $A>0$ big enough and relatively compact open sets $B_{r}(0)\Subset\Omega'\Subset \tilde{\Omega}\Subset \Omega$ such that
\begin{gather*}
    \log\big(\sum_{j=1}^{m}\lvert f_{j} \rvert^{2}+1\big)+\delta\leq A\rho \,\,\, \mathrm{over}\,\,\, \Omega\setminus \tilde{\Omega}\\
    \log\big(\sum_{j=1}^{m}\lvert f_{j}\rvert^{2}\big)-\delta\geq A\rho \,\,\, \mathrm{over}\,\,\, \Omega'\setminus B_{r}(0)
\end{gather*}
We infer that
$$
u_{\mathcal{I},\epsilon}:=
\begin{cases}
A\rho \,\,\, \mathrm{on}\,\,\, \Omega\setminus \tilde{\Omega}\\
\max_{\delta}\Big(\log\big(\sum_{j=1}^{m}\lvert f_{j}\rvert^{2}+\epsilon^{2}\big), A\rho\Big)\,\,\, \mathrm{on}\,\,\, \tilde{\Omega}\setminus \Omega'\\
\log\big(\sum_{j=1}^{m}\lvert f_{j}\rvert^{2}+\epsilon^{2}\big)\,\,\,\mathrm{on}\,\,\, \Omega'
\end{cases}
$$
is a decreasing family (in $\epsilon\in[0,1]$) of plurisubharmonic functions which are smooth in $\bar{\Omega}\setminus\{p\}$ (smooth in $\bar{\Omega}$ for $\epsilon>0$) and which are identically $0$ on $\partial \Omega$. 
Here $\max_{\delta}(\cdot,\cdot)$ denotes the regularized maximum. 
The lemma follows by setting
$
\varphi_{\mathcal{I},\lambda,\epsilon}:=\lambda u_{\mathcal{I},\epsilon}+\phi_{0}.
$
\end{proof}

We now compute the asymptotic behavior, as $\e$ decreases to $0$,
 of the quantities involved in the expected Moser-Trudinger inequality.

\begin{lem}
\label{lem:lct}
Let $\mathcal{I}$ and $\{\varphi_{\mathcal{I},\lambda,\epsilon}\}_{\epsilon\in(0,1]}\subset\mathcal{T}(\Omega)$
be as in Lemma \ref{lem:Cons}. Then for any $\gamma>0, \lambda>0$ fixed there exists a constant $C_{\lambda,\gamma}\in\R$ (independent of $\epsilon$) such that
\begin{equation}
    \label{eqn:EstLct}
    C_{\lambda,\gamma}+\big(\gamma\lambda-\mathrm{lct}(X,\mathcal{I})\big)\log\epsilon^{-2}\leq \log\int_{\Omega}e^{-\gamma \varphi_{\mathcal{I},\lambda,\epsilon}}d\mu_{p}
\end{equation}
for all $0<\epsilon<\epsilon_0$.
\end{lem}

\begin{proof}
Taking a log resolution $\pi: Y\to X$ we obtain
\begin{eqnarray*}
\int_{\Omega}e^{-\gamma \varphi_{\mathcal{I,\lambda,\epsilon}}}d\mu_{p}
&\geq & \int_{\Omega'}\frac{1}{\Big(\sum_{j=1}^{m}\lvert f_{j}\rvert^{2}+\epsilon^{2}\Big)^{\gamma\lambda}}d\mu_{p} \\
&\geq &C_{1}\int_{\pi^{-1}(\Omega')}\frac{\prod_{j=1}^{M}\lvert s_{E_{j}}\rvert^{2a_{j}}}{\Big(\prod_{j=1}^{M}\lvert s_{E_{j}} \rvert^{2b_{j}}+\epsilon^{2}\Big)^{\gamma\lambda}}dV_{\pi^{-1}(\Omega')},
\end{eqnarray*}
where $C_{1}$ is a uniform constant (independent on $\epsilon$).
We set
$$
f:=\frac{\prod_{j=1}^{M}\lvert s_{E_{j}}\rvert^{2a_{j}}}{\Big(\prod_{j=1}^{M}\lvert s_{E_{j}} \rvert^{2b_{j}}+\epsilon^{2}\Big)^{\gamma\lambda}}.
$$

We can assume without loss of generality  that $\mathrm{lct}(X,\mathcal{I})=\frac{a_{1}+1}{b_{1}}$.
Pick $x\in E_{1}, x\notin E_{j}, j=2,\dots,M$. We can find $0<r\ll 1$ so small  
that $B_{r}(x)\cap E_{j}=\emptyset$ for any $j=2,\dots,M$. 
We choose holomorphic coordinates $(z_{1},\dots,z_{n})$ centered at $x$ such that $E_{1}=\{z_{1}=0\}$. Thus, setting $a:=a_{1},b:=b_{1}$ and $c:=\gamma\lambda$ we get
$$
\int_{\pi^{-1}(\Omega')}fdV_{\pi^{-1}(\Omega')}\geq C_{2}\int_{B_{r}(0)}\frac{\lvert z_{1} \rvert^{2a}}{\big(\lvert z_{1}\rvert^{2b}+\epsilon^{2}\big)^{c}}d\lambda(z)=C_{3}\int_{0}^{r}\frac{u^{2a+1}}{\big(u^{2b}+\epsilon^{2}\big)^{c}}du
$$
where $C_{2}, C_{3}$ are uniform constants. If $c\leq \frac{a+1}{b}$ (i.e. $\gamma\lambda\leq \mathrm{lct}(X,\mathcal{I})$) then
$$
\int_{0}^{r}\frac{u^{2a+1}}{\big(u^{2b}+\epsilon^{2}\big)^{c}}du\geq \int_{0}^{1}\frac{u^{2a+1}}{\big(u^{2b}+1\big)^{c}}=:C_{4}
$$
and (\ref{eqn:EstLct}) trivially follows. If $c>\frac{a+1}{b}$ then the substitution $v:=u/\epsilon^{1/b}$ yields
\begin{eqnarray*}
\int_{0}^{r}\frac{u^{2a+1}}{\big(u^{2b}+\epsilon^{2}\big)^{c}}du
&=& \epsilon^{-2\big(c-\frac{a+1}{b}\big)}\int_{0}^{r/\epsilon^{1/b}}\frac{v^{2a+1}}{\big(v^{2b}+1\big)^{c}} \\
&\geq & \epsilon^{-2\big(\gamma\lambda-\mathrm{lct}(X,\mathcal{I})\big)}\int_{0}^{r}\frac{v^{2a+1}}{\big(v^{2b}+1\big)^{c}} dv.
\end{eqnarray*}
The lemma follows.
\end{proof}

\begin{lem}
\label{lem:Mult}
Let $\mathcal{I}$ and  $\varphi_{\mathcal{I},\lambda,\epsilon}\in\mathcal{T}(\Omega)$ be as in Lemma \ref{lem:Cons}. 
There exist positive constants $\{C_{\ell,\lambda}\}_{\ell\in\N,\lambda>0}$ 
and a family of functions $F_{\ell}:(0,1]\to \R_{>0}$ such that
$$
-E_{\phi}(\varphi_{\mathcal{I},\lambda,\epsilon})\leq C_{\ell,\lambda}+\frac{\lambda^{n+1}}{n+1}F_{\ell}(\epsilon)\log\epsilon^{-2},
$$
for any $\epsilon \in (0,1]$, where
\begin{itemize}
\item $\{C_{\ell,\lambda}\}_{\ell\in\N,\lambda>0}$ is independent of $\epsilon\in(0,1]$;
\item $F_{\ell}(\epsilon)\to F_{\ell}(0)=:e_{\ell}>0$ as $\epsilon\searrow 0$;
\item $e_{\ell}\searrow e(X,\mathcal{I})$ as $\ell\to +\infty$.
\end{itemize}
 
\end{lem}

\begin{proof}
We take a sequence $\{\Omega_{\ell}\}_{\ell\in\N}$ of open sets such that 
$\Omega_{\ell+1}\Subset \Omega_{\ell}$ for any $\ell\in\N$ and such that $\bigcap_{\ell\in\N}\Omega_{\ell}=\{p\}$. 
Since $\Omega_{\ell}\subset \Omega'$ (same notation of Lemma \ref{lem:Cons}) for $\ell\in\N$ big enough, 
we obtain
\begin{multline}
    \label{eqn:Calc}
    -E_{\phi}(\varphi_{\mathcal{I},\lambda,\epsilon})=\frac{1}{n+1}\sum_{j=0}^{n}\int_{\Omega}\big(\phi_{0}-\varphi_{\mathcal{I},\lambda,\epsilon}\big)\big(dd^{c}\varphi_{\mathcal{I},\lambda,\epsilon}\big)^{j}\wedge \big(dd^{c}\phi_{0}\big)^{n-j}\\
    =\frac{1}{n+1}\sum_{j=0}^{n}\int_{\Omega\setminus\Omega_{\ell}}\big(\phi_{0}-\varphi_{\mathcal{I},\lambda,\epsilon}\big)\big(dd^{c}\varphi_{\mathcal{I},\lambda,\epsilon}\big)^{j}\wedge \big(dd^{c}\phi_{0}\big)^{n-j}\\
    -\frac{1}{n+1}\sum_{j=0}^{n}\int_{\Omega_{\ell}} \lambda^{j+1} \log\big(\sum_{k=1}^{m}\lvert f_{k}\rvert^{2}+\epsilon^{2}\big)\Big(dd^{c}\log\big(\sum_{k=1}^{m}\lvert f_{k}\rvert^{2}+\epsilon^{2}\big)\Big)^{j}\wedge \Big(dd^{c}\phi_{0}\Big)^{n-j}.
\end{multline}

The first term on the right hand side of (\ref{eqn:Calc}) is uniformly bounded in $\epsilon\in [0,1]$, for $\lambda>0,\ell\in\N$ fixed, since $\{\varphi_{\mathcal{I},\lambda,\epsilon}\}_{\epsilon\in [0,1]}$ is a continuous family of 
smooth functions on $\Omega\setminus \Omega_{\ell}$. 
We let $C_{\ell,\lambda}$ denote a uniform upper bound for this quantity.

The second term on the right hand side of (\ref{eqn:Calc}) is bounded from above by
\begin{multline*}
-\frac{1}{n+1}\sum_{j=0}^{n}\int_{\Omega_{\ell}} \lambda^{j+1} \log\big(\sum_{k=1}^{m}\lvert f_{k}\rvert^{2}+\epsilon^{2}\big)\Big(dd^{c}\log\big(\sum_{k=1}^{m}\lvert f_{k}\rvert^{2}+\epsilon^{2}\big)\Big)^{j}\wedge \Big(dd^{c}\phi_{0}\Big)^{n-j} \\
\leq \frac{\lambda^{n+1}}{n+1}\log\epsilon^{-2}\sum_{j=0}^{n}\int_{\Omega_{\ell}}\Big(dd^{c}\log\big(\sum_{k=1}^{m}\lvert f_{k}\rvert^{2}+\epsilon^{2}\big)\Big)^{j}\wedge \Big(dd^{c}\phi_{0}\Big)^{n-j} .
\end{multline*}
We set
$$
F_{\ell}(\epsilon):=\sum_{j=0}^{n}\int_{\Omega_{\ell}}\Big(dd^{c}\log\big(\sum_{k=1}^{m}\lvert f_{k}\rvert^{2}+\epsilon^{2}\big)\Big)^{j}\wedge \Big(dd^{c}\phi_{0}\Big)^{n-j}.
$$

Observe that for $j=0,\dots,n-1$
$$
\lim_{\ell\nearrow +\infty}\lim_{\epsilon \searrow 0}\int_{\Omega_{\ell}}\Big(dd^{c}\log\big(\sum_{k=1}^{m}\lvert f_{k}\rvert^{2}+\epsilon^{2}\big)\Big)^{j}\wedge \Big(dd^{c}\phi_{0}\Big)^{n-j}=0
$$
since  the ideal sheaf $\mathcal{I}$ generated by $f_{1},\dots,f_{m}$ is supported at one point,
while  
$$
\int_{\Omega_{\ell}}\Big(dd^{c}\log\big(\sum_{k=1}^{m}\lvert f_{k}\rvert^{2}+\epsilon^{2}\big)\Big)^{n}\to e_{\ell}
$$
as $\epsilon\searrow 0$, where $e_{\ell}\geq e(X,\mathcal{I})$ and $e_{\ell}\searrow \int_{\{p\}}\big(dd^c\varphi_{\mathcal{I},1}\big)^n$
as $\ell \nearrow +\infty$.

Proposition \ref{prop:lct} yields $\int_{\{p\}}\big(dd^c\varphi_{\mathcal{I},1}\big)^n=e(X,\mathcal{I})$, ending the proof.
\end{proof}

\subsection{The upper bound}

We are now ready for the proof of the following result.

\begin{thm} \label{thm:UB}
Let $(X,p)$ be a an isolated log terminal singularity. Then
$$
\gamma_{crit}\leq \frac{n+1}{n}\widehat{\mathrm{vol}}(X,p)^{1/n}.
$$
\end{thm}

\begin{proof}
Fix $\gamma<\gamma_{crit}$ and $C_{1}>0$ such that
\begin{equation}
    \label{eqn:MTg}
    \frac{1}{\gamma}\log\int_{\Omega}e^{-\gamma \varphi}d\mu_{p}\leq C_{1}-E_{\phi}(\varphi)
\end{equation}
for any $\varphi\in\mathcal{T}_{\phi}(\Omega)$.

Fix $\mathcal{I}$ coherent ideal sheaf supported at $p$, and let $\{\varphi_{\mathcal{I},\lambda,\epsilon}\}_{\lambda>0,\epsilon\in(0,1]}\in\mathcal{T}(\Omega)$ as defined in Lemma \ref{lem:Cons}. 
Evaluating (\ref{eqn:MTg}) at $\{\varphi_{\mathcal{I},\lambda,\epsilon}\}_{\epsilon\in(0,1]}$ yields
\begin{eqnarray*}
C_{\gamma,\lambda}+\Big(\lambda-\frac{\mathrm{lct}(X,\mathcal{I})}{\gamma}\Big)\log \epsilon^{-2}
& \leq & \frac{1}{\gamma}\log\int_{\Omega}e^{-\gamma\varphi_{\mathcal{I},\lambda,\epsilon}}d\mu_{p} \\
& \leq & C_{1}-E_{\phi}(\varphi_{\mathcal{I},\lambda,\epsilon}) \\
& \leq & C_{1}+ C_{N,\lambda}+\frac{\lambda^{n+1}}{n+1}F_{N}(\epsilon)\log \epsilon^{-2}
\end{eqnarray*}
for any $N\in\N,\epsilon\in (0,1]$ thanks to Lemmas \ref{lem:lct} and \ref{lem:Mult}. 
We infer
$$
\Big(\lambda-\frac{\mathrm{lct}(X,\mathcal{I})}{\gamma}-\frac{\lambda^{n+1}}{n+1}F_{N}(\epsilon)\Big)\log\epsilon^{-2}\leq C_{1}+C_{N,\lambda}-C_{\gamma,\lambda},
$$
hence
\begin{equation}
    \label{eqn:2}
    \lambda-\frac{\lambda^{n+1}}{n+1}e_{N}\leq \frac{\mathrm{lct}(X,\mathcal{I})}{\gamma}
\end{equation}
for any $N\in\N, \lambda>0$ since $F_{N}(\epsilon)\to e_{N}$ as $\epsilon\searrow 0$ (Lemma \ref{lem:Mult}).

The function $g_{N}: \lambda \in (0,+\infty)\mapsto  \lambda-\frac{\lambda^{n+1}}{n+1}e_{N} \in \R$ 
reaches its maximum at $\lambda_{N,M}:=1/e_{N}^{1/n}$. It follows therefore from (\ref{eqn:2}) that
$$
\gamma\leq \frac{\mathrm{lct}(X,\mathcal{I})}{g_{N}(\lambda_{N,M})}=\frac{n+1}{n}\mathrm{lct}(X,\mathcal{I})e_{N}^{1/n}.
$$
Now $e_{N}\searrow e(X,\mathcal{I})$ as $N\to +\infty$ by Lemma \ref{lem:Mult}, hence
$$
\gamma\leq \frac{n+1}{n}\mathrm{lct}(X,\mathcal{I})e(X,\mathcal{I})^{1/n}.
$$
Since this holds for any coherent ideal sheaf $\mathcal{I}$ supported at $p$ , we obtain
$$
\gamma\leq \frac{n+1}{n}\inf_{\mathcal{I}}\mathrm{lct}(X,\mathcal{I})e(X,\mathcal{I})^{1/n}=\frac{n+1}{n}\widehat{\mathrm{vol}}(X,p)^{1/n}
$$
where the equality follows from Theorem \ref{thm:NormVol}.
\end{proof}

\section{Moser-Trudinger inequality} \label{sec:minoralpha}

\subsection{Uniform integrability vs Moser-Trudinger inequality}

Recall that  
$$
    \alpha(X,\mu_p):=\sup\left\{\alpha>0,  \, 
    \sup_{\varphi\in\mathcal{F}_1(\Omega)}\int_{\Omega}e^{-\alpha \varphi}d\mu_{p}<+\infty\right\}
$$
This uniform integrability index is a local counterpart to Tian's celebrated
$\alpha$-invariant, introduced in \cite{Tian87} in the quest for K\"ahler-Einstein metrics
on Fano manifolds.
We refer to \cite{DK01,Dem09,Zer09,ACKPZ09,DP14,GZh15,Pham18} for some contributions
to the local study of analogous invariants.

\smallskip

In this section we prove Theorem A, which can be seen as a  local analog of \cite[Proposition 4.13]{BBEGZ}.

\begin{thm} \label{thm:MT}
One has 
$
\gamma_{crit}(X,p)\geq \frac{n+1}{n}\alpha(X,\mu_p).
$
\end{thm}

When $(X,p)$ is smooth then $\alpha(X,\mu_p)=n$ and this statement is equivalent
(after an appropriate rescaling) to
\cite[Theorem 1.5]{BB22}, \cite[Theorem 9]{GKY13}.

Together with Theorem \ref{thm:UB}, we would obtain the precise value 
$$
\gamma_{crit}(X,p) \stackrel{?}{=} \frac{n+1}{n} \widehat{\mathrm{vol}}(X,p)^{1/n}
$$
if we knew that $\alpha(X,\mu_p)=\widehat{\mathrm{vol}}(X,p)^{1/n}$.
We establish in Section \ref{sec:ubalpha} the bound
$\alpha(X,\mu_p) \leq \widehat{\mathrm{vol}}(X,p)^{1/n}$
and analyze the reverse inequality in Section \ref{sec:lbalpha}.

\subsubsection{Entropy}

We let $\mathcal{P}(\Omega)$ denote the set of probability measures on $\Omega$.
Given two   measures $\mu,\nu\in\mathcal{P}(\Omega)$, 
the \emph{relative entropy of} $\nu$ \emph{with respect to} $\mu$ is  
$$
H_{\mu}(\nu):=\int_{\Omega}\frac{d\nu}{d\mu}\log \frac{d\nu}{d\mu} d\mu=\int_{\Omega}\log \frac{d\nu}{d\mu} d\nu
$$
if $\nu$ is absolutely continuous with respect to $\mu$, and as $H_{\mu}(\nu):=+\infty$ otherwise.

Given  $\mu\in\mathcal{P}(X)$, the relative entropy $H_{\mu}(\cdot)$ is the Legendre transform of the convex 
functional $g \in \mathcal{C}^{0}(\Omega)\cap L^{\infty}(\Omega) \mapsto \log \int_{\Omega}e^{g}d\mu \in \R$, i.e.
$$
H_{\mu}(\nu)=\sup_{g\in\mathcal{C}^{0}(\Omega)\cap L^{\infty}(\Omega)}\Big(\int_{\Omega}g\,d\nu-\log\int_{\Omega}e^{g}d\mu\Big).
$$
We shall need the following duality result.
 
\begin{lem}\cite[Lemma 2.11]{BBEGZ}
\label{lem:2.11}
Fix $\mu\in\mathcal{P}(\Omega)$. Then   
$$
\log\int_{\Omega}e^{g}d\mu=\sup_{\nu\in\mathcal{P}(\Omega)}\Big(\int_{\Omega}g\,d\nu-H_{\mu}(\nu)\Big)
$$
for each lower semicontinuous function $g:\Omega\to \R\cup \{+\infty\}$.
\end{lem}

Recall that we have normalized the adapted volume form so that $\mu_{p}\in\mathcal{P}(\Omega)$.

\begin{coro}
\label{cor:Entropy}
Fix $0<\alpha<\alpha(X,\mu_p)$. Then there exists 
$C_{\alpha}>0$ such that
$$
H_{\mu_{p}}(\nu)\geq -\alpha\int_{\Omega}\varphi\, d\nu -C_{\alpha}
$$
for all $\varphi\in\mathcal{F}_1(\Omega)$ and for all $\nu\in\mathcal{P}(\Omega)$ 
such that $H_{\mu_{p}}(\nu)<+\infty$.
\end{coro}

\begin{proof}
This follows from Lemma \ref{lem:2.11} 
applied to $g=-\alpha \f$ and $\mu=\mu_p$.
By definition of $\alpha(X,\mu_p)$ we obtain
$-\log\int_{\Omega}e^{-\alpha \f}d\mu_p \geq -C_{\alpha}$.
\end{proof}

 This corollary shows in particular that $\mathcal{F}_{1}(\Omega)\subset L^{1}(\nu)$ 
 for any probability measure $\nu\in\mathcal{P}(X)$ with finite $\mu_p$-entropy. 
 Since the measure $\nu$ is moreover {\it non pluripolar}, the following result
 is a consequence of Theorem \ref{thm:Cegrell}.

\begin{prop}
\label{prop:Entropy}
Fix $\nu\in\mathcal{P}(\Omega)$ such that $H_{\mu_{p}}(\nu)<+\infty$. Then there exists 
a unique $v\in  \mathcal{F}_{1}(\Omega)$ such that
$$
\nu=(dd^{c}v)^{n}.
$$
\end{prop}

\subsubsection{Proof of Theorem \ref{thm:MT}}

Fix $\varphi\in \mathcal{T}_{\phi}(\Omega)$ and  $0<\alpha<\alpha(\Omega,\phi)$. 
By Lemma \ref{lem:2.11} for any $\epsilon>0$ there exists $\nu_{\epsilon}\in\mathcal{P}(\Omega)$ such that $H_{\mu_{p}}(\nu_{\epsilon})<+\infty$ and  
\begin{equation}
    \label{eqn:4}
    \log\int_{\Omega}e^{-\frac{n+1}{n}\alpha \varphi}d\mu_{p}\leq \epsilon -\frac{n+1}{n}\alpha\int_{\Omega}\varphi\,d\nu_{\epsilon}-H_{\mu_{p}}(\nu_{\epsilon}).
\end{equation}
  Proposition \ref{prop:Entropy} ensures the existence of
   $v_{\epsilon}\in \mathcal{F}_{1}(\Omega)$
    such that $\nu_{\epsilon}=(dd^{c}v_{\epsilon})^{n}$.
 It follows moreover from  Corollary \ref{cor:Entropy} that
\begin{equation}
    \label{eqn:5}
    H_{\mu_{p}}(\nu_{\epsilon})\geq -\alpha\int_{\Omega}v_{\epsilon}\, d\nu_{\epsilon}-C_{\alpha}.
\end{equation}
 Combining (\ref{eqn:4}) and (\ref{eqn:5}) we obtain
\begin{equation}
    \label{eqn:6}
    \log\int_{\Omega}e^{-\frac{n+1}{n}\alpha \varphi}d\mu_{p}
    \leq \epsilon+C_{\alpha}-\frac{n+1}{n}\alpha\int_{\Omega}\varphi\,d\nu_{\epsilon}
    +\alpha\int_{\Omega}v_{\epsilon}\, d\nu_{\epsilon}.
\end{equation}

We observe that
{\small
\begin{eqnarray*}
-\frac{n+1}{n}\alpha\int_{\Omega}\varphi\,d\nu_{\epsilon}+\alpha\int_{\Omega}v_{\epsilon}\,d\nu_{\epsilon}
&=& \frac{n+1}{n}\alpha \int_{\Omega}(v_{\epsilon}-\varphi) (dd^c v_{\epsilon})^n
-\frac{\alpha}{n} \int_{\Omega} v_{\epsilon} (dd^c v_{\epsilon})^n  \\
& \leq & -\frac{n+1}{n}\alpha E_{\phi}(\f)+
\frac{\alpha}{n} \left\{ (n+1) E_{\phi}(v_{\e})
-  \int_{\Omega} v_{\epsilon} (dd^c v_{\epsilon})^n \right\}.
\end{eqnarray*}
}
by using Lemma \ref{lem:Ener}
(the latter has been stated for functions in ${\mathcal T}_{\phi}(\Omega)$, it easily extends
to the class ${\mathcal F}_{1}(\Omega)$ by approximation).
Since $v_{\e} \leq \phi_0$ and $E_{\phi}(\phi_0)=0$, the same Lemma ensures  
$$
(n+1) E_{\phi}(v_{\e})-  \int_{\Omega} v_{\epsilon} (dd^c v_{\epsilon})^n
\leq -\int_{\Omega} \phi_0 (dd^c v_{\e})^n \leq -\inf_{\Omega} \phi_0,
$$
using that $\nu_{\epsilon}=(dd^c v_{\epsilon})^n$ is a probability measure.
Altogether this yields
$$
 \log \int_{\Omega}e^{-\frac{n+1}{n}\alpha\varphi}d\mu_{p}  
 \leq  \epsilon+C_{\alpha} -\frac{\alpha}{n} \inf_{\Omega} \phi_0
 -\frac{n+1}{n}\alpha E_{\phi}(\varphi).
$$
  Letting $\epsilon\searrow 0$ we conclude that
$$
\Big(\int_{\Omega}e^{-\frac{n+1}{n}\alpha\varphi}d\mu_{p}\Big)^{\frac{n}{(n+1)\alpha} }
\leq C_{\alpha}' e^{-E_{\phi}(\varphi)}
$$
for any function $\varphi\in\mathcal{T}_{\phi}(\Omega)$. Thus
$
\gamma_{crit}(X,p) \geq \frac{n+1}{n}\alpha(X,\mu_p).
$

 \subsection{Upper-bound on the $\alpha$-invariant} \label{sec:ubalpha}

\begin{defi}
We set 
$$
 \tilde{\alpha}(X,\mu_p):=\inf\left\{c(\varphi), \, 
    \varphi \in \mathcal{F}_1(\Omega)   \right\},
$$
where 
$c(\varphi):=\sup\Big\{c>0 ; \, \int_{\Omega}e^{-c\varphi}d\mu_{p}<+\infty\Big\}$.
\end{defi}

\subsubsection{Bounding the $\alpha$-invariant by the normalized volume}

\begin{prop} \label{pro:alphavsvol1}
One has $\alpha(X,\mu_p) \leq \tilde{\alpha}(X,\mu_p) \leq \widehat{\mathrm{vol}}(X,p)^{1/n}$.
\end{prop}

\begin{proof}
It follows from the definition that 
$\alpha(X,\mu_p) \leq \tilde{\alpha}(X,\mu_p)$.

For any $\epsilon>0$ and $\mathcal{I}$ coherent ideal sheaf supported at $0$, the function
$$
\psi_{\mathcal{I},\epsilon}:=\psi_{\mathcal{I},\lambda,\epsilon},
\; \text{ with } \; 
\lambda=\left(\frac{1-\epsilon}{e(X,\mathcal{I})}\right)^{1/n}
$$
given by Lemma \ref{lem:Cons2} below, belongs to $\mathcal{F}_{1}(\Omega)$ and yields
$$
\tilde{\alpha}(X,\mu_p)\leq c(\psi_{\mathcal{I},\epsilon})
=\frac{1}{(1-\epsilon)^{1/n}}\mathrm{lct}(X,\mathcal{I})e(X,\mathcal{I})^{1/n}.
$$
The latter equality is a consequence of Proposition \ref{prop:lct}.
We conclude the proof by taking the infimum over all  ${\mathcal I}$'s and letting $\epsilon\searrow 0$.
\end{proof}

\begin{lem}
\label{lem:Cons2}
Let $\mathcal{I}$ be a coherent ideal sheaf supported at $p$. Then, for any $\lambda, \epsilon>0$ there exists a function $\psi_{\mathcal{I},\lambda,\epsilon}\in \mathcal{F}(\Omega)$ 
s.t.
\begin{itemize}
    \item[i)] $\psi_{\mathcal{I},\lambda,\epsilon}=\lambda\log\Big(\sum_{j=1}^{m}\lvert f_{j} \rvert^{2}\Big)$ near $0$ for local generators $f_{1},\dots,f_{m}$ of $\mathcal{I}$;
    \item[ii)] $\lambda^{n}e(X,\mathcal{I}) \leq 
    \int_{\Omega}\big(dd^{c}\psi_{\mathcal{I},\lambda,\epsilon}\big)^{n}
    \leq \lambda^{n}e(X,\mathcal{I})+\epsilon$.
\end{itemize}
\end{lem}

\begin{proof}
Assume that $\phi_0$ is the maximal psh extension of $\phi$ to $\Omega$, i.e.
the largest psh function in $\Omega$ 
which lies below $\phi$ on $\partial \Omega$. It satisfies $(dd^c \phi_0)^n=0$ in $\Omega$.

Fix $f_{1},\dots,f_{m}$ local generators of the ideal $\mathcal{I}$ and set 
$\psi_{\lambda}:=\lambda\log \big(\sum_{j=1}^{m}\lvert f_{j} \rvert^{2}\big)$. 
We can assume without loss of generality that 
the $f_j$'s are well-defined in $\Omega$ and normalized so that
$\psi_{\lambda} \leq \phi_0-1$ in $\Omega$.
For $r>0$ we consider
$$
\f_{r}:=\sup \left\{ u \in PSH(\Omega), \; u \leq \psi_{\lambda} \text{ in } B(r)
\; \text{ and } \; u \leq \phi_0 \text{ in } \Omega \right\}.
$$
 The corresponding family of psh functions is non empty as it contains $\psi_{\lambda}$.
For $A>1$ large enough, the function
$$
w_r=\left\{ 
\begin{array}{ll}
\psi_{\lambda}  & \text{ in } B(r) \\
\max(\psi_{\lambda}, A \rho +\phi_0) & \text{ in } \Omega \setminus B(r)
\end{array}
\right.
$$
is psh and coincides with $A \rho+\phi_0$ near $\partial \Omega$.
It follows that  
\begin{itemize}
\item $\f_r \in PSH(\Omega)$ with $\f_r =\phi$ on $\partial \Omega$;
\item $\f_r \equiv \psi_{\lambda}$ in $B(r)$ hence 
$\lambda^{n}e(X,\mathcal{I}) \leq \int_{\Omega} (dd^{c} \f_r)^n$;
\item $(dd^{c} \f_r)^n=0$ in $\Omega \setminus \overline{B}(r)$ (balayage argument).
\end{itemize}

The family $r \mapsto \f_r$ increases -as $r>0$ decreases to $0$- to some psh limit $\f$
whose Monge-Amp\`ere measure $(dd^c \f)^n$ is concentrated at the origin.
It follows from Bedford-Taylor continuity theorem that 
$(dd^c \f)^n $ is the weak limit of $(dd^c \f_r)^n \geq \lambda^{n}e(X,\mathcal{I})  \delta_0$, hence
$(dd^c \f)^n \geq \lambda^{n}e(X,\mathcal{I})  \delta_0$.
Conversely $\psi_{\lambda} \leq \f$ near $0$, hence Demailly's comparison theorem ensures
that 
$$
(dd^c \f)^n(0) \leq (dd^c \psi_{\lambda})^n(0) \leq \lambda^{n}e(X,\mathcal{I}),
$$
whence equality.
Thus  $\phi_{\mathcal{I},\lambda,\epsilon}:=\f_{r_{\e}}$ satisfies the required properties.
\end{proof}

\subsubsection{Normalized volume vs uniform integrability}

\begin{prop} \label{pro:Skoda}
One has
$
\tilde{\alpha}(X,\mu_p)= \widehat{\mathrm{vol}}(X,p)^{1/n}.
$
\end{prop}

We refer the reader to the Appendix for a more algebraic approach based on \cite{BdFF}, which 
moreover provides a slightly stronger result.
 
When $(X,p)$ is mooth, it follows from \cite{DK01} that $\tilde{\alpha}(X,\mu_p)={\alpha}(X,\mu_p)$.
The situation is however more subtle in the singular context (see Section \ref{sec:resol}).

\begin{proof}
By Proposition \ref{pro:alphavsvol1}  it suffices to show that
 $\tilde{\alpha}(X,\mu_p)\geq \widehat{{\rm vol}}(X,p)^{1/n}$, i.e.
$    \int_\Omega e^{-\alpha \varphi}d\mu_p<+\infty$
for all $\varphi\in \mathcal{F}_1(\Omega)$ and  $\alpha< \widehat{\mathrm{vol}}(X,p)^{1/n}=\inf_{\mathcal{I}} {\rm lct}(X,\mathcal{I})^n e_p({\mathcal I})$.

In a log resolution $\pi:\Tilde{\Omega}\to \Omega$, this boils down to
{\small
$    \int_{\Tilde{\Omega}}e^{-\alpha \varphi\circ \pi}\prod_{i=1}^M \lvert s_i\rvert_{h_i}^{2a_i} dV<+\infty$,
}
where $s_i$ are holomorphic sections defining simple normal crossing exceptional divisors $E_1,\dots,E_M$, $K_{\tilde{\Omega}/\Omega}=\sum_{j=1}^M a_i E_i$ and where $dV$ is a smooth volume form. 
The log terminal condition ensures that $a_i>-1$ for all $i=1,\dots,M$.

As $\alpha<\widehat{\mathrm{vol}}(X,p)^{1/n}\leq n$, 
the integrability condition is equivalent to show that for any point $x\in \cup_{i=1}^M E_i$ there exists a small ball $B(x,r)$ such that
\begin{equation}
    \label{eqn:I3}
    \int_{B(x,r)}e^{-\alpha \varphi\circ \pi}\prod_{i=1}^M \lvert s_i\rvert_{h_i}^{2a_i} dV<+\infty.
\end{equation}

Set $U:=\sum_{i\, : \, a_i\geq 0} a_i\log \lvert s_i \rvert_{h_i}^2$,  $V:=\alpha \varphi\circ \pi$ and $W:=-\sum_{i\, :\, a_i<0}a_i\log \lvert s_i\rvert_{h_i}^2$. 
By \cite[Theorem B.5]{BBJ21} the
condition (\ref{eqn:I3}) holds iff
 there exists $\epsilon>0$ such that
\begin{equation}
    \label{eqn:2bis}
    \nu(U\circ g, F)+A_{\tilde{\Omega}}(F)\geq (1+\epsilon) \nu(V\circ g, F)+(1+\epsilon)\nu(W\circ g, F)
\end{equation}
for any $F$ prime divisor \emph{over} $\tilde{\Omega}$ with center in a small ball $\overline{B(x,r')}\subset B(x,r)$, i.e. $F\subset \Omega'$ for $g:\Omega'\to\tilde{\Omega}$ modification. 
Observe that
\begin{eqnarray*}
    \nu(U\circ g,F)+A_{\tilde{\Omega}}(F)-\nu(W\circ g, F)
    &=& \ord_F( g^*K_{\tilde{\Omega}/\Omega})+1+\ord_F( K_{\Omega'/\tilde{\Omega}}) \\
    &=& 1+\ord_F (K_{\Omega'/\Omega})=A_{\Omega}(F).
\end{eqnarray*}
Thus (\ref{eqn:2bis}) becomes
\begin{equation}
    \label{eqn:I4}
    \alpha(1+\epsilon)\leq \frac{A_{\Omega}(F)-\epsilon\nu(W\circ g,F)}{\nu(\varphi\circ \pi\circ g,F)}.
\end{equation}
As $a_i>-1$ for all $i$, \cite[Theorem B.5]{BBJ21} ensures the existence of $a>0$ such that
$
A_{\tilde{\Omega}}(F)\geq (1+a)\nu(W\circ g,F)
$
for any prime divisor $F$ over $\tilde{\Omega}$ as above. Thus
$$
A_{\tilde{\Omega}}(F)\leq A_\Omega(F)+\nu(W\circ g, F)\leq \frac{1}{1+a}A_{\tilde{\Omega}}(F)+A_\Omega(F),
$$
and
$
\nu(W\circ g,F)\leq \frac{1}{1+a}A_{\tilde{\Omega}}(F)\leq \frac{1}{a}A_\Omega(F).
$
Therefore (\ref{eqn:I4}) holds if
\begin{equation}
    \label{eqn:I5}
    \alpha(1+\epsilon)\leq \frac{a-\epsilon}{a}\frac{A_\Omega(F)}{\nu(\varphi\circ \pi\circ g,F)}.
\end{equation}

Since $\f \in {\mathcal F}_1(\Omega)$, it follows from the comparison theorem 
of Demailly  \cite[Theorem 4.2]{Dem85}
that for a coherent ideal sheaf $\mathcal{I}$ supported at $p\in \Omega$, 
\begin{equation}
    \label{eqn:BBB}
    1\geq \int_\Omega (dd^c \varphi)^n\geq \nu_\mathcal{I}(\varphi,p)^n\int_{\C^N}(dd^c f_\mathcal{I})^n\wedge [X]=\nu_\mathcal{I}(\varphi,p)^n e_p(\mathcal{I})
\end{equation}
where $f_\mathcal{I}= \log (\sum_i\lvert f_i\rvert^2)$ for generators $\{f_i\}_i$ of $\mathcal{I}$
and 
$$
\nu_\mathcal{I}(\varphi,p):=\sup\big\{s>0\,:\, \varphi\leq s f_\mathcal{I}+ O(1)\big\}
=\min_{G}\frac{\nu(\varphi\circ\pi\circ g,G)}{\ord_G \mathcal{I}},
$$
where $\pi\circ g: \Omega'\to \Omega$ is a log resolution for $\mathcal{I}$
and the minimum is over all exceptional divisors of $\Omega'\to \Omega$.
  Lemma \ref{lem:Lelongblowup} below ensures that
 for any prime divisor  $F$ and $\delta>0$ there exists an ideal $\mathcal{I}$ such that
\begin{eqnarray*}
    \frac{A_\Omega(F)}{\nu(\varphi\circ \pi\circ p,F)}
    &\geq& (1-\delta)\frac{A_\Omega(F)}{\ord_F\mathcal{I}}\big(\nu_\mathcal{I}(\varphi,p)\big)^{-1}
    \geq(1-\delta) \frac{A_\Omega(F)}{\ord_F\mathcal{I}} e_p(\mathcal{I})^{1/n} \\
    &\geq& (1-\delta){\rm lct}(X,\mathcal{I})e_p(\mathcal{I})^{1/n}
    \geq (1-\delta)\widehat{\mathrm{vol}}(X,p)^{1/n}.
\end{eqnarray*}
Thus (\ref{eqn:I5}) holds if
$ 
\alpha(1+\epsilon)\leq \frac{(a-\epsilon)(1-\delta)}{a}\widehat{\mathrm{vol}}(X,p)^{1/n} ,
$
concluding the proof.
 \end{proof}

 \begin{lemma} \label{lem:Lelongblowup}
Fix $\varphi\in \mathcal{F}_1(\Omega)$ and $F\subset \Omega'$ prime divisor such that $\pi\circ g(F)=p$. For any $\epsilon>0$ there exists a coherent ideal sheaf $\mathcal{I}$ supported at $p$ such that
$$
\nu_\mathcal{I}(\varphi,p)\geq(1-\epsilon)\frac{\nu(\varphi\circ \pi\circ g,F)}{\ord_F (\mathcal{I})}.
$$
\end{lemma}

\begin{proof}
Let $c:=\nu(\varphi\circ \pi\circ g,F)$ and for $c'\in \Q, c'\leq c$, set 
$$
\mathcal{A}_{mc'}(F):=\big\{f\in \mathcal{O}_{X,p}\, :\, \ord_F (f\circ \pi\circ g)\geq mc'\big\}
$$
 for $m\in \N$ divisible enough. Then $\mathcal{A}_{mc'}(F)$ is an ideal sheaf and 
\begin{equation}
    \label{eqn:Ro1}
    \limsup_{m\to +\infty} \frac{\ord_F (\mathcal{A}_{mc'}(F))}{m}=c'.
\end{equation}

In particular if $\varphi_{mc'}\in \PSH\big(B(p,r)\big)$ has algebraic singularities 
along $\mathcal{A}_{mc'}(F)$, then  for any $\epsilon>0$, $\varphi_{mc'}$ is less singular 
than $\frac{mc'}{c-\epsilon}\varphi$ around $p$ if $m\geq m_1(\epsilon)\gg 1$. 
For any $G$ exceptional divisor on $\Omega'$ and  $m\geq m_1(\epsilon)$ we infer
\begin{equation}
    \label{eqn:Ro2}
    \frac{\nu(\varphi\circ \pi\circ g,G)}{\ord_{G}(\mathcal{A}_{mc'}(F))/m}=\frac{\nu(\varphi\circ \pi\circ g,G)}{\nu(\varphi_{mc'}\circ \pi\circ g,G)/m}\geq \frac{c-\epsilon}{c'}.
\end{equation}

On the other hand, (\ref{eqn:Ro1}) implies that there exists $m_0(\epsilon)\geq m_1(\epsilon)\gg 1$ with
\begin{equation}
    \label{eqn:Ro3}
    \frac{\nu(\varphi\circ \pi\circ g,F)}{\ord_F(\mathcal{A}_{m_0c'}(F))/m_0}\leq \frac{c}{c'-\epsilon}.
\end{equation}

Combining (\ref{eqn:Ro2}) and (\ref{eqn:Ro3}) we obtain
\begin{eqnarray*}
    \min_{G}\frac{\nu(\varphi\circ \pi\circ g, G)}{\ord_G(\mathcal{A}_{m_0c'}(F))/m}
    \geq \frac{c-\epsilon}{c'}
    &\geq &  \Big(1-\epsilon\frac{c+c'}{cc'}\Big)\frac{c}{c'-\epsilon} \\
    &\geq & \Big(1-\epsilon\frac{c+c'}{cc'}\Big) 
    \frac{\nu(\varphi\circ \pi\circ g,F)}{\ord_F(\mathcal{A}_{m_0c'}(F))/m}.
\end{eqnarray*}

Since $c'$ and $\epsilon$ are arbitrary, and $x \mapsto f(x)=\frac{c+x}{cx}$ is decreasing, we deduce that for any $\epsilon>0$ there exists $c'\in\Q$ and $m_0=m_0(c,c',\epsilon)$ 
such that
$$
\nu_{\mathcal{A}_{m_0c'}(F)}(\varphi,p)\geq (1-\epsilon)\frac{\nu(\varphi\circ \pi\circ g,F)}{\ord_F(\mathcal{A}_{m_0c'}(F))}.
$$
Setting $\mathcal{A}:=\mathcal{A}_{m_0c'}(F)$  concludes the proof.
\end{proof}

\subsection{Lowerbounds on the $\alpha$-invariant} \label{sec:lbalpha}

We provide in this section two effective (but not sharp) lower bounds on $\alpha(X,\mu_p)$.

\subsubsection{Using projections on $n$-planes}

A celebrated result of Skoda ensures that $e^{- \f}$ is integrable
if the Lelong numbers of $ \f$ are small enough (see \cite[Theorem 2.50]{GZbook}).
This result has been largely extended by Demailly  
and Zeriahi  who provided uniform integrability
results for functions $\f \in {\mathcal F}_1(\Omega)$ 
\cite{Dem09,ACKPZ09}.

\smallskip

In this section we partially extend these results to our singular setting.

\begin{thm} \label{thm:lowerboundalpha}
One has 
$\alpha(X,\mu_p) \geq \frac{n}{{\rm mult}(X,p)^{1-1/n}} \frac{{\rm lct}(X,p)}{1+{\rm lct}(X,p)}$.
\end{thm}

\begin{proof}
Recall that $\mu_p=fdV_X$ with $f \in L^r(dV_X)$. The exponent $r>1$ has been estimated in Lemma \ref{lem:integrabilityf}. Using H\"older inequality, we thus obtain
$$
\alpha(X,\mu_p) \geq \frac{{\rm lct}(X,p)}{1+{\rm lct}(X,p)} \alpha(\Omega,dV_X).
$$

The remainder of the proof thus consists in establishing the lower bound
$$
\alpha(\Omega,dV_X) \geq \frac{n}{{\rm mult}(X,p)^{1-1/n}}.
$$
Recall that $dV_X=\omega_{eucl}^n \wedge [X]$, where $\omega_{eucl}$ denotes the euclidean 
K\"ahler form. Thus $dV_X=\sum_I (\pi_I)^*(dV_I)$, where $I=(i_1,\ldots,i_n)$ is a n-tuple,
$\pi_I:\C^N \rightarrow \C_I^n$ denotes the linear projection on   $\C_I^n$,
and $dV_I$ is the euclidean volume form on $\C_I^n$.

We choose coordinates in $\C^N$ so that each projection map
$\pi_I: \Omega \rightarrow \Omega_I \subset \C^n$ is proper. For $\f \in {\mathcal F}_1(\Omega)$
we obtain
$$
\int_{\Omega} e^{-\alpha \f} dV_X = \sum_I \int_{\Omega_I} (\pi_I)_* (e^{-\alpha \f}) dV_I
\leq {\rm mult}(X,p) \sum_I \int_{\Omega_I} e^{-\alpha  (\pi_I)_* \f} dV_I.
$$

We assume here -without loss of generality- that $\f \leq 0$, and use the (sub-optimal) inequality
$(\pi_I)_* (e^{-\alpha \f}) \leq {\rm mult}(X,p) e^{-\alpha  (\pi_I)_* \f}$.
The function $\f_I:=(\pi_I)_* \f$ is psh in $\Omega_I=\pi_I(\Omega)$, with boundary values
$(\pi_I)_*(\phi)$. We claim that  
\begin{equation} \label{eq:massesMA}
\int_{\Omega_I} (dd^c \f_I)^n \leq {\rm mult}(X,p)^{n-1}.
\end{equation}
Once this is established, it follows from the main result of \cite{ACKPZ09} that
for all $0<\e$ small enough, there exists $C_{\e}>0$ independent of $\f$ such that
$$
\int_{\Omega_I} e^{- \frac{n-\e}{{\rm mult}(X,p)^{1-1/n}}  \f_I} dV_I \leq C_{\e},
$$
which yields the desired lower bound
$\alpha(\Omega,dV_X) \geq \frac{n}{{\rm mult}(X,p)^{1-1/n}}$.

\smallskip

It remains to check \eqref{eq:massesMA}. We decompose $\f_I(z)=\sum_{i=1}^m \f(x_i)$,
where $m={\rm mult}(X,p)$ and $x_1,\ldots,x_m$ denote the preimages of $z$ counted
with multiplicities. The assumption on the Monge-Amp\`ere mass of $\f$ reads
$$
\sum_{i=1}^m \int (dd^c \f)^n(x_i) \leq 1.
$$
We set $a_i^n:=\int (dd^c \f)^n(x_i)$ and use \cite[Corollary 5.6]{Ceg04} to estimate
\begin{eqnarray*}
\int (dd^c \f_I)^n
&=& \sum_{i_1,\ldots,i_n=1}^m \int dd^c \f(x_{i_1}) \wedge \cdots \wedge dd^c \f (x_{i_n}) \\
&\leq & \sum_{i_1,\ldots,i_n=1}^m a_{i_1} \cdots a_{i_n}=\left( \sum_{i}^m  a_i \right)^n.
\end{eqnarray*}
The latter sum is maximized when $a_1=\cdots=a_m=m^{-1/n}$, yielding \eqref{eq:massesMA}.
\end{proof}

\begin{example}
Let $X=\{z \in \C^{n+1}, \; F(z)=0\}$ be the $A_k$-singularity,
where $F(z)=z_0^{k+1}+z_1^2+\cdots z_n^2$. Arguing as we have done for the ODP $(k=1)$, 
one can check that 
$\mu_p \sim \frac{dV_X}{||F'||^2}$
so that ${\rm mult}(X,p)=2$ and ${\rm lct}(X,p)=n-2+\frac{2}{k+1}$. Now
\begin{equation*}
    \label{eqn:LB2}
    \widehat{\mathrm{vol}}(A_k,p)^{1/n}=
    \begin{cases}
    2^{1/n}\big(\frac{n-2}{n-1}\big)^{1-1/n}n & \mathrm{if}\,\,\, \frac{k+1}{2}\geq \frac{n-1}{n-2},\\
    (k+1)^{1/n}\big(\frac{(n-2)(k+1)+2}{k+1}\big) & \mathrm{if}\,\,\, \frac{k+1}{2}<\frac{n-1}{n-2},
    \end{cases}
\end{equation*}
as computed by C.Li in \cite[Example 5.3]{Li18}. For $n>>1$, the lower bound provided by Theorem 
\ref{thm:lowerboundalpha} is thus short of a factor $2={\rm mult}(X,p)$
by comparison with the expected lower bound $\widehat{\mathrm{vol}}(A_k,p)^{1/n}$.
\end{example}

\subsubsection{Using resolutions} \label{sec:resol}

\begin{prop} \label{pro:lowerboundalpha2}
\label{prop:Alpha_Resolution}
Let $\pi:\tilde{\Omega}\to \Omega$ be a resolution of singularities with simple normal crossing, and let $\{a_i\}_{i=1,\dots,M}$ be the discrepancies. Then
$$
\alpha(X,\mu_p)\geq \frac{\widehat{\mathrm{vol}}(X,p)^{1/n}}{1+\big(\max_i a_i\big)_+}.
$$

In particular if the singularity is "admissible"  then 
$\alpha(X,\mu_p)=\widehat{\mathrm{vol}}(X,p)^{1/n}$.
\end{prop}

Following \cite[Definition 1.1]{LTW21} we say here that $(X,p)$ is an admissible singularity if
there exists a resolution $\pi:\tilde{X} \rightarrow X$ (with snc exceptional divisor $E=\sum_j E_j$ and $\pi$- ample divisor $-\sum b_j E_j$, $b_j \in \Q^+$)
such that the discrepancies $a_i \in (-1,0]$ are all non-positive. Recall that 
\begin{itemize}
\item any $2$-dimensional log terminal   singularity is admissible;
\item the vertex of the cone over a smooth Fano manifold is admissible;
\item $(X,p)$ is  admissible if it is $\Q$-factorial and admits a  crepant resolution.
\end{itemize}

 Theorem B from the introduction follows from the combination of 
 Proposition \ref{pro:alphavsvol1}, Theorem \ref{thm:lowerboundalpha} and Proposition \ref{pro:lowerboundalpha2}.

\begin{proof}
 We seek $\alpha>0$ such that
\begin{equation}
    \label{eqn:Request}
    \sup_{\varphi\in \mathcal{F}_1(\Omega)}\int_{\tilde{\Omega}} e^{-\alpha \varphi\circ \pi}\prod_{i=1}^M \lvert s_i\rvert_{h_i}^{2a_i} dV<+\infty. 
\end{equation}

If all the  $a_i$'s are non-positive we can use \cite[Main Theorem]{DK01} 
to show that $\alpha(X,\mu_p)=\Tilde{\alpha}(X,\mu_p)$, hence 
$\alpha(X,\mu_p)=\widehat{\mathrm{vol}}(X,p)^{1/n}$ by Proposition \ref{pro:Skoda}
(this follows from a simple contradiction argument).

In general  we set $U:=\sum_{i:a_i>0}a_i\log \lvert s_i\rvert^2_{h_i}$
and $W:=-\sum_{i:a_i\leq 0} a_i \log \lvert s_i \rvert^2_{h_i}$. 
Using \cite[Main Theorem]{DK01} we obtain
\begin{equation}
    \label{eqn:Alpha_Resolution}
    \alpha(X,\mu_p)\geq \inf_{\varphi\in \mathcal{F}_1(\Omega)}c_W(\varphi\circ \pi).
\end{equation}
where
$$
c_W(\varphi\circ\pi):=\sup\big\{\alpha>0\, : \, \int_{\tilde{\Omega}}e^{-\alpha \varphi\circ \pi- W}dV<+\infty\big\}.
$$
is the twisted complex singularity exponent. 
It then remains to estimate $c_W(\varphi\circ\pi)$ for a fixed $\varphi\in \mathcal{F}_1(\Omega)$. 
As $\pi^*d\mu_p=e^{U-W}dV$, H\"older inequality yields
\begin{equation}
    \label{eqn:Holder}
    \int_{\tilde{\Omega}}e^{-\alpha \varphi\circ \pi-W}dV\leq \Big(\int_{\tilde{\Omega}}e^{-p'\alpha \varphi\circ \pi}\pi^{*}d\mu_p\Big)^{1/p'}\Big(\int_{\tilde{\Omega}}e^{(1-q')U-W}dV\Big)^{1/q'}.
\end{equation}
Set $A:=(\max_i a_i)_+>0$. The second factor on the right-hand side of (\ref{eqn:Holder}) is finite for any $q'<\frac{A+1}{A}$, while the first factor on the right-hand side gives the condition $p'\alpha< \tilde{\alpha}=\widehat{\mathrm{vol}}(X,p)^{1/n}$. We infer
$
c_W(\varphi\circ \pi)\geq \frac{\widehat{\mathrm{vol}}(X,p)^{1/n}}{1+A},
$
which concludes the proof.
\end{proof}

As the proof shows, the main obstruction to proving the equality
$\alpha(X,\mu_p)=\Tilde{\alpha}(X,\mu_p)=\widehat{\mathrm{vol}}(X,p)^{1/n}$
is the lack of a Demailly-Kollàr result on complex spaces. 
Resolving the singularities, one ends up with a twisted version of Demailly-Kollàr's problem
on a smooth manifold. It is known that the general form of such a problem 
has a negative answer (see \cite[Remark 1.3]{Pham14}).

\section{Ricci inverse iteration}
\label{sec:RicciIteration}

In this section we prove Theorem C from the introduction.
The strategy is similar to that of \cite[Theorem 1]{GKY13}, with
a singular twist.

We fix $\gamma < \gamma_{crit}(X,p)$ and consider, for $j \in \N$,
the sequence of functions $\f_j \in PSH(\Omega)$ defined
by induction as follows: pick $\f_0 \in {\mathcal T}_{\phi}^{\infty}(\Omega)$ 
a smooth initial data, and let  
$\f_{j+1} \in PSH(\Omega) \cap {\mathcal C}^0(\overline{\Omega}) \cap {\mathcal C}^{\infty}(\overline{\Omega} \setminus \{p\}) $ be the unique solution to
$$
(dd^c \f_{j+1})^n=\frac{e^{-\gamma \f_j} \mu_p}{\int_{\Omega} e^{-\gamma \f_j} \mu_p}
$$
with boundary values ${\f_{j+1}}_{|\partial \Omega}=\phi$.
The existence and regularity of $\f_j$ off the singular locus follows from 
\cite[Theorem 1.4]{Fu21}, while the continuity of $\f_j$ near $p$ is a consequence of 
\cite[Theorem A]{GGZ23}.

We are going to establish uniform a priori estimates  on arbitrary derivatives of  the $\f_j$'s
in $\overline{\Omega} \setminus \{p\}$, thus $(\f_j)$ admits "smooth" cluster values.
We will show that the functional $F_{\gamma}$ is constant on 
the set ${\mathcal K}$ of these cluster points, so that any such   $\p$
is a solution of the Monge-Amp\`ere equation
$$
(dd^c \p)^n=\frac{e^{-\gamma \p} \mu_p}{\int_{\Omega} e^{-\gamma \p} \mu_p}
$$
with boundary values ${\p}_{|\partial \Omega}=\phi$.

\subsection{Uniform estimates}

\begin{prop} \label{prop:C0estimate}
There exists $C_0>0$ such that $||\f_j||_{L^{\infty}(\Omega)} \leq C_0$ for all $j \in \N$.
\end{prop}

This uniform estimate relies crucially on a technique introduced by Kolodziej in \cite{Kol98},
which has been extended to this singular setting in \cite{GGZ23}.

\begin{proof}
We assume without loss of generality that $\phi_0$ is the maximal psh extension of $\phi$ in $\Omega$.
In particular $\f_j \leq \phi_0$ for all $j \in \N$, and 
$E_{\phi}(\f_j) \leq E_{\phi}(\phi_0)=0$.
Our task is to establish a uniform lower bound $\f_j \geq -C_0$.

The assumption $\gamma < \gamma_{crit}(X,p)$ ensures, by Lemma \ref{lem:coercive}, that 
the functional $F_{\gamma}$ is coercive, in particular there exist  
$0 <a<1$ and $0<b$ such that 
$$
F_{\gamma}(\f_j) \leq a E_{\phi}(\f_j)+b
$$
for all $j \in \N$. 
It follows from \cite[Proposition 12]{GKY13} (exactly the same proof applies here)
that $j \mapsto F_{\gamma}(\f_j)$ is increasing, hence
$$
F_{\gamma}(\f_0) \leq F_{\gamma}(\f_j) \leq a E_{\phi}(\f_j) +b \leq b,
$$
showing that the energies $(E_{\phi}(\f_j))$ are uniformly bounded,
$-b' \leq E_{\phi}(\f_j) \leq 0$.

 The corresponding  family ${\mathcal G}_{b'}$
 of psh functions with $\phi$-boundary values and
   energy bounded by $b'$ is compact, and all its members have zero Lelong number 
   at all points in $\Omega$ (see Theorem \ref{thm:Cegrell}). 
   Passing through a resolution, one can thus  invoke
   Skoda's uniform integrability theorem \cite[Theorem 2.50]{GZbook} to conclude that
 the densities  $e^{-\gamma \f_j}$ are uniformly in $L^r(dV_X)$ for any $r>1$.
 
 Now $\mu_p=fdV_X$ with $f \in L^{1+\e}$ for some $\e>0$ since $(X,p)$ is log-terminal.
  H\"older inequality thus ensures
that the densities $g_j:=\frac{e^{-\gamma \f_j} f}{\int_{\Omega} e^{-\gamma \f_j} d\mu_p}$
are uniformly in $L^{1+\e'}(dV_X)$ for some $0<\e'<\e$. 

It therefore follows from \cite[Proposition 1.8]{GGZ23} (an extension of the main result of
\cite{Kol98} to the setting of pseudoconvex subsets of a singular complex space) that 
the $\f_j$'s are uniformly bounded.
\end{proof}

\subsection{${\mathcal C}^2$-estimates}

In this section we establish the following a priori estimates.

\begin{prop} \label{prop:C2estimate}
For all  compact subset $K$ of $\overline{\Omega} \setminus \{p\}$, 
there exists a constant  $C_2(K)>0$ such that for all $j \in \N$,
$$
 0 \leq \sup_K   \Delta_{\omega_X} \f_j   \leq C_2(K).
$$
\end{prop}

Here $\Delta_{\omega_X}h:=n \frac{dd^c h \wedge \omega_X^{n-1}}{\omega_X^n}$ denotes the 
Laplace operator with respect to the K\"ahler form $\omega_X$.
Such an estimate goes back to the  regularity theory developed in \cite{CKNS85}.
The strategy of the proof is similar to that of \cite[Theorem 15]{GKY13}, with a twist
due to the presence of the singular point $p$.

\begin{proof}
To obtain these estimates, one considers a resolution of the singularity  
$\pi:\tilde{\Omega} \rightarrow \Omega$. We let 
$E=\cup_{\ell=1}^m E_{\ell}$ denote the exceptional
divisor and let
\begin{itemize}
\item $s_{\ell}$ denote a holomorphic section of ${\mathcal O}(E_{\ell})$ such that $E_{\ell}=(s_{\ell}=0)$;
\item $b_{\ell}$ be positive rational numbers such that $-\sum_{\ell} b_{\ell}E_{\ell}$ is $\pi$-ample;
\item $h_{\ell}$ denote a smooth hermitian metric of ${\mathcal O}(E_{\ell})$
and $K>>1$  such that
$$
\beta:=K dd^c \rho \circ \pi- \sum_{\ell=1}^m b_{\ell} \Theta_{h_{\ell}}
\; \text{ is a K\"ahler form on } \tilde{\Omega}.
$$
\end{itemize} 
Observe that  the function $\rho':=K \rho \circ \pi+\sum_{\ell=1}^m b_{\ell} \log|s_{\ell}|_{h_{\ell}}^2$ is strictly 
psh   in $\tilde{\Omega}$, with $dd^c \rho' \geq \beta$ and $\rho'(z) \rightarrow -\infty$
as $z \rightarrow E$.

\smallskip

Recall that $\pi^* \mu_p=\Pi_{\ell=1}^m |s_{\ell}|^{2a_{\ell}} dV_{\tilde{\Omega}}$ with $a_{\ell}>-1$,
and set $|s|^2=\Pi_{\ell=1}^m |s_{\ell}|^{2b_{\ell}}$.
We are going to show that
there exist uniform constants $C_2>0, m \in \N$ such that 
\begin{equation} \label{eq:c2estimate}	
 0 \leq |s|^{2m}  |\Delta_{\beta} \f_{j}| (z)   \leq C_2
\end{equation}
for all $j \in \N, z \in \overline{\Omega}$, from which Proposition \ref{prop:C2estimate} follows.
Slightly abusing notation, we  still denote here by $\f_j$ the function $\f_j \circ \pi$.

We  approximate $\f_j$ by  the  smooth solutions $\f_{j,\e}$ of the Dirichlet problem
\begin{equation}
    \label{eqn:CMA}
    \begin{cases}
        (\e \beta+ dd^c\varphi_{j+1,\e})^n=\frac{e^{-\gamma \varphi_{j,\e}} \prod_{l=1}^m (\lvert s_l \rvert_{h_l}^2+\e^2)^{a_l}}{c_j} dV_{\tilde{\Omega}}\\
        \varphi_{j+1,\e|\partial \tilde{\Omega}}=\phi
    \end{cases}
\end{equation}
with $\f_{0,\e}=\f_0$ and  $c_j=\int_{\Omega} e^{-\gamma \f_j} d\mu_p$.
We are going to establish a priori estimates on these smooth approximants,
whose existence is guaranteed by \cite[Theorem 1.1]{GL10}.
We then show that $\f_{j,\e}$ converges to $\f_j$ as $\e$ decreases to zero.

\smallskip

\noindent {\it Step 1.}
We first claim that for all $j,\e$,
\begin{equation} \label{eq:c2estimate1}
\sup_{\partial \tilde{\Omega}} |\nabla \f_{j+1,\e}| \leq A_{1,j,\e},
\end{equation}
where $A_{1,j,\e}>0$ only depends on an upper-bound on $||\f_{j,\e}||_{L^{\infty}(\tilde{\Omega})}$.

Let $\Phi^-$ be a smooth psh extension of $-\phi$ to a neighborhood of $\overline{\Omega}$.
Observe that $\f_{j+1,\e}+\Phi^-\circ \pi$ is $\beta$-psh in $\tilde{\Omega}$,
with zero boundary values. 
Thus $\f_{j+1,\e}+\Phi^-\circ \pi \leq u$, where $u$ is the smooth solution in $\tilde{\Omega}$
to the Laplace equation $\Delta_{\beta} u=-n$ with zero boundary values.
We infer $\f_{j+1,\e} \leq \p_1:=u-\Phi^-\circ \pi$ in $\tilde{\Omega}$.

We now construct a psh function $\p_2 \leq \f_{j+1,\e}$  
with $\phi$-boundary values and such that $\sup_{\partial \tilde{\Omega}} |\p_2|$
is controlled from above by $||\f_{j,\e}||_{L^{\infty}(\tilde{\Omega})}$.
The upper bound on $\sup_{\partial \tilde{\Omega}} |\nabla \f_{j+1,\e}|$
thus follows from the inequalities $\p_2 \leq \f_{j,\e} \leq \p_1$.

Recall that $\pi^* \mu_p =\Pi_{\ell=1}^m |s_{\ell}|^{2a_{\ell}} dV_{\tilde{\Omega}}$.
We let $P \subset [1,m]$ denote the subset of indices such that $-1<a_{\ell} <0$.
For $\delta>0$ small enough, we observe that $v:=\rho'+\delta \sum_{\ell \in P} |s_{\ell}|^{2\delta}$
is strictly psh in $\tilde{\Omega}$ and satisfies, in $\tilde{\Omega} \setminus E$,
$$
dd^c v \geq c\left\{ \beta+\sum_{\ell \in P} 
\frac{i ds_{\ell} \wedge d\overline{s_{\ell}}}{|s_{\ell}|^{2(1-\delta})} \right\}
$$
for some $c>0$, hence $(dd^c v)^n \geq c' \pi^* \mu_p$.
Replacing $v$ by $\lambda_{j,\e} v$, we obtain
$$
(\e \beta+dd^c \lambda_{j,\e} v)^n \geq \lambda_{j,\e}^n (dd^c v)^n \geq
\frac{e^{-\gamma \varphi_{j,\e}} \prod_{l=1}^m (\lvert s_l \rvert_{h_l}^2+\e^2)^{a_l}}{c_j} dV_{\tilde{\Omega}},
$$ 
for some $\lambda_{j,\e}>0$ which only depends on an upper-bound 
on $||\f_{j,\e}||_{L^{\infty}(\tilde{\Omega})}$.
In other words $\lambda_{j,\e} v$ is a subsolution to the Monge-Amp\`ere equation
 in $\tilde{\Omega} \setminus E$.

We  modify $\lambda_{j,\e} v$ near  $\partial \tilde{\Omega}$ to produce a subsolution
with the right boundary values. Let $\chi$ be a cut-off function which is   $1$ 
near $E$ and has compact support in $\tilde{\Omega}$. The function
$\p_2=\chi \lambda_{j,\e} v+(1-\chi) \phi_0+A \rho \circ \pi$
satisfies all our requirements for $A>0$ large enough. 
Note however that it is only locally bounded in $\tilde{\Omega} \setminus E$.

Consider finally $\max(\p_2,\f_{j,\e})$.
This is a subsolution of the Dirichlet problem which is globally bounded in $\tilde{\Omega}$.
It follows  from the maximum principle that $\max(\p_2,\f_{j,\e}) \leq \f_{j,\e}$, hence
$\p_2  \leq \max(\p_2,\f_{j,\e}) \leq \f_{j,\e}$.

\medskip

 \noindent {\it Step 2.}
We next claim that there exist constants $A_2, A_{3,j+1,\e}>0$ such that 
\begin{equation} \label{eq:c2estimate2}
\sup_{\tilde{\Omega}}[\lvert s\rvert^{2A_2}_h\lvert\nabla \f_{j+1,\e}\rvert^2_{\beta}] \leq A_{3,j+1,\e}
\end{equation}
where $ A_{3,j+1,\e}$ only depends on an upper-bound on 
$\lVert \f_{k,\e} \rVert_{L^\infty(\tilde{\Omega})}$ for $k\leq j+1$.

\begin{proof}
The proof is a variant of \cite[Proposition 2.2]{DFS21}, which itself relies on previous estimates due
to Blocki and Phong-Sturm.

As we work in $\tilde{\Omega}\setminus\mathrm{Supp}(E)$, we identify $\beta$ with
$dd^c \big( K\rho\circ \pi +\log \lvert s\rvert^2_h\big)$. 
Replacing $\varphi_{j+1,\epsilon}$ by 
$\tilde{\varphi}_{j+1,\epsilon}:=\varphi_{j+1,\epsilon}-(K\rho \circ \pi+\log\lvert s \rvert^2_h)$ the 
 equation (\ref{eqn:CMA}) becomes
\begin{equation}
    \begin{cases}
        \big((1+\epsilon)\beta+ dd^c \tilde{\varphi}_{j+1,\epsilon}\big)^n
        =c_j^{-1}e^{-\gamma \varphi_{j,\epsilon}}\prod_{l=1}^m (\lvert s_l\rvert^2_{h_l}+\epsilon^2)^{a_l}dV_{\tilde{\Omega}},\\
        \tilde{\varphi}_{j+1,\epsilon|\partial \tilde{\Omega}}=\phi-\log\lvert s \rvert^2_h.
    \end{cases}
\end{equation}
As
$$
\Big\lvert\lvert\nabla\varphi_{j,\epsilon}\rvert_\beta -\lvert\nabla \tilde{\varphi}_{j,\epsilon}\rvert_\beta\Big\rvert\leq \frac{\lvert \nabla \lvert s \rvert^2_h\rvert_\beta}{\lvert s\rvert_h}
+C,
$$
to get the estimate (\ref{eq:c2estimate2}) for $\tilde{\varphi}_{j,\epsilon}$ it is enough to prove by induction that there exists positive constants $B_2, B_{3,j+1,\e}$ such that
\begin{equation}
    \label{eqn:Induction}
    \sup_{\tilde{\Omega}} [\lvert s\rvert_h^{2B_2}\lvert \nabla \tilde{\f}_{j+1,\e}\rvert^2_\beta]\leq \max\Big\{ \sup_{\tilde{\Omega}}[\lvert s \rvert_h^{2B_2}\lvert \nabla \tilde{\f}_{j,\e}\rvert_\beta^2], B_{3,j+1,\e}\Big\} 
\end{equation}
where $B_2$ is uniform in $j,\e$ while $ B_{3,j+1,\e}$ only depends on upper bounds on $\lVert \f_{j+1,\e}\rVert_{L^\infty(\tilde{\Omega})}, \lVert \f_{j,\e} \rVert_{L^\infty(\tilde{\Omega})}$, and where
 $\tilde{\f}_{0,\e}:=-\big(K\rho\circ \pi + \log \lvert s \rvert^2_h\big)$.
To lighten notations  we rewrite the equation 
\begin{equation}
    \label{eqn:MA_Gradient}
    \begin{cases}
        \big(\beta_\epsilon+ dd^c u\big)^n=e^{-v}\prod_{l=1}^m (\lvert s_l\rvert^2_{h_l}+\epsilon^2)^{a_l}\beta_\epsilon^n\\
        u_{|\partial \tilde{\Omega}}=\tilde{\phi}
    \end{cases}
\end{equation}
where $\beta_\epsilon:=(1+\epsilon)\beta$ is a non-degenerate smooth family of K\"ahler forms. Note that (\ref{eqn:Induction}) becomes
\begin{equation}
    \label{eqn:Induction2}
    \sup_{\tilde{\Omega}} [\lvert s\rvert_h^{2B_2}\lvert \nabla u\rvert^2_\beta]\leq \max\Big\{ \sup_{\tilde{\Omega}}[\lvert s \rvert_h^{2B_2}\lvert \nabla(v/\gamma-\log\lvert s \rvert_h^2-f_\epsilon)\rvert_\beta^2], B_{3,j+1,\e}\Big\} 
\end{equation}
where $\{f_\epsilon\}_{\epsilon>0}$ is a non-degenerate smooth family. In the estimates that follows we indicate with $C_i$ all the constants \emph{under control}, i.e. that depend on a upper bounds on $\lVert \varphi_{j+1,\epsilon}\rVert_{L^\infty(\tilde{\Omega})}, \lVert \varphi_{j,\epsilon} \rVert_{L^\infty(\Omega)}$. Observe that $\lVert u+\log\lvert s\rvert^2_h\rVert_{L^\infty(\tilde{\Omega})}, \lVert v\rVert_{L^\infty(\tilde{\Omega})}$ and $\sup_{\partial\tilde{\Omega}}\lvert \nabla u \rvert$ are under control. The constant $B_{3,j+1,\e}$ in (\ref{eqn:Induction2}) will clearly depend on the $C_i$'s.
We indicate with $D_i$ all the constants uniform in $j,\e$, which will be used to determine the uniform constant $B_2$ in (\ref{eqn:Induction2}).


We denote by $\Delta_\epsilon,\Delta'_\epsilon$ respectively the Laplacian operators with respect to $\beta_\epsilon$ and to $\eta_{\epsilon}:=\beta_\epsilon + dd^c u$.
Consider
$$
H:=\log \lvert \nabla u\rvert_{\beta_\epsilon}^2+\log \lvert s \rvert_h^{2k}-G(u)
$$
where $G(x)=Ax-\frac{B}{x+C+1}$ for $C$ chosen so that $u\geq -C$, while $A>0, B>0$ to be determined later. The constants $A,k$ are chosen to be uniform in $j,\e$ while $ B$ is under control. If $H$ reaches its maximum at $x_M$, then
\begin{equation}
    \label{eqn:Maximum}
    \lvert \nabla u \rvert_{\beta_\epsilon}^2\lvert s\rvert_h^{2(k+A)}\leq
    C_1\big( \lvert \nabla u \rvert_{\beta_\epsilon}^2\lvert s\rvert_h^{2(k+A)}\big)(x_M)
\end{equation}
for a constant $C_1$ under control.

As $u+\log\lvert s\rvert^2_h$ is smooth on $\tilde{\Omega}$, we ensure that $H(x)\simeq (k+A-1)\log\lvert s\rvert_h^2\to -\infty$ as $x\to  \mathrm{Supp}(E_j)$ by imposing $k\geq 1$. 
If $H$ reaches its maximum on $\partial \tilde{\Omega}$ then we are done 
since $\sup_{\partial \tilde{\Omega}}\lvert\nabla u \rvert$ is under control.
From now on we thus suppose   that $H$ reaches its maximum in $\tilde{\Omega}\setminus\{s=0\}$. A direct computation \cite[eq. (5.11), (5.20)]{PSS12} yields
\begin{multline}
    \label{eqn:Est1}
    \Delta'_\epsilon \log \lvert \nabla u\rvert^2_{\beta_\epsilon}\geq \frac{2\mathrm{Re}\langle \nabla v+\sum_{l=1}^m a_l\nabla \log \big(\lvert s_l \rvert^2_{h_l}+\epsilon^2\big), \nabla u\rangle_{\beta_{\epsilon}}}{\lvert \nabla u\rvert_{\beta_\epsilon}^2}
    -\Lambda \mathrm{tr}_{\eta_\epsilon}\beta_\epsilon\\ 
    +2\mathrm{Re}\langle \frac{\nabla \lvert \nabla u\rvert_{\beta_\epsilon}^2}{\lvert \nabla u\rvert_{\beta_\epsilon}^2},\frac{\nabla u}{\lvert \nabla u\rvert_{\beta_\epsilon}^2} \rangle_{\eta_\epsilon}-2\mathrm{Re}\langle \frac{\nabla \lvert \nabla u\rvert_{\beta_\epsilon}^2}{\lvert \nabla u\rvert_{\beta_\epsilon}^2},\frac{\nabla u}{\lvert \nabla u\rvert_{\beta_\epsilon}^2} \rangle_{\beta_\epsilon}
\end{multline}
where $\Lambda$ denotes a (uniform in $\e$) lower bound on the holomorphic bisectional curvature of $\beta_\epsilon$. 
At the point where $H$ reaches its maximum we obtain
$$
\frac{\nabla\lvert\nabla u\rvert_{\beta_\epsilon}^2}{\lvert \nabla u\rvert^2_{\beta_\epsilon}}=\nabla \log \lvert \nabla u\rvert^2_{\beta_\epsilon}=-\nabla \big(\log \lvert s \rvert_h^{2k}-G(u)\big)=-\frac{k\nabla \lvert s \rvert^2_h }{\lvert s\rvert^2_h}+G'(u)\nabla u,
$$
hence
\begin{eqnarray*}
\lefteqn{
    2\mathrm{Re}\langle \frac{\nabla \lvert \nabla u\rvert_{\beta_\epsilon}^2}{\lvert \nabla u\rvert_{\beta_\epsilon}^2},\frac{\nabla u}{\lvert \nabla u\rvert_{\beta_\epsilon}^2} \rangle_{\eta_\epsilon}-2\mathrm{Re}\langle  \frac{\nabla \lvert \nabla u\rvert_{\beta_\epsilon}^2}{\lvert \nabla u\rvert_{\beta_\epsilon}^2},\frac{\nabla u}{\lvert \nabla u\rvert_{\beta_\epsilon}^2} \rangle_{\beta_\epsilon} } \\
&=&    2k\mathrm{Re}\langle \frac{\nabla \lvert s \rvert^2_h}{\lvert s \rvert^2_h}, \frac{\nabla u}{\lvert \nabla u \rvert^2_{\beta_\epsilon}} \rangle_{\beta_\epsilon} -2k \mathrm{Re}\langle  \frac{\nabla \lvert s \rvert^2_h}{\lvert s \rvert^2_h}, \frac{\nabla u}{\lvert \nabla u \rvert^2_{\beta_\epsilon}} \rangle_{\eta_\epsilon}+2G'(u)\frac{\lvert \nabla u\rvert^2_{\eta_\epsilon}}{\lvert \nabla u\rvert^2_{\beta_\epsilon}}-2G'(u)\\
&\geq&    2k\mathrm{Re}\langle \frac{\nabla \lvert s \rvert^2_h}{\lvert s \rvert^2_h}, \frac{\nabla u}{\lvert \nabla u \rvert^2_{\beta_\epsilon}} \rangle_{\beta_\epsilon} -2k \mathrm{Re}\langle  \frac{\nabla \lvert s \rvert^2_h}{\lvert s \rvert^2_h}, \frac{\nabla u}{\lvert \nabla u \rvert^2_{\beta_\epsilon}} \rangle_{\eta_\epsilon}-2G'(u),
\end{eqnarray*}
using the monotonicity of $G(x)$ in  the last inequality. 
By (\ref{eqn:Maximum}) and asking $k\geq 2$, we can assume that
 $\lvert s \rvert_h^2 \lvert \nabla u\rvert_{\beta_\epsilon}\geq 1$ at $x_M$. 
 Thus, 
$$
    \lvert 2\mathrm{Re}\langle \frac{\nabla \lvert s \rvert^2_h}{\lvert s \rvert^2_h}, \frac{\nabla u}{\lvert \nabla u \rvert^2_{\beta_\epsilon}} \rangle_{\beta_\epsilon}\rvert\leq 2 \lvert \mathrm{Re}\langle \nabla \lvert s \rvert_h^2, \frac{\nabla u}{\lvert \nabla u\rvert_{\beta_\epsilon}} \rangle_{\beta_\epsilon} \rvert\leq D_1
    $$
    and
    \begin{eqnarray*}
        \lvert 2\mathrm{Re}\langle \frac{\nabla \lvert s \rvert^2_h}{\lvert s \rvert^2_h}, \frac{\nabla u}{\lvert \nabla u \rvert^2_{\beta_\epsilon}} \rangle_{\eta_\epsilon}\rvert
        &\leq&  2 \lvert \mathrm{Re}\langle \nabla \lvert s \rvert_h^2, \frac{\nabla u}{\lvert \nabla u\rvert_{\beta_\epsilon}} \rangle_{\eta_\epsilon} \rvert\leq  \lvert \nabla \lvert s\rvert^2_h \rvert_{\eta_\epsilon}^2+\frac{\lvert s\rvert_h^4\lvert \nabla u\rvert_{\eta_\epsilon}^2}{\lvert s\rvert_h^4\lvert \nabla u\rvert_{\beta_\epsilon}^2} \\
        &\leq&  \lvert \nabla \lvert s\rvert^2_h \rvert_{\beta_\epsilon}^2\mathrm{tr}_{\eta_\epsilon}{\beta_\epsilon}+\lvert s\rvert_h^4\lvert \nabla u\rvert^2_{\eta_\epsilon}.
    \end{eqnarray*}
We infer that at $x=x_M$,
\begin{eqnarray*} \label{eqn:Est2}
\lefteqn{
    2\mathrm{Re}\langle \frac{\nabla \lvert \nabla u\rvert_{\beta_\epsilon}^2}{\lvert \nabla u\rvert_{\beta_\epsilon}^2},\frac{\nabla u}{\lvert \nabla u\rvert_{\beta_\epsilon}^2} \rangle_{\eta_\epsilon}-2\mathrm{Re}\langle  \frac{\nabla \lvert \nabla u\rvert_{\beta_\epsilon}^2}{\lvert \nabla u\rvert_{\beta_\epsilon}^2},\frac{\nabla u}{\lvert \nabla u\rvert_{\beta_\epsilon}^2} \rangle_{\beta_\epsilon}  } \\
   &&\geq  -kD_1-k\lvert \nabla \lvert s\rvert^2_h \rvert_{\beta_\epsilon}^2\mathrm{tr}_{\eta_\epsilon}\beta_\epsilon-k\lvert s\rvert_h^4\lvert \nabla u\rvert^2_{\eta_\epsilon}-2G'(u) \\
   && \geq -kD_1-kD_2\mathrm{tr}_{\eta_\epsilon}\beta_\epsilon-k\lvert s \rvert_h^4 \lvert \nabla u \rvert^2_{\eta_\epsilon}-2G'(u),
\end{eqnarray*}
which is the first estimate of the right-hand side in (\ref{eqn:Est1}).

Next, as we want to prove (\ref{eqn:Induction2}), as a consequence of (\ref{eqn:Maximum}) and of $\lvert \nabla u\rvert^2_\beta\leq D_3 \lvert \nabla u \rvert^2_{\beta_\e}$ in the estimate that follows we can assume that 
$$
D_3C_1\lvert \nabla u\rvert^2_{\beta_\e} \geq \max\big\{\lvert \nabla (v/\gamma- \log\lvert s \rvert^2_h-f_\e) \rvert^2_\beta,  1\big\}
$$
at the point $x_M$. We deduce 
\begin{eqnarray*} \label{eqn:Est3}
    \lefteqn{ \Big\lvert \frac{2\mathrm{Re}\langle \nabla v+\sum_{l=1}^m \nabla a_l\log (\lvert s_l \rvert^2_{h_l}+\epsilon^2), \nabla u\rangle_{\beta_\epsilon}}{\lvert \nabla u \rvert_{\beta_\epsilon}^2} \Big\rvert}\\
    &\leq D_4+\frac{\lvert \nabla (v-\gamma\log\lvert s\rvert_h^2-\gamma f_\e) \rvert_{\beta_\epsilon}^2}{\lvert \nabla u \rvert_{\beta_\epsilon}^2}+\frac{D_5}{\lvert \nabla u \rvert_{\beta_\e}^2}\lvert s\rvert_h^{-2}+\frac{D_6}{\lvert \nabla u\rvert^2_{\beta_\epsilon}}\sum_{l=1}^m \lvert s_l \rvert_{h_l}^{-2}
    \leq C_2 + C_3\lvert s\rvert_h^{-2M}
\end{eqnarray*}
for $M:=\frac{1}{\min_lb_l}$ so that $Mb_l\geq 1$ for any $l$.
The previous inequalities yield
{\small
\begin{equation}
    \label{eqn:Est4}
    \Delta'_\epsilon\log \lvert \nabla u \rvert^2_{\beta_\epsilon}\\
    \geq -kD_1-C_2-C_3 \lvert s\rvert_h^{-2M}-\big(kD_2+\Lambda\big)\mathrm{tr}_{\eta_\epsilon}\beta_\epsilon-k\lvert s\rvert_h^4\lvert \nabla u\rvert^2_{\eta_\epsilon}-2G'(u).
\end{equation}
}
Moreover  
$$
    -\Delta'_\epsilon G(u)=-G'(u)\Delta'_\epsilon u-G''(u)\lvert \nabla u \rvert^2_{\eta_\epsilon}=G'(u) \mathrm{tr}_{\eta_\epsilon}\beta_\epsilon-nG'(u)-G''(u)\lvert \nabla u\rvert_{\eta_\epsilon}^2
    $$
and $\Delta'_\epsilon\log \lvert s \rvert^{2k}_h\geq -kD_7 \mathrm{tr}_{\eta_\epsilon}\beta_\epsilon$.
Together with (\ref{eqn:Est4}) we obtain
{\small
$$
 \Delta'_\epsilon H\geq \big(G'-kD_2-\Lambda-kD_7\big)\mathrm{tr}_{\eta_\epsilon}\beta_\epsilon-(n+2)G'-\big(G''+k\lvert s\rvert_h^4\big)\lvert \nabla u\rvert^2_{\eta_\epsilon}-kD_1-C_2-C_3\lvert s \rvert^{-2M}_h
$$
}
Taking $k=M(n+1)+1$, this can be rewritten
\begin{equation}
    \label{eqn:Est5}
     \Delta'_\epsilon H\geq (G'-D_8)\mathrm{tr}_{\eta_\epsilon}\beta_\epsilon-(n+2)G'-(G''+D_9\lvert s\rvert^4_h)\lvert \nabla u\rvert^2_{\eta_\epsilon}-C_4\lvert s \rvert_h^{-2M}.
\end{equation}

We now define
$
G(x):=(D_8+1)x-\frac{B}{x+C+1}
$
where $B>0$ is so large that
$$
\frac{2B}{(u+C+1)^3}-D_9\lvert s\rvert^4_h\geq \lvert s\rvert^2_h
$$
at $x_M$. Note that $B$ can be chosen such that it only depends on $C, D_9$ and on $\lVert u+\log\lvert s\rvert^2_h\rVert_{L^{\infty}(\tilde{\Omega})}$, i.e. it is under control. From (\ref{eqn:Est5}) we deduce at $x_M$
$$
0\geq \Delta'_\epsilon H\geq \mathrm{tr}_{\eta_\epsilon}\beta_\epsilon+\lvert s \rvert^2_h\lvert \nabla u \rvert^2_{\eta_\epsilon}-C_5\lvert s \rvert^{-2M}_h.
$$
This yields  
$
    \mathrm{tr}_{\eta_\epsilon}\beta_\epsilon\leq C_5\lvert s \rvert_h^{-2M}
     \text{ and } 
    \lvert\nabla u \rvert_{\eta_\epsilon}^2\leq C_5\lvert s \rvert_h^{-2M-2},
$
hence
{\small
\begin{equation*}
    \lvert\nabla u \rvert^2_{\beta_\epsilon}\leq \lvert\nabla u \rvert^2_{\eta_\epsilon}\mathrm{tr}_{\beta_\epsilon}\eta_\epsilon\leq
    \lvert \nabla u \rvert^2_{\eta_\e} \big(\mathrm{tr}_{\eta_\epsilon}\beta_\epsilon \big)^{n-1} \big(\frac{\eta_\e^n}{\beta_\e^n}\big)\leq
    C_6 \lvert s\rvert^{-2M}_h\lvert \nabla u \rvert^2_{\eta_\e} \big(\mathrm{tr}_{\eta_\epsilon}\beta_\epsilon \big)^{n-1}\leq
    C_7 \lvert s \rvert_h^{-2k}
\end{equation*}
}
where we also used \cite[Lemma 14.4]{GZbook}, the Monge-Ampère equation (\ref{eqn:MA_Gradient}) and the fact that $\prod_{l=1}^m\big(\lvert s_l\rvert^2_{h_l}+\e^2\big)^{a_l}\leq D_{10} \lvert s\rvert_h^{-2M}$ as $a_l>-1$. From (\ref{eqn:Maximum}) we deduce
$
\lvert\nabla u \rvert^2_{\beta_\epsilon}\lvert s\rvert^{2(k+D_8+1)}_h\leq C_8.
$
As $\{\beta_\epsilon\}_{\epsilon>0}$ is a non-degenerate continuous family of K\"ahler forms converging to $\beta$ as $\epsilon\to 0$, we get
$$
\lvert s\rvert^{2(k+D_8+1)}_h\lvert\nabla u \rvert^2_\beta\leq \max\big\{ \sup_{\tilde{\Omega}}[\lvert s \rvert_h^{2(k+D_8+1)} \lvert \nabla(v/\gamma-\log\lvert s\rvert^2-f_\e) \rvert_\beta^2], C_9\big\},
$$
i.e. (\ref{eqn:Induction2}), which concludes the proof by setting $B_2:=k+D_8+1, B_{3,j+1,\e}:=C_9$.
\end{proof}

\medskip

\noindent {\it Step 3.}
Fix $V$ a small neighborhood of $\partial \tilde{\Omega}$ (intersected with $\overline{\tilde{\Omega}}$). We claim that  
\begin{equation} \label{eq:c2estimate3}
\sup_{\partial \tilde{\Omega}} |\Delta_{\beta}   \f_{j,\e}| \leq C_V[1+\sup_V |\nabla \f_{j,\e}|^2],
\end{equation}  
for some uniform constant $C_V$   independent of $j,\e$.
This follows from a long series of estimates established in \cite[Lemma 18]{GKY13}
(which itself was adapting the technique developed by \cite{CKNS85})
when $\mu_p$ and $\Omega$ are smooth.
The statement of \cite[Lemma 18]{GKY13} mentions $\sup_{\tilde{\Omega}} |\nabla \f_j|^2$,
however the arguments only involve 
\begin{itemize}
\item local reasonings in a small fixed neighborhood of the boundary;
\item smoothness of $\mu_p$ in this neighborhood and pseudoconvexity of $\partial \tilde{\Omega}$.
\end{itemize}

\medskip

\noindent {\it Step 4.}
We now show that there exist   constants $m,B_{3,j,\e}>0$ such that 
\begin{equation} \label{eq:c2estimate4}
\sup_{\tilde{\Omega}} |s|^{2m} |\Delta_{\beta} \f_{j,\e}| \leq 
B_{3,j,\e} \left[1+\sup_{\partial \tilde{\Omega}} |\Delta_{\beta}   \f_{j,\e}| \right],
\end{equation}
where $B_{3,j,\e}$   only depends on an upper-bound on 
$||\f_{k,\e}||_{L^{\infty}(\tilde{\Omega})}$, for $k \leq j$.
This is a variant of \cite[Lemma 17]{GKY13}, for which we provide a detailed proof.

We set $\omega_j:=\e\beta+dd^c \f_{j,\e} $ and
observe that 
$$
\omega_j^n=e^{\p_{\e}-\f_{j-1,\e}-c'_{j-1}} \beta^n,
$$
where $\psi_{\e}$ is a difference of quasi-psh functions
in $\tilde{\Omega}$ such that  $e^{\p_{\e}} \leq c_1 |s|^{-2a}$ and 
$dd^c \p_{\e} \geq -c_1 |s|^{-2} \beta$ in $\tilde{\Omega}$,
for some uniform constants $a,c_1>0$. We consider
$$
H_j:=\log {\rm Tr}_{\beta}(\omega_j)+\f_{j-1,\e}  -A \f_{j,\e}  +A \rho',
$$
where $A>0$ is chosen below.  We use here the classical notations
$$
{\rm Tr}_{\eta}(\omega):=n \frac{\omega \wedge \eta^{n-1}}{\eta^n}
\; \; \text{ and } \; \;
\Delta_{\eta}(h):=n \frac{dd^c h \wedge \eta^{n-1}}{\eta^n}.
$$

Either $H_j$ reaches its maximum on $\partial \tilde{\Omega}$
and we are done, or it reaches its maximum at some point $x_j \in \tilde{\Omega} \setminus E$
since $\rho \rightarrow -\infty$ along $E$.
We are going to estimate $\Delta_{\omega_j} H_j$ from below and use the fact that
$0 \geq \Delta_{\omega_j} H_j(x_j)$.

It follows from \cite{Siu87} that 
$$
\Delta_{\omega_j} \log {\rm Tr}_{\beta}(\omega_j) \geq 
-\frac{{\rm Tr}_{\beta}( {\rm Ric}(\omega_j))}{{\rm Tr}_{\beta}(\omega_j)}-B {\rm Tr}_{\omega_j}(\beta),
$$
where $-B$ is a lower bound on the holomorphic bisectional curvature of $\beta$. Now
$$
-{\rm Ric}(\omega_j)=-{\rm Ric}(\beta)+dd^c (\p_{\e}-\f_{j-1,\e}) \geq -\omega_{j-1}-\frac{A_1}{|s|^2} \beta
$$
 in $\tilde{\Omega} \setminus E$. Moreover 
 ${\rm Tr}_{\beta}(\omega_{j-1}) \leq {\rm Tr}_{\beta}(\omega_j) {\rm Tr}_{\omega_j}(\omega_{j-1})$
 hence
 $$
 \Delta_{\omega_j} \log {\rm Tr}_{\beta}(\omega_j) \geq -{\rm Tr}_{\omega_j}(\omega_{j-1})
 -\frac{nA_1}{|s|^2{\rm Tr}_{\beta}(\omega_j)} -B {\rm Tr}_{\omega_j}(\beta).
 $$
 Using that $dd^c \rho' \geq \beta$, we obtain
 $$
 \Delta_{\omega_j} H_j \geq -An+(A-B) {\rm Tr}_{\omega_j}(\beta)-\frac{n A_1}{|s|^2{\rm Tr}_{\beta}(\omega_j)}.
 $$
 Using the classical inequality 
 $n [{\rm Tr}_{\omega_j}(\beta)]^{n-1}  \geq (\beta^n/\omega_j^n) {\rm Tr}_{\beta}(\omega_j)$,
 we infer
 \begin{equation} \label{eq:c2finale}
  \Delta_{\omega_j} H_j \geq -An+
 c(A-B) e^{\frac{-\p_{\e}}{n-1}}[{\rm Tr}_{\beta}(\omega_j)]^{\frac{1}{n-1}}
 -\frac{n A_1}{|s|^2{\rm Tr}_{\beta}(\omega_j)}.
 \end{equation}
Let us stress that the constant $c$ depends here on an upper 
bound on $||\f_{j-1,\e}||_{L^{\infty}(\tilde{\Omega})}$.

 We fix $A$ so large that $A>B$ and $\p_{\e}+A\rho' \leq c_1'-a \log|s|^2+A \rho'$ is bounded from above.
 At the point $x_j$ we obtain $0 \geq \Delta_{\omega_j} H_j$, therefore
 \begin{itemize}
 \item either $|s|^2 {\rm Tr}_{\beta}(\omega_j) \leq 1$ hence 
 $H_j(x_j) \leq (\f_{j-1,\e}  -A \f_{j,\e}  +A \rho')(x_j) \leq C$;
 \item or $|s|^2{\rm Tr}_{\beta}(\omega_j) \geq 1$ and \eqref{eq:c2finale} yields
 ${\rm Tr}_{\beta}(\omega_j) \leq C' e^{\p_{\e}(x_j)}$ hence
 $$
 H_j(x_j) \leq \p_{\e}(x_j)+A\rho'(x_j)+C'' \leq C'''.
 $$
 \end{itemize}
 Thus $H_j$ is uniformly bounded from above in both cases, and \eqref{eq:c2estimate4} follows
 (we use here   an upper 
bound on $||\f_{j-1,\e}||_{L^{\infty}(\tilde{\Omega})}$ and 
$||\f_{j,\e}||_{L^{\infty}(\tilde{\Omega})}$).

\medskip

\noindent {\it Step 5.}
We finally show by induction on $j$ that $\f_{j,\e}$ uniformly converges towards $\f_j$
as $\e$ decreases to $0$. There is nothing to prove for $j=0$ since  $\f_{0,\e}=\f_0$. 

For $j=1$, it follows from (a slight generalization of) 
\cite[Proposition 1.8]{GGZ23} that 
$||\f_{1,\e}||_{L^{\infty}(\tilde{\Omega})} \leq C_1$ is bounded uniformly in $\e>0$.
Proceeding by induction we similarly obtain that for all $j \in \N$,
$$
||\f_{j,\e}||_{L^{\infty}(\tilde{\Omega})} \leq C_j
$$
 is bounded uniformly in $\e>0$. By previous steps, the family 
 $(\f_{j,\e})_{\e}$ is relatively compact in ${\mathcal C}^{1,\alpha}$ for all $0<\alpha < 1$.
 Any cluster point $\p_j$, as $\e \rightarrow 0$, is a solution of 
 $$
(dd^c \p_{j+1})^n=\frac{e^{-\gamma \p_j} \mu_p}{c_j}
$$
with boundary values ${\p_{j+1}}_{|\partial \Omega}=\phi$,
hence $\p_j=\f_j$ by uniqueness.
Thus $\f_{j,\e}$ converges to $\f_j$ as $\e$ decreases to zero, and the convergence is moreover
uniform on $\tilde{\Omega}$ by \cite[Proposition 1.8]{GGZ23}.

We can thus let $\e$ tend to zero in previous inequalities.
Now $||\f_{j,\e}||_{L^{\infty}(\tilde{\Omega})} \rightarrow ||\f_j||_{L^{\infty}(\Omega)}$,
and the latter is uniformly bounded in $j$ by Proposition \ref{prop:C0estimate}.
For $\e=0$,  \eqref{eq:c2estimate1}, \eqref{eq:c2estimate2}, \eqref{eq:c2estimate3} and \eqref{eq:c2estimate4} thus provide uniform bounds in $j$,
and conclude the proof of \eqref{eq:c2estimate}.
The proof of Proposition \ref{prop:C2estimate} is thus complete.
\end{proof}

\subsection{Higher order estimates and convergence}

Once the uniform ${\mathcal C}^2$-estimate is established 
(Proposition \ref{prop:C2estimate}), one can then linearize the 
complex Monge-Amp\`ere equation and apply standard elliptic theory 
(Evans-Krylov method and Schauder bootstrapping) to derive higher order estimates:

\begin{prop}
Given $K$ a compact subset of ${\Omega} \setminus \{p\}$ and $\alpha>0,\ell \in \N$, there exists
$C(K,\ell,\alpha)>0$ such that for all $j \in \N$,
 $
 ||\f_j||_{{\mathcal C}^{\ell,\alpha}(K)} \leq C(K,\ell,\alpha).
 $
\end{prop}

  It follows that the sequence $(\f_j)$ is relatively compact in
  ${\mathcal C}^{\infty}({\Omega} \setminus \{p\})$
  We let ${\mathcal K}$ denote the set of cluster values of the sequence $(\f_j)$.
  Any function $\p \in {\mathcal K}$ is
\begin{itemize}
\item psh   in  $\Omega$ and smooth in ${\Omega} \setminus \{p\}$,
with $\p_{|\partial \Omega}=\phi$;
  \item uniformly bounded in  $\overline{\Omega}$ (Proposition \ref{prop:C0estimate});
  \item continuous on $\overline{\Omega}$, as the uniform limit of $(\f_{j_k})$
  \cite[Proposition 1.8]{GGZ23};
\end{itemize}  
  
  The set ${\mathcal K}$ is invariant under the action of  
  $T_{\gamma} : \f \in {\mathcal T}_{\phi}(\Omega) \mapsto \p \in {\mathcal T}_{\phi}(\Omega)$,
  which associates, to a given $\f \in  {\mathcal T}_{\phi}(\Omega)$, the unique
  solution $\p \in {\mathcal T}_{\phi}(\Omega)$ to the complex Monge-Amp\`ere equation
  $$
  (dd^c \p)^n =\frac{e^{-\gamma\f} \mu_p}{\int_{\Omega} e^{-\gamma \f} d\mu_p}.
  $$
  
  It follows from \cite[Proposition 12]{GKY13} that the functional $F_{\gamma}$ is constant on 
  ${\mathcal K}$ and that ${\mathcal K}$ is pointwise invariant under the action of $T_{\gamma}$.
  Thus a cluster value of $(\f_j)$ provides a desired solution to Theorem C.

  \section{Appendix by S.Boucksom}

  The purpose of this appendix is to provide an alternative approach to Proposition \ref{pro:Skoda}, emphasizing the role of $b$-divisors. We use~\cite{BdFF} as a main reference for what follows.

\subsection{Nef $b$-divisors over a point}

Consider for the moment any normal singularity $(X,p)$, and set $n:=\dim X$. 

In what follows, a \emph{birational model} means a projective birational morphism $\pi\colon X_\pi\to X$ with $X_\pi$ normal. A \emph{$b$-divisor over $p$} is defined as a collection $B=(B_\pi)_\pi$ of $\R$-divisors $B_\pi$ on $X_\pi$ for all birational models $\pi$, compatible under push-forward, and such that each $B_\pi$ has support in $\pi^{-1}(p)$. The $\R$-vector space of $b$-divisors over $p$ can thus be written as the projective limit
$$
\Divb(X,p):=\varprojlim_\pi\Div_p(X_\pi),
$$
where $\Div_p(X_\pi)$ denotes the (finite dimensional) $\R$-vector space of divisors on $X_\pi$ with support in $\pi^{-1}(p)$, and we endow $\Divb(X,p)$ with the projective limit topology. 

\medskip

A $b$-divisor $B\in\Divb(X,p)$ is said to be \emph{Cartier} if it is determined by some birational model $\pi$, in the sense that $B_{\pi'}$ is the pullback of $B_\pi$ for any higher birational model $\pi'$. There is a symmetric, multilinear \emph{intersection pairing} 
\begin{equation}\label{equ:intpairing}
(B_1,\dots,B_n)\mapsto(B_1\inter B_n)\in\R
\end{equation}
for Cartier $b$-divisors $B_i$, defined as the intersection number $(B_{1,\pi}\inter B_{n,\pi})$ computed on $X_\pi$ for any choice of common determination $\pi$ of the $B_i$ (the result being independent of the choice of $\pi$, by the projection formula). 

\medskip

A \emph{valuation centered at $p$} is a valuation $v \colon \cO_{X,p}\to \R_{\ge 0}$ such that
 $v(\mathfrak{m}_p)>0$ on the maximal ideal 
 $\mathfrak{m}_p \subset \cO_{X,p}$.
  It is further \emph{divisorial} if it can be written as $v=c\ord_E$ for a prime divisor $E\subset\pi^{-1}(p)$ on some birational model $X_\pi$ and $c\in\Q_{>0}$. Given a $b$-divisor $B$ over $p\in X$, we then set $v(B):=c \, \ord_E(B_\pi)$. The function $v\mapsto v(B)$ so defined on the space $\DivVal(X,p)$ of divisorial valuations centered at $p$ is homogeneous with respect to the scaling of $\Q_{>0}$, and this yields a topological vector space  isomorphism between $\Divb(X,p)$ and the space of homogeneous functions on $\DivVal(X,p)$, endowed with the topology of pointwise convergence. 

\medskip

Pick a $b$-divisor $B$ over $p$. If $B$ is Cartier, we say that $B$ is \emph{(relatively) nef} if $B_\pi$ is $\pi$-nef for some (hence any) determination $\pi$. In the general case, we say that $B$ is nef if it can be written as a limit of  nef Cartier $b$-divisors. By the Negativity Lemma, any nef $b$-divisor $B\in\Divb(X,p)$ is automatically antieffective, \ie $v(B)\le 0$ for all $v\in\DivVal(X,p)$. By~\cite[Lemma~2.10]{BdFF}, we further have: 

\begin{lem}\label{lem:nefmov} A $b$-divisor $B$ over $p$ is nef iff, for each birational model $\pi$, the numerical class of $B_\pi$ in $\Num(X_\pi/X)$ is \emph{nef in codimension $1$} (aka \emph{movable}). 
\end{lem}

\begin{example}\label{exam:Zfa} 
Consider an ideal $\fa\subset\cO_{X,p}$, and assume that $\fa$ is \emph{primary}, \ie containing some power of the maximal ideal. Then $\fa$ determines a nef Cartier $b$-divisor $Z(\fa)$, defined by $v(Z(\fa))=-v(\fa)$ for each $v\in\DivVal(X,p)$, and determined on the normalized blowup of $\fa$. For any tuple of primary ideals $\fa_1,\dots,\fa_n$, 
$$
e(\fa_1,\dots,\fa_n)=-(Z(\fa_1)\inter Z(\fa_n))
$$
further coincides with the mixed multiplicity of the $\fa_i$. 
\end{example}

\begin{example}\label{exam:Zv} 
For any valuation $v$ centered at $p$, the valuation ideals
$$
\fa_m(v):=\{f\in\cO_{X,p}\mid v(f)\ge m\}
$$
define a graded sequence of primary ideals $\fa_\bullet(v)$, and hence a nef $b$-divisor over $p$ 
$$
Z(v):=Z(\fa_\bullet(v))=\lim_m m^{-1} Z(\fa_m(v)),
$$
(see~\cite[Lemma~2.11]{BdFF}), which is not Cartier in general. 
\end{example} 

\begin{lem}\label{lem:neglem} If $B\in\Divb(X,p)$ is nef, then $B\le - v(B) Z(v)$ for all $v\in\DivVal(X,p)$, 
\end{lem}

\begin{proof} Write $v=c \, \ord_E$ for a prime divisor $E$ on $X_\pi$ and $c\in\Q_{>0}$. Then $Z(v)$ coincides with $\mathrm{Env}_\pi(-c^{-1} E)$ (see~\cite[Definition~2.3]{BdFF}), and the result thus follows from~\cite[Proposition~2.12]{BdFF}. 
\end{proof}

\subsection{Normalized volume and $b$-divisors}

From now on, we assume that the normal singularity $p\in X$ is further \emph{isolated}. 

By~\cite[Theorem~4.14]{BdFF}, the intersection pairing~\eqref{equ:intpairing} then extends to arbitrary tuples of nef $b$-divisors over $p$. This extended pairing takes values in $\R\cup\{-\infty\}$, and is symmetric, additive and non-decreasing in each variable, and continuous along decreasing nets. 

\begin{defi} 
For any nef $b$-divisor $B$ over $p$, we define the \emph{Hilbert--Samuel multiplicity} of $B$ as 
$$
e(B):=-B^n\in [0,+\infty]. 
$$
When $(X,p)$ is further klt, we define the \emph{log canonical threshold} of $B$ as 
$$
\lct(B):=\inf_{v\in\DivVal(X,p)}\frac{\ld_X(v)}{-v(B)}\in [0,+\infty), 
$$
where $\ld_X(v)\ge 0$ denotes the log discrepancy of $v$. 
\end{defi}

\begin{example}\label{exam:idealb} 
For any primary ideal $\fa\subset\cO_{X,p}$, the associated nef Cartier $b$-divisor $B:=Z(\fa)$ (see Example~\ref{exam:Zfa}) satisfies $e(B)=e(\fa)$, and $\lct(B)=\lct(\fa)$ when $(X,p)$ is klt. 
\end{example}

\begin{example} Pick any valuation $v$ centered at $p$, with associated nef $b$-divisor $Z(v)$ (see Example~\ref{exam:Zv}). Then it follows from~\cite[Remark~4.17]{BdFF} that the volume
$
\vol(v):=\lim_{m\to\infty}\frac{n!}{m^n}\dim\cO_{X,p}/\fa_m(v)
$
satisfies
\begin{equation}\label{equ:volint}
\vol(v)=e(Z(v)). 
\end{equation}
\end{example}

\begin{lem}\label{lem:approx} For each nef $b$-divisor $B$ over $p$, we have $e(B)=\sup_{C\ge B} e(C)$, where $C$ ranges over all nef Cartier $b$-divisors of the form $C=m^{-1}Z(\fa)$ for a primary ideal $\fa\subset\cO_{X,p}$ and $m\in\Z_{>0}$, and such that $C\ge B$. 
\end{lem}

\begin{proof} 
Since $B$ is the limit of the decreasing net $(\mathrm{Env}_\pi(B_\pi))$ (see~\cite[Remark~2.17]{BdFF}), it is enough to prove the result when $B=\mathrm{Env}_\pi(B_\pi)$, by continuity of the intersection pairing along decreasing nets. By~\cite[Theorem~4.11]{BdFF}, we can then write $B$ as the limit of a decreasing sequence $(C_i)$ of nef Cartier $b$-divisors of the desired form, and we are done since $e(C_i)\to e(B)$. 
\end{proof}

Consider now a psh function $\f$ on $X$. The collection of its Lelong numbers on all birational models defines a homogeneous function $v\mapsto v(\f)$ on $\DivVal(X,p)$, and hence an antieffective $b$-divisor $Z(\f)$ over $p$, such that $v(Z(\f))=-v(\f)$. 

\begin{prop}\label{prop:Zpsh} 
The $b$-divisor $Z(\f)$ is nef. Further:
\begin{itemize}
\item[(i)] if $\f$ is locally bounded outside $p$, then 
$$
e(Z(\f))\le e(\f):=(dd^c\f)^n(\{p\});
$$
\item[(ii)] if $(X,p)$ is klt, then $\lct(Z(\f))=\lct(\f)$. 
\end{itemize}
\end{prop}

\begin{proof} 
Consider the closed positive $(1,1)$-current $T:=dd^c\f$, and pick a log resolution $\pi\colon X_\pi\to X$ of $(X,p)$. The Siu decomposition of $\pi^\star T=dd^c\pi^\star\f$ shows that $\pi^\star T+[Z(\f)_\pi]$ is a positive current with zero generic Lelong numbers along each component of $\pi^{-1}(p)$. By Demailly regularization, it follows that the class of $Z(\f)_\pi$ in $\Num(X_\pi/X)$ is nef in codimension $1$, and hence that $Z(\f)$ is nef (see Lemma~\ref{lem:nefmov}). 

Assume next that $\f$ is locally bounded outside $p$, and pick a primary ideal $\fa\subset\cO_{X,p}$ and $m\in\Z_{>0}$ such that $C:=m^{-1} Z(\fa)\ge Z(\f)$. Choose a finite set of local generators $(f_i)$ of $\fa$, and consider the psh function $\p:=m^{-1}\log\sum_i|f_i|$. Then $Z(\f)\le C=Z(\p)$, and hence $\f\le\p+O(1)$ (to see this, pull back $\f$ and $\p$ to a log resolution of $\fa$, and use the Siu decomposition). By Demailly's comparison theorem, it follows that $e(C)=e(\p)\le e(\f)$, and taking the supremum over $C$ yields (i), by Lemma~\ref{lem:approx}. 

Finally, (ii) is a rather simple consequence of~\cite[Theorem~B.5]{BBJ21} applied to the pullback of $\f$ to a log resolution of $(X,p)$. 
\end{proof}

We can now state the following variant of Proposition~4.8. 

\begin{thm} Let $(X,p)$ be an isolated klt singularity. Then 
$$
\hvol(X,p)=\inf_B e(B)\lct(B)^n=\inf_\f e(\f)\lct(\f)^n,
$$
where $B$ runs over all nef $b$-divisors over $p$, and $\f$ runs over all psh functions on $X$ that are locally bounded outside $p$. 
\end{thm}
\begin{proof} By Theorem \ref{thm:NormVol} we have $\hvol(X,p)=\inf_\fa e(\fa)\lct(\fa)^n$, where $\fa\subset\cO_{X,p}$ runs over all primary divisors, and hence $\hvol(X,p)\ge\inf_B e(B)\lct(B)^n$, by Example~\ref{exam:idealb}. Conversely, pick a nef $b$-divisor $B$ over $p$. For any $v\in\DivVal(X,p)$, Lemma~\ref{lem:neglem} yields $B\le -v(B) Z(v)$. By monotonicity and homogeneity of the intersection pairing, this yields $B^n\le (-v(B))^n Z(v)^n$, \ie $e(B)\ge (-v(B))^n \vol(v)$, by~\eqref{equ:volint}. Thus 
$$
e(B)\left(\frac{\ld_X(v)}{-v(B)}\right)^n\ge\ld_X(v)^n\vol(v)\ge\hvol(X,p). 
$$
Taking the infimum over $v$ yields $e(B)\lct(B)^n\ge\hvol(X,p)$ for any nef $b$-divisor $B$ over $p$, and hence also $e(\f)\lct(\f)^n\ge\hvol(X,p)$ for any psh function $\f$ locally bounded outside $p$, by Proposition~\ref{prop:Zpsh}. 
\end{proof}


\begin{thebibliography}{99}


  \bibitem[ACC12]{ACC12} P.\AA hag, U.Cegrell, R.Czy\.z,
\emph{ On Dirichlet's principle and problem. }
Math. Scand. 110 (2012), no. 2, 235-250. 

  \bibitem[ACKPZ09]{ACKPZ09} P.\AA hag, U.Cegrell, S.Ko\l odziej, H.H.Pham, A.Zeriahi,
\emph{Partial pluricomplex energy and integrability exponents of psh functions. }
Adv. Math. 222 (2009), no. 6, 2036-2058.

 
 
  

  \bibitem [BGL22] {BGL20} B.Bakker, H.Guenancia, C.Lehn, 
{\it Algebraic approximation and decomposition theorem for K\"ahler Calabi-Yau varieties},
 Invent. Math. 228 (2022), no. 3, 1255-1308. 

  
  \bibitem[BT82]{BT82} E.~Bedford, B.~A. Taylor, 
\emph{A new capacity for plurisubharmonic functions}.
Acta Math. {149} (1982), no.~1-2, 1--40.  
  
 
    \bibitem[BB22]{BB22} R.Berman, B.Berndtsson
\emph{Moser-Trudinger type inequalities for complex Monge-Amp\`ere operators and Aubin's "hypoth\`ese fondamentale"}.
  Ann. Fac. Sci. Toulouse Math. (6) 31 (2022), no. 3, 595-645. 
 
 \bibitem[BB10]{BB10} R.J.Berman, S.Boucksom, 
 \emph{Growth of balls of holomorphic sections and energy at equilibrium. }
 Invent. Math. 181 (2010), no. 2, 337-394.
 

\bibitem[BBEGZ19]{BBEGZ} R.J.Berman, S.Boucksom, P.Eyssidieux, V.Guedj, A.Zeriahi,
  \emph{K\"{a}hler-{E}instein metrics and {K}\"{a}hler-{R}icci flow on log  {F}ano varieties}, 
  J. Reine Ang. Math. {751} (2019), 27--89.
   

\bibitem[BBGZ13]{BBGZ13} R.~J. Berman, S.~Boucksom, V.~Guedj, and A.~Zeriahi, 
\emph{A variational approach to complex {M}onge-{A}mp\`ere equations}, 
  Pub.Math. I.H.E.S. {117} (2013), 179--245.  

\bibitem[BBJ21]{BBJ21}
R.~J. Berman, S.Boucksom, M.Jonsson, 
\emph{{A variational approach to the  Yau-Tian-Donaldson conjecture}},
   J. Amer. Math. Soc. (2021) no. 3, 605-652.
   
   \bibitem[BdFF12]{BdFF} 
S.~Boucksom, T.~de Fernex, C.~Favre. 
\newblock\emph{The volume of an isolated singularity}. 
\newblock Duke Math. J. \textbf{161} (2012), 1455--1520. 

\bibitem[BFJ08]{hiro} 
S.~Boucksom, C.~Favre and M.~Jonsson.
\newblock \emph{Valuations and plurisubharmonic singularities}.
\newblock Publ. Res. Inst. Math. Sci.  \textbf{44}  (2008), 449--494. 

\bibitem[Blo06]{Blo06} Z.B\l ocki,
\emph{ The domain of definition of the complex Monge-Amp\`ere operator. }
Amer. J. Math. 128 (2006), no. 2, 519-530. 

\bibitem[Blum18]{Blum18} H.Blum
\emph{Existence of valuations with smallest normalized volume.}
 Compos. Math. 154 (2018), no. 4, 820-849.
   
\bibitem[CKNS85]{CKNS85} L.Caffarelli, J.J.Kohn, L.Nirenberg, J.Spruck,
\emph{The Dirichlet problem for nonlinear second-order elliptic equations. II.
 Complex Monge-Ampère, and uniformly elliptic, equations.}
 Comm. Pure Appl. Math. 38 (1985), no. 2, 209-252. 

 \bibitem[Ceg98]{Ceg98} U.Cegrell,
\emph{Pluricomplex energy. }
Acta Math. 180 (1998), no. 2, 187-217. 

 \bibitem[Ceg04]{Ceg04} U.Cegrell,
\emph{The general definition of the complex Monge-Amp\`ere operator.}
 Ann. Inst. Fourier (Grenoble) 54 (2004), no. 1, 159-179. 
 

 

  

\bibitem[CDS15]{CDS15} X.X.~Chen, S.~Donaldson, S.~Sun,
\emph{K\"ahler-Einstein metrics on Fano manifolds, I, II \&  III }, 
  J. Amer. Math. Soc. 28 (2015), 183-197, 199-234 \& 235-278.
  
  
    \bibitem[DFS23]{DFS21} V.Datar, X.Fu, J.Song,
\emph{K\"ahler-Einstein metric near an isolated log canonical singularity}.
 J. Reine Angew. Math. 797 (2023), 79-116.
 

  \bibitem[Dem85]{Dem85}  J.-P. Demailly,
  {\it Mesures de Monge-Amp\`ere et caract\'erisation g\'eom\'etrique des vari\'et\'es alg\'ebriques affines.} 
  M\'em. Soc. Math. France (N.S.) No. 19 (1985), 124 pp. 
  
  
    
  
  \bibitem[Dem09]{Dem09} J.-P. Demailly, {\it Estimates on Monge-Ampère operators derived from a local algebra inequality}.
  Complex analysis and digital geometry, 131-143, 
Uppsala Universitet, 
2009. 
  

    
  
   \bibitem[Dem12]{Dem12}  J.-P. Demailly,  
   \emph{Analytic methods in algebraic geometry},
Surveys of Modern Mathematics, 1. International Press; 
Higher Education Press, Beijing, 2012. viii+231 pp.

 


 \bibitem[DK01]{DK01}  J.-P. Demailly, J.Kollàr,
  {\it Semi-continuity of complex singularity exponents and K\"ahler-Einstein metrics on Fano orbifolds. }
Ann. Sci. E.N.S.. (4) 34 (2001), no. 4, 525-556. 
  
  
  \bibitem[DP14]{DP14}  J.-P. Demailly, H.H.Pham,
  {\it A sharp lower bound for the log canonical threshold.}
   Acta Math. 212 (2014), no. 1, 1-9.

  
 

   
  
  
   
 

   
   
   
      \bibitem[Dr18]{Druel18}  S.Druel
\emph{A decomposition theorem for singular spaces with trivial canonical class of dimension at most five.}
    Invent. Math. 211 (2018), no. 1, 245-296.
   
   
   \bibitem[ELS03]{ELS03} L.Ein, R.Lazarsfeld, K.E.Smith,
\emph{Uniform approximation of Abhyankar valuation ideals in smooth function fields. }
Amer. J. Math. 125 (2003), no. 2, 409-440.
  

\bibitem[EGZ09]{EGZ09} P.~Eyssidieux, V.~Guedj, A.~Zeriahi, 
\emph{Singular {K}\"{a}hler-{E}instein  metrics},
 J. Amer. Math. Soc. {22} (2009), no.~3, 607--639.

 
 \bibitem[FN80]{FN80} J.-E.Forn\ae ss, R.Narasimhan,
\emph{The Levi problem on complex spaces with singularities.}
 Math. Ann. 248 (1980), no. 1, 47-72. 
 
  \bibitem[Fu23]{Fu21} X.Fu,
\emph{Dirichlet problem of complex Monge-Amp\`ere equation near an isolated KLT singularity.}
Commun. Contemp. Math. 25 (2023), no. 5, Paper No. 2250019, 18 pp.



 
 \bibitem[Fuj19]{Fuj19} K. Fujita, \textit{Uniform K-stability and plt blowups of log Fano pairs.}, Kyoto Journal of Mathematics 59, 2 (2019), 399–418
 
  \bibitem[GL10]{GL10}  B.Guan, Li
\emph{Complex Monge-Amp\`ere equations and totally real submanifolds.}
 Adv. Math. 225 (2010), no. 3, 1185-1223. 
 
 
  \bibitem[GZh15]{GZh15}  Q.Guan, X.Zhou,
\emph{A proof of Demailly's strong openness conjecture.}
 Ann. of Math. (2) 182 (2015), no. 2, 605-616. 
 
 
  \bibitem[GGZ23]{GGZ23} V.~Guedj, H.Guenancia, A.Zeriahi,
\emph{Continuity of singular K\"ahler-Einstein potentials}.
Int. Math. Res. Not. (2023), no. 2, 1355-1377.


   
 \bibitem[GKY13]{GKY13} V.~Guedj, B.Kolev, N.Yeganefar
\emph{K\"ahler-Einstein fillings.}
 J. Lond. Math. Soc. (2) 88 (2013), no. 3, 737-760.
 


 

\bibitem[GZ]{GZbook} V.~Guedj, A.~Zeriahi, 
\emph{Degenerate complex {M}onge-{A}mp\`ere  equations}, 
EMS Tracts in Mathematics, vol.~26, European Mathematical Society  (EMS), Z\"{u}rich, 2017.


\bibitem[Hart77]{Hart77} R.Hartshorne,
\emph{Algebraic geometry.}
 Graduate Texts in Mathematics, No. 52. Springer-Verlag, New York-Heidelberg, 1977. xvi+496 pp.

 
 \bibitem [HS17] {HS17} H.-J.Hein, S.Sun,
 {\it Calabi-Yau manifolds with isolated conical singularities. }
 Publ. Math. I.H.E.S. 126 (2017), 73-130.
 
  \bibitem [HP19] {HP19} A.H\"oring, T.Peternell,
{\it Algebraic integrability of foliations with numerically trivial canonical bundle. }
Invent. Math. 216 (2019), 395-419. 
 
  
  \bibitem[Kol98]{Kol98} S.~Ko{\l}odziej, 
  \emph{The complex {M}onge-{A}mp\`ere equation}, 
  Acta Math.  {180} (1998), 69--117.
  
    \bibitem[Li18]{Li18} C.Li,
{\it  Minimizing normalized volumes of valuations.}
 Math. Z. 289 (2018), no. 1-2, 491-513.
  
    \bibitem[Li22]{Li22} C.Li,
{\it G-uniform stability and K\"ahler-Einstein metrics on Fano varieties. }
Invent. Math. 227 (2022), no. 2, 661-744. 

 

 \bibitem[LTW21]{LTW21} C. Li, G. Tian, F. Wang,
{\em On Yau-Tian-Donaldson conjecture for singular Fano varieties},
Comm. Pure Appl. Math. 74 (2021), no. 8, 1748-1800.

  \bibitem[LWX21]{LWX21} C. Li, X.Wang, C.Xu,
 {\em Algebraicity of the Metric Tangent Cones and Equivariant K-stability.}
 J. Amer. Math. Soc.  34 (2021) no. 4, 1175-1214

  
  \bibitem[Liu18]{Liu18} Y.Liu ,
   \emph{The volume of singular K\"ahler-Einstein Fano varieties.}
    Compos. Math. 154 (2018), no. 6, 1131-1158. 
    
      \bibitem[LWZ22]{LWZ22} Y.Liu , X.Wang, Z.Zhang,
   \emph{  Finite generation for valuations computing stability thresholds and applications to K-stability.}
   Annals of Math. 196 (2022), Issue 2, 507-566. 
   
   \bibitem[Pham14]{Pham14} H.Hiep Pham,
   \emph{The weighted log canonical threshold. }
   C. R. Math. Acad. Sci. Paris 352 (2014), no. 4, 283-288. 
   
    \bibitem[Pham18]{Pham18} H.Hiep Pham,
   \emph{  Log canonical thresholds and Monge-Ampère masses.}
    Math. Ann. 370 (2018), no. 1-2, 555-566.
    
    
  \bibitem[PSS12]{PSS12} D.H.Phong, J.Song, K.Sturm,
   \emph{Complex Monge-Amp\`ere equations.}
    Surveys in Differential Geometry 17, 1 (2012), 327-410.
   

\bibitem[Siu87]{Siu87} Y.-T. Siu,
\emph{Lectures on Hermitian-Einstein metrics for stable bundles and K\"ahler-Einstein metrics.}
 DMV Seminar, 8. Birkh\"auser Verlag, Basel, 1987.
   
       \bibitem[Tian87]{Tian87}  G.Tian,
   \emph{On K\"ahler-Einstein metrics on certain K\"ahler manifolds with $C_1(M)>0$.}
    Invent. Math. 89 (1987), no. 2, 225-246.
   
   
  
      \bibitem[Zer09]{Zer09}  A.Zeriahi
   \emph{A stronger version of Demailly's estimate on Monge-Amp\`ere operators.}
    Complex analysis and digital geometry, 144-146, 
    Uppsala Universitet, Uppsala, 2009. 
 
\bibitem [Yau78]{Yau78}  S.~T.~Yau,
\emph{On the Ricci curvature of a compact K{\"a}hler manifold and the complex Monge-Amp{\`e}re equation. I.}
 Comm. Pure Appl. Math. {31} (1978), no. 3, 339-411.  
  

\end{thebibliography}
  \end{document}